\documentclass[11pt, a4paper]{article}
\usepackage[french, english]{babel}
\usepackage[utf8,latin1]{inputenc}
\usepackage{amsmath}
\usepackage{amssymb}
\usepackage{amsthm}
\usepackage{latexsym,tikz}
\usepackage{hyperref}
\usepackage{enumerate}
\usepackage[all]{xy}
\usepackage{mathrsfs}
\usepackage{dsfont}
\usepackage{tikz}
\usetikzlibrary{decorations.pathmorphing}
\usepackage{appendix}
\usetikzlibrary{decorations.pathreplacing}
\usepackage{pdfpages}

\usepackage[nottoc]{tocbibind}

\newtheorem{thm}{Theorem}[section]
\newtheorem{cor}[thm]{Corollary}
\newtheorem{lem}[thm]{Lemma}
\newtheorem{prop}[thm]{Proposition}
\newtheorem{defn}[thm]{Definition}

\newtheorem{rem}[thm]{Remark}

\def\fA{\mathfrak{A}}
\def\rA{\mathrm{A}}

\def\bC{\mathbb{C}}
\def\nC{\mathnormal{C}}

\def\cD{\mathcal{D}}

\def\det{\mathrm{det}}

\def\rF{\mathrm{F}}
\def\rG{\mathrm{G}}

\def\bG{\textbf{G}}

\def\GL{\mathrm{GL}}

\def\Hom{\mathrm{Hom}}
\def\rH{\mathrm{H}}

\def\Ind{\mathrm{Ind}}
\def\ind{\mathrm{ind}}

\def\cK{\mathcal{K}}
\def\rK{\mathrm{K}}
\def\Ker{\mathrm{Ker}}

\def\rL{\mathrm{L}}
\def\bL{\textbf{L}}

\def\rM{\mathrm{M}}

\def\rN{\mathrm{N}}
\def\fo{\mathfrak{o}}

\def\rO{\mathrm{O}}
\def\rP{\mathrm{P}}

\def\Q{\mathbb{Q}}

\def\rQ{\mathrm{Q}}
\def\obQ{\overline{\mathbb{Q}}}

\def\res{\mathrm{res}}

\def\SL{\mathrm{SL}}

\def\rT{\mathrm{T}}

\def\bT{\mathbf{T}}

\def\rU{\mathrm{U}}

\def\rV{\mathrm{V}}

\def\rX{\mathrm{X}}
\def\rY{\mathrm{Y}}
\def\rZ{\mathrm{Z}}

\begin{document}





\hypersetup{							
pdfauthor = {Peiyi Cui},			
pdftitle = {representation mod l of SL_n},			
pdfkeywords = {Tag1, Tag2, Tag3, ...},	
}					
\title{Modulo $\ell$-representations of $p$-adic groups $\mathrm{SL_n}(F)$}
\author{Peiyi Cui \footnote{peiyi.cuimath@gmail.com, Institut de recherche math\'ematique de Rennes -- I.R.M.A.R.-U.M.R. 6625 du C.N.R.S., Universit\'e de Rennes 1 Beaulieu, 35042 Rennes CEDEX, France}}

\date{\vspace{-2ex}}
\maketitle

        \begin{abstract}
        Let $k$ be an algebraically closed field of characteristic $l\neq p$. We construct maximal simple cuspidal $k$-types of Levi subgroups $\rM'$ of $\mathrm{SL}_n(F)$, when $F$ is a non-archimedean locally compact field of residual characteristic $p$. And we show that the supercuspidal support of irreducible smooth $k$-representations of Levi subgroups $\rM'$ of $\mathrm{SL}_n(F)$ is unique up to $\rM'$-conjugation, when $F$ is either a finite field of characteristic $p$ or a non-archimedean locally compact field of residual characteristic $p$.
\end{abstract}

\tableofcontents

\section{Introduction}

\subsection{Introduction}
Let $F$ be a finite field of characteristic $p$, or a non-archimedean locally compact field whose residue field is of characteristic $p$, and $\textbf{G}$ a reductive connected algebraic group defined over $F$. We denote by $\rG$ the group $\textbf{G}(F)$ of the $F$-points of $\textbf{G}$, and endow it with the locally pro-finite topology through $F$. Let $k$ be an algebraically closed field of characteristic $\ell$ ($\neq p$), and $W(k)$ its ring of Witt vectors. We use $\cK=\mathrm{Frac}(W(k))$ to denote the fraction field of $W(k)$. We denote by $\mathrm{Rep}_{k}(\rG)$ the category of smooth $k$-representations of $\rG$, where a $k$-representation $(\pi,V)$ (here $V$ is the $k$-space of representation $\pi$) of $\rG$ is smooth if any element $v\in V$ is stabilised by an open subgroup of $\rG$. In the thesis, when we say a $k$-representation, we always assume it is smooth. 

When $\ell=0$ and $F$ is a non-archimedean locally compact field, the existence of Bernstein decomposition of the category $\mathrm{Rep}_k(\rG)$ has been proved by Bernstein: The category $\mathrm{Rep}_k(\rG)$ is equivalent to the direct product of some full-subcategories, which are indecomposable and called blocks. This means that each $k$-representation is isomorphic to a direct sum of sub representations belonging to different blocks, and each morphism of $k$-representations is isomorphic to a product of morphisms belonging to different blocks. We say that a full-subcategory of $\mathrm{Rep}_k(\rG)$ is indecomposable (or a block) if it is not equivalent to a product of any two non-trivial full-subcategories.


This decomposition has a counterpart in the ``Galois side" through local Langlands correspondence. Let $r$ be a prime number such that $r\ne p$. When $\ell>0$, we will take $r=\ell$.  Let ${}^{L}\textbf{G}$ denote the $L$-group of $\textbf{G}$ and  $W_F$ the Weil group of $F$. 
In the case when $\textbf{G}$ equals $\mathrm{GL}_n$, the local Langlands correspondence (LLC) was proved when $F$ has characteristic $p$ by Laumon, Rapoport and Sthuler \cite{LRS}, and, when $F$ has characteristic $0$, independently by Harris and Taylor \cite{Ha}, by Henniart \cite{Hen}, and by Scholze \cite{Sch}. It provides a canonical bijection between the set of isomorphism classes of $r$-adic irreducible representations of $\GL_n(F)$ and the set of isomorphism classes of $r$-adic $n$-dimensional $W_F$ -semisimple Deligne representations, generalizing the Artin reciprocity map of local class field theory.  A nice property of LLC is that the Rankin-Selberg local factors of a pair of irreducible $\overline{\Q}_r$-representations of $\GL_n(F)$ and $\GL_m(F)$, and the Artin-Deligne local factors of the corresponding tensor product of representations of $W_F$ are equal, and moreover this condition characterizes LLC completely.
Under the local Langlands correspondence, two irreducible $k$-representations $\pi$ and $\pi'$ belong to the same block if and only if their Langlands parameters are isomorphic when restricting to the inertial subgroup $I_F$ of $W_F$.  For $\textbf{G}$ an arbitrary connected reductive group defined over $F$, an analog of the Bernstein decomposition for (enhanced) Langlands parameters is constructed in \cite{AMS}.

When $\ell$ is positive and $F$ is a non-archimedean locally compact field. Helm gives a proof of the Bernstein decomposition of $\mathrm{Rep}_{\overline{W(\mathbb{F}}_{\ell})}(\rG)$ in \cite{Helm}, where $W(\mathbb{F}_{\ell})$ denotes the ring of Witt vectors of $\mathbb{F}_{\ell}$, and this deduces the Bernstein decomposition of $\mathrm{Rep}_{\mathbb{F}_{\ell}}(\rG)$. S\'echerre and Stevens gave a proof of the Bernstein decomposition of the category of smooth $k$-representations of $\mathrm{GL}_n(F)$ and its inner forms in \cite{SeSt}. The Bernstein decomposition of $\mathrm{Rep}_k(\rG)$ is unknown for general reductive groups $\textbf{G}$ defined over $F$. In the case where $\textbf{G}$ equal $\mathrm{GL}_n$, Vign\'eras constructed in \cite{V4} a bijection between the set of isomorphism classes of $\ell$-modular irreducible representations of $\GL_n(F)$ and the set of isomorphism classes of $\ell$-modular $n$-dimensional $W_F$- semisimple Deligne representations with nilpotent Deligne operator. Combining with the Bernstein decomposition, it implies that two irreducible $\overline{\mathbb{F}}_{\ell}$-representations $\pi,\pi'$ of $\rG$ belong to the same block if and only if their Langlands parameters are isomorphic when restricting to $I_F^{\ell}$, which is the kernel of the canonical map $I_F\rightarrow \mathbb{Z}_{\ell}$, as observed by Dat in \cite{DaII} \S $1.2.1$. 

The theory of Rankin-Selberg local factors of Jacquet, Shalika and Piatetski-Shapiro has a natural extension at least to  generic $k$-representations of $\GL_n(F)$. However, via the $\ell$-modular local Langlands correspondence these factors do not agree with the factors of Artin-Deligne.
In \cite{KuMa},  Kurinczuk and Matringe classified the indecomposable $\ell$-modular $W_F$-semisimple Deligne representations, extended the definitions of Artin-Deligne factors to this setting, and define an $\ell$-modular local Langlands correspondence where in the generic case, the Rankin-Selberg factors of representations on one side equal the Artin-Deligne factors of the corresponding representations on the other.


In this thesis, we study the category $\mathrm{Rep}_k(\mathrm{SL}_n(F))$. In particular, we study cuspidal and supercuspidal $k$-representations of Levi subgroups of $\mathrm{SL}_n(F)$. 

\begin{defn}
\label{cusp 1}
Let $\rG$ be the group of $F$-points of a connected reductive group defined over $F$, where $F$ is non-archimedean with residual characteristic $p$ or finite with characteristic $p$. Let $\pi$ be an irreducible $k$-representation of $\rG$. 
\begin{itemize}
\item We say that $\pi$ is cuspidal, if for any proper Levi subgroup $\rM$ of $\rG$ and any irreducible $k$-representation $\rho$ of $\rM$, $\pi$ does not appear as a subrepresentation or a quotient representation of $i_{\rM}^{\rG}\rho$;
\item We say that $\pi$ is supercuspidal, if for any proper Levi subgroup $\rM$ of $\rG$ and any irreducible $k$-representation $\rho$ of $\rM$, $\pi$ does not appear as a subquotient representation of $i_{\rM}^{\rG}\rho$.
\end{itemize}
\end{defn}

\begin{defn}
\label{cusp 2}
Let $\pi$ be an irreducible $k$-representation of $\rG$, where $\rG$ is assumed as in definition \ref{cusp 1}.
\begin{itemize}
\item Let $(\rM,\rho)$ be a cuspidal pair of $\rG$, which means $\rM$ is a Levi subgroup of $\rG$, and $\rho$ is an irreducible cuspidal $k$-representation of $\rM$. We say that $(\rM,\rho)$ belongs to the cuspidal support of $\pi$, if $\pi$ is a subrepresentation or a quotient representation of $i_{\rM}^{\rG}\rho$;
\item Let $(\rM,\rho)$ be a supercuspidal pair of $\rG$, which means $\rM$ is a Levi subgroup of $\rG$, and $\rho$ is an irreducible supercuspidal $k$-representation of $\rM$. We say that $(\rM,\rho)$ belongs to the supercuspidal support of $\pi$, if $\pi$ is a subquotient of $i_{\rM}^{\rG}\rho$.
\end{itemize}
\end{defn}

The proofs of Helm, S\'echerre and Stevens in \cite{Helm} and \cite{SeSt} of the Bernstein decompositions are based on the fact  that the supercuspidal support (Definition \ref{cusp 2}) of any irreducible $k$-representation of $\mathrm{GL}_n(F)$ is unique, which has been proved by Vign\'eras in \cite{V2}. As the main results of this thesis, we prove the uniqueness of supercuspidal support for $\mathrm{SL}_n(F)$ in both cases that $F$ is finite (Theorem \ref{theorem 1}) and $F$ is non-archimedean (Theorem \ref{Uthm 3}):


\begin{thm}
Let $\rM'$ be a Levi subgroup of $\mathrm{SL}_n(F)$, and $\rho$ an irreducible $k$-representation of $\rM'$. The supercuspidal support of $\rho$ is a $\rM'$-conjugacy class of a pair $(\rL',\tau')$, where $\rL'$ is a Levi subgroup of $\rM'$, and $\tau'$ is an irreducible supercuspidal $k$-representation of $\rL'$.
\end{thm}

However, the uniqueness of supercuspidal support of irreducible $k$-representations is not always true for general reductive groups when $\ell$ is positive. A counter-example has been found in \cite{Da} by Dat and Dudas for $\mathrm{Sp}_8(F)$.

From now on, we use $\rG$ to denote $\mathrm{GL}_n(F)$, and $\rG'$ to denote $\mathrm{SL}_n(F)$ unless otherwise specified. This manuscript has two parts: in section \ref{chapter 3}, we study the $k$-representations of finite groups; next we consider the case that $F$ is non-archimedean locally compact from section \ref{chapter 4}. There is a fact that for any irreducible $k$-representation $\pi'$ of $\rM'$, a Levi subgroup of $\rG'$, its cuspidal support (see \ref{cusp 2} for the definition) is unique. Hence we could reduce our problem to the uniqueness of supercuspidal support for irreducible cuspidal $k$-representations of $\rM'$, where $\rM'$ denote any Levi subgroup of $\rG'$. In both parts, for any irreducible cuspidal $k$-representation $\pi'$ of $\rM'$, there exists an irreducible cuspidal $k$-representation $\pi$ of $\rM$ such that $\pi'$ is a component of the semisimple $k$-representation $\res_{\rM'}^{\rM}\pi$ which has finite length  (see \cite{TA} when $\ell=0$, and Proposition \ref{prop 6} when $\ell$ is positive). Our strategy is to study $\pi'$ by considering $\pi$, in other words, to reduce the problem of $\pi'$ to the one of $\pi$.

In the first part, we describe the supercuspidal support of an irreducible cuspidal $k$-representation $\pi'$ of $\rM'$ in terms of its projective cover (see the paragraph after Theorem \ref{theorem 1}), which has been considered by Hiss \cite{Hiss}. Using Deligne-Lusztig theory, we construct the projective cover $\mathrm{P}_{\pi'}$ of $\pi'$, which is one of the indecomposable components of the restriction of the projective cover $\mathrm{P}_{\pi}$ of $\pi$ to $\rM'$. The construction is based on the Gelfand-Graev lattice. We deduce the uniqueness of supercuspidal support of $\pi'$ by considering the parabolic restrictions of $\mathrm{P}_{\pi'}$ to any Levi subgroup of $\rM'$. 

The projective covers $\mathrm{P}_{\pi'}$ constructed in this part are interesting in their own right. Let $\bar{\cK}$ denote an algebraic closure of $\cK$.
In the article \cite{Helm}, Helm gave the relation between Bernstein decompositions of $\mathrm{Rep}_{\bar{\cK}}(\mathrm{GL}_n(F))$ and of $\mathrm{Rep}_{k}(\mathrm{GL}_n(F))$. One of the key objects of his article is a family of projective objects associated to irreducible cuspidal $k$-representations. These projective objects are constructed by projective covers of irreducible cuspidal $k$-representations of finite groups of $\mathrm{GL}_m$ type, where $m$ divides $n$. 

In the second part, $\rG$ and $\rG'$ are defined over a non-archimedean local field. We prove the uniqueness of the supercuspidal support (Theorem \ref{Uthm 3}) in two steps. In section \ref{section 06}, we construct maximal simple cuspidal $k$-types of $\rM'$ (Theorem \ref{Ltheorem 5}), where $\rM'$ denote any Levi subgroup of $\rG'$. This gives a first description of the supercuspidal support for any irreducible cuspidal $k$-representation $\pi'$ of $\rM'$. In section \ref{chapter 5}, we describe precisely the supercuspidal support of $\pi'$ by considering the derivatives of the elements in the supercuspidal support, and deduce that it is unique.

\begin{defn}
An maximal simple cuspidal $k$-type of $\rM'$ for an irreducible cuspidal $k$-representation $\pi'$ of $\rM'$ is a pair $(K', \tau_{\rM'})$ consisting of an open and compact modulo center subgroup $\rK'$ of $\rM'$, and an irreducible $k$-representation $\tau_{\rM'}$ of $\rK'$, such that:
\begin{equation}
\label{equa 0}
\ind_{\rK'}^{\rM'}\tau_{\rM'}\cong\pi'.
\end{equation}
\end{defn}

Inspired by \cite{BuKuI}, \cite{BuKuII}, \cite{GoRo} and \cite{TA}, we construct the maximal simple cuspidal $k$-types of $\rM'$  from those of $\rM$, where $\rM$ is a Levi subgroup of $\rG$ such that $\rM\cap\rG'=\rM'$. More precisely, let $\pi$ be an irreducible cuspidal $k$-representation of $\rM$, such that $\pi'$ is an component of $\res_{\rM'}^{\rM}\pi$. Let $(J_{\rM},\lambda_{\rM})$ be a maximal simple cuspidal $k$-type of $\rM$ of $\pi$ (we inherit the notations from those of \cite{BuKu}), which means the compact induction $\ind_{\rK}^{\rM}\Lambda_{\rM}$ is isomorphic to $\pi$, where $\rK$ is an open subgroup of $\rM$ compact modulo center, which contains $J_{\rM}$ as the unique maximal compact subgroup, and $\Lambda_{\rM}$ is an extension of $\lambda_{\rM}$ to $\rK$. In the equation (\ref{equa 0}), the group $\rK'$ is also compact modulo centre. Furthermore, the group $\rK'$ contains $\tilde{J}_{\rM}\cap\rM'$ as the unique maximal open compact subgroup, and $\rK'$ is a normal subgroup of $(\rK\cap\rM')(\tilde{J}_{\rM}\cap\rM')$ with finite index, where $\tilde{J}_{\rM}$ is an open compact subgroup of $\rM$ containing $J_{\rM}$. The irreducible $k$-representation $\tau_{\rM'}$ of $\rK'$ contains some irreducible component of the semisimple representation $\res_{\tilde{J}_{\rM}\cap\rM'}^{\tilde{J}_{\rM}}\ind_{J_{\rM}}^{\tilde{J}_{\rM}}\lambda_{\rM}$. When $\rM'=\rG'$, the group $\rK'$ equals $\tilde{J}_{\rM}\cap\rM'$, and this simple case is considered in section \ref{section 08}, based on which the case for proper Levi is dealt.

In the construction, the technical difficulty is to prove that the compact induction $\ind_{\rK'}^{\rM'}\tau_{\rM'}\cong\pi'$ is irreducible. When the characteristic of $k$ is $0$, it is sufficient to prove that the intertwining group of $\tau_{\rM'}$ equals to $\rK'$. In our case, besides of this condition about intertwining group, we need to verify the second condition explained in section \ref{section 07}, which is given by Vign\'eras in \cite{V3}. After this construction, we give the first description of thesupercuspidal support of any irreducible cuspidal $k$-representation $\pi'$ in Proposition \ref{Lproposition 10}:

\begin{prop}
Let $\pi'$ be an irreducible cuspidal $k$-representation of $\rM'$, and $\pi$ an irreducible cuspidal $k$-representation of $\rM$ such that $\pi$ contains $\pi'$. Let $[\rL,\tau]$ be the supercuspidal support of $\pi$, where $\rL$ is a Levi subgroup of $\rM$ and $\tau$ an irreducible supercuspidal $k$-representation of $\rL$. Let $\tau'$ be a direct component of $\res_{\rL'}^{\rL}\tau$. The supercuspidal support of $\pi'$ is contained in the $\rM$-conjugacy class of $(\rL',\tau')$.
\end{prop}

We finish the proof of the uniquess of the supercuspidal support of $\pi'$ by proving that there is only one irreducible component $\tau_0'$ of $\res_{\rL'}^{\rL}\tau$, such that $(\rL',\tau_0')$ belongs to the supercuspidal support of $\pi'$. The idea is to study the Whittaker model of $\tau'$ and apply the derivative formula given by Bernstein and Zelevinsky in \cite{BeZe}. For this, we need to generalise their formula to the case of $k$-representations of $\rM'$. In fact, let $\mathrm{T}_{\rM}$ be a fixed maximal split torus of $\rM$ defined over $F$ and $\mathrm{T}_{\rM'}=\mathrm{T}_{\rM}\cap\rM'$. Fix $\mathrm{B}_{\rM}=\mathrm{T}_{\rM}\rU_{\rM}$ a Borel subgroup of $\rM$, and $\mathrm{B}_{\rM'}=\mathrm{T}_{\rM'}\rU_{\rM}$ a Borel subgroup of $\rM'$. There is a non-degenerate character $\dot{\theta}$ of $\rU_{\rM}$ such that the highest derivative of $\pi'$ according to $\dot{\theta}$ is non-zero. On the other hand, assume that $\rL$ is a standard Levi subgroup of $\rM$, and $\theta$ denotes $\dot{\theta}\vert_{\rU_{\rL}}$, which is also a non-degenerate character of $\rU_{\rL}$. There is only one irreducible component $\tau_0'$ of $\res_{\rL'}^{\rL}\tau$, such that the highest derivative of $\tau_0'$ according to $\theta$ is non-zero. If $\pi'$ is a subquotient of $i_{\rL'}^{\rM'}\tau'$ for some irreducible component $\tau'$ of $\res_{\rL'}^{\rL}\tau$, then the highest derivative of $i_{\rL'}^{\rM'}\tau'$ according to $\dot{\theta}$ is also non-zero (Proposition \ref{Uprop 2}). Applying the generalised formula of derivative in Corollary \ref{dcor 6} $(4)$, we obtain that the highest derivative of $i_{\rL'}^{\rM'}\tau'$ according to $\dot{\theta}$ is isomorphic to the highest derivative of $\tau'$ according to $\theta$. Hence $\tau'\cong\tau_0'$. This ends this thesis.

\section*{Acknowledgement}
The article is based on my PhD thesis. I would like to thank my advisors Anna-Maire Aubert and Michel Gros, for the suggestions of the project and support during the writing of my thesis. I thank Vincent S\'echerre for his helpful suggestions of improvements.

\section{$k$-representations of finite groups $\mathrm{SL}_n(F)$}
\label{chapter 3}
\subsection{Representation theory of finite groups}
Let $\textbf{G}'$ and $\bG$ be the connected reductive group defined over $\mathrm{F}_q$ with type $\mathrm{SL}_n$ and type $\mathrm{GL}_n$ respectively, where $q$ is a power of a prime number $p$. Note $\rG'=\textbf{G}'(\mathrm{F}_q)$ and $\rG=\bG(\mathrm{F}_q)$. We have two main purposes in this section:\\
- Prove Theorem \ref{theorem 1}.\\
- For any irreducible cuspidal $k$-representation of $\rG'$, construct its $W(k)$-projective cover.\\
Notice that the center of $\bG'$ is disconnected but the center of $\bG$ is connected, so we want to follow the method of \cite{DeLu} (page $132$), which is also applied in \cite{Bonn}: consider the regular inclusion $i: \bG' \rightarrow \bG$, then we want to use functor $\mathrm{Res}_{\rG'}^{\rG}$ to induce properties from $\rG$-representations to $\rG'$-representations. 

\subsubsection{Preliminary}\
\paragraph{Regular inclusion $i$}\
We summarize the context we will need in section 2 of [Bon]:

The canonical inclusion $i$ commutes with $F$ and maps $F$-stable maximal torus to $F$-stable maximal torus. If we fix one $F$-stable maximal torus $\bT$ of $\rG$ and note $\bT'= i^{-1}(\bT)$, then $i$ induces a bijection between the root systems of $\bG$ and $\bG'$ relative to $\bT$ and $\bT'$. Furthermore, $i$ gives a bijection between standard $F$-stable parabolic subgroups of $\bG$ and $\bG'$ with inverse $\_ \cap \bG'$, which respects subsets of simple roots contained by parabolic subgroups. Besides, restrict $i$ to any $F$-stable Levi subgroup $\bL$ of any $F$-stable parabolic subgroup of $\bG$, it is the canonical inclusion from $\bL'$ to $\bL$. 

From now on, we fix a $F$-stable maximal torus $\bT_0$ of $\bG$, whence fix one of $\bG'$ as well, noted as $\bT'_0$. For any $F$-stable standard Levi subgroup $\bL$, we always use $\bL'$ to denote the $F$-stable Levi subgroup of $\bG'$ under $i$, and use $\rL$ and $\rL'$ to denote the corresponding split Levi subgroups $\bL^{F}$ and $\bL'^{F}$ respectively.

Now we consider the dual groups. Let $(\bG^{\ast}, \bT_0^{\ast}, F^{\ast} )$ and  $(\bG'^{\ast}, \bT'^{\ast}, F^{\ast})$ be triples dual to $(\bG, \bT_0, F)$ and $(\bG', \bT', F)$ respectively. We can induce from $i$ a surjective morphism $i^{\ast}: \bG^{\ast} \rightarrow \bG'^{\ast}$, which commutes with $F^{\ast}$ and maps $\bT_0^{\ast}$ to $\bT_0^{'\ast}$. For any $F$-stable standard parabolic subgroup $\textbf{P}$ and its $F$-stable Levi soubgroup $\bL$, let us use $\textbf{P}'$ and $\bL'$ to denote the $F$-stable standard parabolic subgroups $\textbf{P}\cap\bG'$ and Levi subgroups $\bL\cap\bG'$, then we have:
$$i^{\ast} (\bL^{\ast}) = \bL'^{\ast},$$
whence, if we note $\bL'^{\ast^{F^{\ast}}}= \rL'^{\ast}$ and $\bL^{\ast^{F^{\ast}}}= \rL^{\ast}$, then:
$$i^{\ast} (\rL^{\ast}) = \rL'^{\ast}.$$

\paragraph{Lusztig series and $\ell$-blocks}\
From now on, if we consider a semisimple element $\tilde{s} \in \rL^{\ast}$ for any split Levi subgroup $\rL^{\ast}$ of $\rG^{\ast}$, we always use $s$ to denote $i^{\ast}(\tilde{s})$ and $[\tilde{s}]$ (resp. $[s]$) to denote its $\rL^{\ast}$(resp. $\rL'^{\ast}$)-conjugacy class. Notice that the order of $\tilde{s}$ is divisible by the order of $s$, whence $s$ is $\ell$-regular when $\tilde{s}$ is $\ell$-regular, where $\ell$ denotes a prime number different from $p$.

Let $\textbf{G} (\mathrm{F}_q)$ be any finite group of Lie type, where $\textbf{G}$ is a connected reductive group defined over $\mathrm{F}_q$. For any irreducible representation $\chi$ of $\textbf{G} (\mathrm{F}_q)$, let $e_{\chi}$ denote the central idempotent of $\bar{\cK}(\textbf{G}(\mathrm{F}_q))$ associated to $\chi$ (see definition in the beginning of section 2 of \cite{BrMi}). Fixing some semisimple element $s \in \textbf{G}^{\ast}(\mathrm{F}_q)$, where $\textbf{G}^{\ast}$ denotes the dual group of $\textbf{G}$, then $\mathcal{E}(\textbf{G} (\mathrm{F}_q),(s))$ denotes the Lusztig serie of $\textbf{G} (\mathrm{F}_q)$ corresponding to the $\textbf{G}^{\ast}(\mathrm{F}_q)$-conjugacy class $[s]$ of $s$. If $s$ is $\ell$-regular (i.e. its order is prime to $p$), define
$$\mathcal{E}_\ell (\textbf{G} (\mathrm{F}_q),s) := \bigcup_{t \in (C_{\textbf{G}^{\ast}}(s)^{F^{\ast}} )_\ell} \mathcal{E}(\textbf{G}(\mathrm{F}_q), (ts)).$$
Here $(C_{\textbf{G}^{\ast}}(s)^{F^{\ast}} )_\ell$ denotes the group consisting with all $\ell$-elements of $C_{\textbf{G}^{\ast}}(s)^{F^{\ast}}$, so $ts$ is still semisimple. Now define:
$$b_s= \sum_{\chi \in \mathcal{E}_\ell(\textbf{G}(\mathrm{F}_q),s)} e_{\chi},$$ 
which obviously belongs to $\bar{\cK}(\textbf{G}(\mathrm{F}_q))$. 

\begin{thm}[Brou$\mathrm{\acute{e}}$, Michel]
\label{theorem Mi,Br}
Let $s \in \textbf{G}^{\ast}(\mathrm{F}_q)$ be any semisimple $\ell$-regular element, and $\mathcal{L}'$ be the set of prime numbers without $\ell$. Define $\bar{\mathbb{Z}}_\ell = \bar{\mathbb{Z}}[1/r]_{r\in \mathcal{L}'}$, where $\bar{\mathbb{Z}}$ denotes the ring of algebraic integers, then $b_s \in \bar{\mathbb{Z}}_\ell (\textbf{G}(\mathrm{F}_q))$.  
\end{thm}

\begin{rem}
The theorem above tells us that $\mathcal{E}_\ell(\textbf{G}(\mathrm{F}_q),s)$ is an union of $\ell$-blocks.
\end{rem}

Let $K$ be a finite field extension of $\cK$ which is sufficiently large for $\rG$, with valuation ring $\rO$, which contains $W(k)$ as a subring. We have $\cK$ is complete, implying $K$ is complete. Notice that $K$ is also sufficiently large for any split Levi subgroup $\rL$ of $\rG$, which means all the irreducible $\bar{\cK}$-representations of $\rL$ is defined over $K$, whence there is a natural bijection:
$$\mathrm{Irr}_{\bar{\cK}}(\rL) \longleftrightarrow \mathrm{Irr}_K(\rL),$$
so we define Lusztig series for $\mathrm{Irr}_K(\rL)$ through this bijection. Since $K$ is also sufficiently large for any split Levi $\rL'$ of $\rG'$, we have the same bijection for $\mathrm{Irr}_K(\rL')$ and $b_s \in \rO[\rL']$.

\begin{prop}
\label{proposition 2.3}
For any split Levi subgroup $\rL$ (resp. $\rL'$) and any semisimple $\ell$-regular element $\tilde{s} \in \rL^{\ast}$ (resp. $s \in \rL'^{\ast}$), we have: $b_{\tilde{s}} \in \rO[\rL]$ (resp. $b_s \in \rO[\rL']$).
\end{prop}

\begin{proof}
We deduce from the analysis above and the definition that $e_{\chi} \in K(\rL)$. Combining this with theorem \ref{theorem Mi,Br}, we conclude that $b_{\tilde{s}} \in \rO[\rL]$. The same for $b_s$'s. 
\end{proof}

\paragraph{Gelfand-Graev lattices and its projective direct summands}\
For any split Levi subgroup $\rL'$ of $\rG'$, fix one rational maximal torus $\rT'$ and let $\mathrm{B}_{\rL'}'$ be a standard split Borel subgroup with unipotent radical $\rU_{\rL'}'$, then $\mathrm{O_{U}}(\rL')$ denotes the set of non-degenerated characters of $\rU_{\rL'}'$. Consider any $\mu \in \rO_{\rU} (\rL')$, of which the representation space is $1$-dimensional, so it obviously has an $\rO [\rU_{\rL'}']$-lattice , noted as $\rO_{\mu}$. Define $\rY_{\rL', \mu} = \Ind_{\rU_{\rL'}'}^{\rL'} \rO_{\mu}$, the \textbf{Gelfand-Graev lattice} associated to $\mu$. In fact, we have that $\rY_{\rL', \mu}$ is defined up to the $\rT'$-conjugacy class of $\mu$. Take any $\ell$-regular semisimple element $s \in \rL'^{\ast}$, define:
$$\rY_{\rL', \mu, s} = b_s \cdot \rY_{\rL', \mu}.$$
Meanwhile, from definition we have directly that $\sum_{[s]}b_s=1$, where the sum runs over all the $\ell$-regular semisimple $\rL'^{\ast}$-conjugacy class $[s]$. So:
$$\rY_{\rL', \mu} = \bigoplus_{[s]} \rY_{\rL', \mu, s}.$$
Since $\rO_{\mu}$ is projective (free and rank 1) and induction respect projectivity, we see that $\rY_{\rL', \mu}$ is a projective $\rO[\rL']$-module. Proposition \ref{proposition 2.3} implies that $\rY_{\rL', \mu, s}$ are  $\rO[\rL']$-modules and direct components of projective $\rO[\rL']$-module $\rY_{\rL', \mu}$, so we conclude that $\rY_{\rL', \mu, s}$ are projective $\rO[\rL']$-modules.

Let $\rG$ be the group of $\mathrm{F}_q$ points of an algebraic group defined over $\mathrm{F}_q$, and $(\mathcal{K},\mathcal{O},k)$ be a splitting $\ell$-modular system. We define 
$$\mathcal{E}_{\ell'}(\rG):=\bigcup_{z \text{ semi-simple, }\ell-\text{regular}}\mathcal{E}(\rG,z)$$

\begin{defn}[Gruber, Hiss]
\label{defn fin aa}
Let $\rG$ be the group of $\mathrm{F}_q$ points of an algebraic group defined over $\mathrm{F}_q$, and $(\mathcal{K},\mathcal{O},k)$ be a splitting $\ell$-modular system. Let $Y$ be an $\mathcal{O}[\rG]$-lattice with ordinary character $\psi$. Write $\psi=\psi_{\ell'}+\psi_{\ell}$, such that all constituents of $\psi_{\ell'}$ and non of $\psi_{\ell}$ belong to $\mathcal{E}_{\ell'}(\rG)$. Then there exists a unique pure sublattice $V\leq Y$, such that $Y\slash V$ is an $\mathcal{O}[\rG]$-lattice whose character is equal to $\psi_{\ell'}$. The quotient $Y\slash V$ is called the $\ell$-regular quotient of $Y$ and noted by $\pi_{\ell'}(Y)$.
\end{defn}

\begin{cor}
Let $\rL'$ be a split Levi subgroup of $\rG'$, and $s$ be an $\ell$-regular semisimple element in $\rL'^{\ast}$. For any $\mu \in \rO_{\mathrm{U}}(\rL')$, the module $\rY_{\rL', \mu, s}$ is indecomposable.
\end{cor}

\begin{proof}
Since $\rY_{\rL', \mu, s}$ is a projective $\rO [\rL']$-module, the section \S $4.1$ of \cite{GrHi} or Lemma 5.11(Hiss) in \cite{Geck} tells us that it is indecomposable if and only if its $\ell$-regular quotient $\pi_{l'} (\rY_{\rL', \mu, s})$ (see \S 3.3 in \cite{GrHi}) is indecomposable. Inspired by section 5.13. of \cite{Geck}, we consider $K \otimes \pi_{l'}(\rY_{\rL', \mu, s})$, which is the irreducible sub-representation of $K \otimes \rY_{\rL', \mu}$ lying in Lusztig serie $\mathcal{E}(\rL', (s))$. The module $\pi_{\ell'}(\rY_{\rL',\mu,s})$ is torsion-free, so we deduce that $\pi_{l'}(\rY_{\rL', \mu, s})$ is indecomposable.
\end{proof}

\begin{prop}
\label{proposition 2.5}
Let $\rL'$ be split Levi subgroup of $\rG'$, and $\mu \in \rO_{\rU}(\rL')$. All the projective indecomposable direct summands $\rY_{\rL', \mu ,s }$ of Gelfand-Graev lattice $\rY_{\rL',\mu}$ are defined over $W(k)$, which means there exist projective $W(k)[\rL']$-modules $\mathcal{Y}_{\rL',\mu,s}$ such that $\mathcal{Y}_{\rL',\mu,s} \otimes \rO= \rY_{\rL', \mu, s}$. In particular, $\mathcal{Y}_{\rL',\mu,s}$ are indecomposable.
\end{prop}

\begin{proof}
Notice that $\rU_{\rL'}'$ are $p$-groups, whence $\mu$ is defined over $\mathcal{K}$, and there is a $W(k)[\rU_{\rL'}']$-module $\mathcal{O}_{\mu}$ such that $\rO_{\mu}= \mathcal{O}_{\mu} \otimes_{W(k)[\rU'_{\rL'}]} \rO$. Define a projective $W(k)[\rL']$-module $\mathcal{Y}_{\rL',\mu}= \mathrm{Ind}_{\rU_{\rL'}'}^{\rL'}(\mathcal{O}_{\mu})$. Since $k$ is algebraically closed, then $\bar{\rY}_{\rL',\mu}$, the reduction modulo $\ell$ of $\rY_{\rL',\mu}$, coincides with $\bar{\mathcal{Y}}_{\rL',\mu}$, the reduction modulo $\ell$ of $\mathcal{Y}_{\rL',\mu}$. Proposition 42 (b) and Lemma 21 (b) in \cite{Serre} imply that the decomposition of $\rY_{\rL',\mu}$ gives an decomposition:
$$\bar{\rY}_{\rL',\mu}= \sum_{[s]} \bar{\rY}_{\rL',\mu,s}.$$
By the same reason, this gives an decomposition by indecomposable projective modules:
$$\mathcal{Y}_{\rL',\mu} = \mathcal{Y}_{\rL',\mu,s}$$
such that the reduction modulo $\ell$ of $\mathcal{Y}_{\rL',\mu,s}$ equals to $\bar{\rY}_{\rL',\mu,s}$. In particular, we can check directly through Proposition 42 (b) of \cite{Serre} that $\mathcal{Y}_{\rL',\mu,s} \otimes \rO= \rY_{\rL', \mu, s}$.
\end{proof}

\begin{rem}
Since $\rU_{\rL'}'$ is also the unipotent radical of $\mathrm{B}_{\rL}$, where $\mathrm{B}_{\rL'}'$ is the inverse image of the regular inclusion $i$ of $\mathrm{B}_{\rL}$. We can repeat the proof for $\rY_{\rL,\tilde{s}}$ and see that they are also defined over $W(k)$ in the same manner.
\end{rem}

For split Levi subgroup $\rL$ of $\rG$, we know from \cite{DiMi} that if we fix one rational maximal torus and define $\mathrm{O_{U}}(\rL)$, this set of non-degenerate characters consists with only one orbit under conjugation of the fixed torus. So the Gelfand-Graev lattice is unique, and we note it as $\rY_{\rL}$. All the analysis above still work for $\rY_{\rL}$, and take some $\ell$-regular semisimple element $\tilde{s} \in \rL^{\ast}$. In particular, we use $\rY_{\rL, \tilde{s}}$ to denote the indecomposable projective direct summand $b_{\tilde{s}} \cdot \rY_{\rL}$. 


\begin{cor}
\label{corollary 3.1}
Let $\tilde{s} \in \rL^{\ast}$ be a semisimple $\ell$-regular element, then:
$$\mathrm{Res}_{\rL'}^{\rL} (b_{\tilde{s}} \cdot \rY_{\rL}) \hookrightarrow b_{s} \cdot \mathrm{Res}_{\rL'}^{\rL} \rY_{\rL}.$$
\end{cor}

\begin{proof}
We know directly from definition that, for any semisimple $\ell$-regular $s' \in \rG'^{\ast}$:
$$b_{s'} \cdot \mathrm{Res}_{\rL'}^{\rL} (b_{\tilde{s}} \cdot \rY_{\rL}) \hookrightarrow b_{s'} \cdot \mathrm{Res}_{\rL'}^{\rL} \rY_{\rL},$$
Meanwhile $b_{s'} \cdot \mathrm{Res}_{\rL'}^{\rL} (b_{\tilde{s}} \cdot \rY_{\rL})$ is a projective $\rO [\rG']$-module, so it is free when considered as $\rO$-module. And Proposition 11.7 in \cite{Bonn} told us that $b_{s'} \cdot \mathrm{Res}_{\rL'}^{\rL} (b_{\tilde{s}} \cdot \rY_{\rL}) \otimes \bar{\rK} = 0$ if $[s'] \neq [s]$ with $s= i^{\ast}(\tilde{s})$, which means $b_{s'} \cdot \mathrm{Res}_{\rL'}^{\rL} (b_{\tilde{s}} \cdot \rY_{\rL})  = 0$.
Combine this with
$$\bigoplus_{[s']} b_{s'} \cdot \mathrm{Res}_{\rL'}^{\rL} (b_{\tilde{s}} \cdot \rY_{\rL})= \mathrm{Res}_{\rL'}^{\rL} (b_{\tilde{s}} \cdot \rY_{\rL}),$$
we obtain the result.
\end{proof}

\begin{prop}
\label{proposition 3}
For any split Levi subgroup $\rL$ of $\rG$, let $\rL'$ denote the split Levi subgroup $\rL\cap \rG'$ of $\rG'$, and $\rZ(\rL)$, $\rZ(\rL')$ denote the center of $\rL$ and $\rL'$ respectively, then we have the equation:
$$\mathrm{Res}_{\rL'}^{\rL} \rY_{\rL} = \vert \rZ(\rL) : \rZ(\rL') \vert \bigoplus_{[\mu] \in \rO_{\mathrm{U}}(\rL')} \rY_{\rL',\mu},$$
where $[\mu]$ denote the $\rT'$-orbit of $\mu$.
\end{prop}

\begin{proof}
Let $\mathrm{B}$ be a split Borel subgroup of $\rL$ and $\mathrm{B}'= \mathrm{B} \cap \rL'$ the corresponding split Borel of $\rL'$, and $\rU'$ denotes the unipotent radical of $\mathrm{B}'$, observing that $\rU'$ is also the unipotent radical of $\mathrm{B}$. Fixing one non-degenerate character $\mu$ of $\rU'$, let $\rO_{\mu}$ be its $\rO[\rU']$-lattice. By the transitivity of induction, we have:
$$\rY_{\rL}= \mathrm{Ind}_{\rL'}^{\rL} \circ \mathrm{Ind}_{\rU'}^{\rL'} \rO_{\mu} = \mathrm{Ind}_{\rL'}^{\rL} \rY_{\rL', \mu}.$$
Since $[\rT:\rT']=[\rL: \rL']$, by using Mackey formula we have:
$$\mathrm{Res}_{\rL'}^{\rL} \rY_{\rL} = \bigoplus_{\alpha_i \in [\rT : \rT']} ad(\alpha_i) ( \rY_{\rL', \mu}).$$
Furthermore, $ad(\alpha_i) ( \mathrm{Ind}_{\rU'}^{\rL'} \rO_{\mu} ) = \mathrm{Ind}_{\rU'}^{\rL'} ( ad(\alpha_i) ( \rO_{\mu}) )$. 
Notice that after fixing one character of $\rU'$, all its $\rO[\rU']$-lattices are equivalent, so $ad(\alpha_i) ( \rY_{\rL', \mu} )= \rY_{\rL', ad(\alpha_i) (\mu)}$. Whence, let $[\mu]$ denote the $\rT'$-orbit of $\mu$ in $\rO_{\mathrm{U}}(\rL')$, we have
$$\mathrm{Stab_{\rT}([\mu])} \subset \mathrm{Stab}_{\rT}(\rY_{\rL',\mu}) \subset \mathrm{Stab}_{\rT}(\rY_{\rL', \mu} \otimes \bar{\rK}).$$
On the other hand, the proof of lemma 2.3 a) in \cite{DiFl} tells that 
$$\mathrm{Stab}_{\rT}(\rY_{\rL', \mu} \otimes \bar{\rK}) \subset \mathrm{Stab_{\rT}([\mu])}.$$
So the inclusion above is in fact a bijection.
Combine this with the statement of lemma 2.3 a) in \cite{DiFl}, we finish our proof.
\end{proof}

\begin{lem}
\label{lemma a.1}
Fix a semisimple $\ell$-regular $s \in \rG'^{\ast}$, define $\mathcal{S}_{[s]}$ to be the set of semisimple $\ell$-regular $\widetilde{\rG}^{\ast}$-conjugacy classes $[\tilde{s}] \subset \widetilde{\rG}^{\ast}$ such that $i^{\ast}[\tilde{s}]=[s]$. Then
$$\bigoplus_{[\tilde{s}] \in \mathcal{S}_{[s]}} \mathrm{Res}_{\rL'}^{\rL} \rY_{\rL,\tilde{s}}= \vert \rZ(\rL) : \rZ(\rL') \vert \bigoplus_{\mu \in \rO_{\rU'}(\rL')} \rY_{\rL',\mu,s} $$
\end{lem}

\begin{proof}
By definition, $\rY_{\rL, \tilde{s}}= b_{\tilde{s}} \cdot \rY_{\rL}$. Multiplying $b_s$ on both sides of the equation in Proposition \ref{proposition 3}  and considering Corollary \ref{corollary 3.1}, we conclude that for any $\ell$-regular semisimple $\rG'^{\ast}$-conjugacy class $[s]$, $\bigoplus_{[\tilde{s}] \in \mathcal{S}_{[s]}} \mathrm{Res}_{\rL'}^{\rL} \rY_{\rL,\tilde{s}}$ is a projective direct summand of $\vert \rZ(\rL) : \rZ(\rL') \vert \bigoplus_{\mu \in \rO_{\rU}(\rL')} \rY_{\rL',\mu,s}$. Meanwhile, let $\mathcal{S}=\{ \mathcal{S}_{[s]}\vert\text{ } s\in\rG^{'\ast}, s\text{ semisimlpe }\ell\text{-regular} \}$, then Proposition \ref{proposition 3} can be written as:
$$\bigoplus_{\mathcal{S}}\bigoplus_{[\tilde{s}] \in \mathcal{S}_{[s]}} \mathrm{Res}_{\rL'}^{\rL} \rY_{\rL,\tilde{s}}= \vert \rZ(\rL) : \rZ(\rL') \vert \bigoplus_{[s]} \bigoplus_{\mu \in \rO_{\rU}(\rL')} \rY_{\rL',\mu,s} $$
So they equal to each other.
\end{proof}

\subsubsection{Uniqueness of supercuspidal support}
In this part, we will proof the main theorem \ref{theorem 1} for this section. First we recall the notions of cuspidal, supercuspidal and supercuspidal support.

We always use $i$ and $r$ to denote the functors of parabolic induction and parabolic restriction. Let $\pi$ be an irreducible $k$-representation of a finite group of Lie type or a $p$-adic group $\rG$. We say $\pi$ is cuspidal, if for any proper Levi subgroup $\rL$ of $\rG$, the representation $r_{\rL}^{\rG}\pi$ is $0$. Let $\tau$ be any irreducible $k$-representation of $\rL$, if $\pi$ is not isomorphic to any irreducible subquotient of $i_{\rL}^{\rG}\tau$ for any pair $(\rL,\tau)$, we say $\pi$ is supercuspidal. It is clear that a supercuspidal representation is cuspidal. We say $(\rL,\tau)$ is a (super)cuspidal pair, if $\tau$ is a (super)cuspidal $k$-representation of $\rL$

The cuspidal (resp. supercuspidal) support of $\pi$ consists of the cuspidal (resp. supercuspidal) pairs $(\rL,\tau)$, such that $\pi$ is an irreducible subrepresentation (resp. subquotient) of $i_{\rL}^{\rG}\tau$.

\begin{thm} \label{theorem 1}
Let $\rL'$ be any standard split Levi subgroup of $\rG'$ and $\nu$ be any cuspidal $k$-representation of $\rL'$.Then the supercuspidal support of $\nu$ is unique up to $\rL'$-conjugation.
\end{thm}

Let $\mathrm{P}_{\nu}$ denote the $\rO[\rL']$-projective cover of $\nu$. To prove the theorem above, we will follow the strategy below:

\begin{enumerate}
\item For any standard Levi subgroup $\rM'$ of $\rL'$, prove that $r_{\rM'}^{\rL'} \mathrm{P}_{\nu}$ is either equal to $0$ or indecomposable.
\item Prove that there is only one unique standard split Levi subgroup $\rM'$ of $\rL'$, such that $r_{\rM'}^{\rL'} \mathrm{P}_{\nu}$ is cuspidal.
\end{enumerate}
Let $(\rM', \theta)$ be a supercuspidal $k$-pair of $\rL'$. From the proof of Proposition 3.2 of \cite{Hiss}, we know that $(\rM', \theta)$ belongs to the supercuspidal support of $(\rL', \nu)$, if and only if $\mathrm{Hom}(r_{\rM'}^{\rL'} \mathrm{P}_{\nu}, \theta) \neq 0$. Combining this fact with $(1)$, we find that $r_{\rM'}^{\rL'}\mathrm{P}_{\nu}$ is the projective cover of $\theta$. Proposition $2.3$ of \cite{Hiss} states that an irreducible $k$-representation of $\rM'$ is supercuspidal if and only if its projective cover is cuspidal, whence Theorem \ref{theorem 1} is equivalent to $(2)$.

\begin{rem}
\label{remark 2.11}
- If we consider standard Levi subgroups $\rL$ of $\rG$, the analysis above is true as well.

- Proposition 3.2 of \cite{Hiss} concerns $k[\rL']$-projective cover, but from Proposition 42 of \cite{Serre} we know that there is a surjective morphism of $k[\rL']$-modules from the $W(k)[\rL']$-projective cover to the $k[\rL']$-projective cover, and hence obtain the same result for $W(k)[\rL']$-projective cover.
\end{rem}

\begin{prop}
\label{prop a}
Let $\nu$ be an irreducible cuspidal $k$-representation of $\rL'$. There exists a simple $k\rL$-module $\widetilde{\nu}$, and a surjective morphism $\mathrm{Res}_{\rL'}^{\rL} \widetilde{\nu} \twoheadrightarrow \nu $. Furthermore, let $\rY_{\rL, \tilde{s}}$ be the projective cover of $\widetilde{\nu}$, where $\tilde{s} \in
\rG^{\ast}$ is an $\ell$-regular semisimple element, then there exists $\mu \in \rO_{\rU'}(\rL')$ such that the composed morphism:
\begin{equation}
\label{equa 1}
\rY_{\rL', \mu, s}\hookrightarrow \mathrm{Res}_{\rL'}^{\rL} \rY_{\rL,\tilde{s}} \twoheadrightarrow 
\mathrm{Res}_{\rL'}^{\rL} \widetilde{\nu} \twoheadrightarrow \nu
\end{equation}
is surjective, which means $\rY_{\rL',\mu,s}$ is the $\rO[\rL']$-projective cover of $\nu$.
\end{prop} 

\begin{proof}
By the property of Mackey formula, we can find such $\widetilde{\nu}$.

For the second part of this proposition, since $\mathrm{Res}$ respect projectivity, we know the fact that $\mathrm{Res}_{\rL'}^{\rL} \rY_{\rL, \tilde{s}}$ is a projective direct summand of $\mathrm{Res}_{\rL'}^{\rL} \rY_{\rL}$, and contained in $\vert \rZ(\rL) : \rZ(\rL') \vert \bigoplus_{\mu \in \rO_{\rU'}(\rL')} \rY_{\rL',\mu,s}$ as lemma \ref{lemma a.1} proved, whence all the projective indecomposable direct summands belong to $\{ \rY_{\rL', \mu, s} \}_{\rO_{\rU'}(\rL')}$. Then there must exists $\mu \in \rO_{\rU'}(\rL')$ such that the composed morphism: $\rY_{\rL', \mu, s} \rightarrow \nu$ is non trivial hence a surjection.
\end{proof}

\begin{rem}
We induce from Proposition \ref{proposition 2.5} that $\mathcal{Y}_{\rL',\mu,s}$ is the $W(k)[\rL']$-projective cover of $\nu$, noted as $\mathcal{P}_{\nu}$.
\end{rem}

Let $\rM'$ be any standard split Levi subgroup of $\rL'$. It is clear that $\mu \downarrow_{\rM'}$ belongs to $\rO_{\rU'}(\rM')$. Now consider the intersection $[s] \cap \rM'^{\ast}$. As in the paragraph above Proposition 5.10 of \cite{Helm}, $[\tilde{s}] \cap \rM^{\ast}$ consists of one $\rM^{\ast}$-conjugacy class or is empty, so does $[s] \cap \rM'^{\ast}$. For the first case, notation $\rY_{\rM', \mu \downarrow_{\rM'}, [s] \cap \rM'^{\ast}}$ is well defined, and for the second case, we define it to be 0. From now on, we will always use $\rY_{\rM', \mu, s}$ to simplify $\rY_{\rM', \mu \downarrow_{\rM'}, [s] \cap \rM'^{\ast}}$. We use the same manner to define $\rY_{\rM, \tilde{s}}$.

\begin{prop}
\label{prop b}
Let $\nu$ be an irreducible cuspidal $k\rL'$-representation, and $\widetilde{\nu}$, $\rY_{\rL', \mu, s}$, $\rY_{\rL, \tilde{s}}$ be as in Proposition $\ref{prop a}$. Then $r_{\rM'}^{\rL} \rY_{\rL', \mu, s}$ is equal to $0$ or indecomposable and isomorphic to $\rY_{\rM', \mu, s}$ as $\rO[\rM']$-module.
\end{prop}

\begin{proof}
In the proof of lemma \ref{lemma a.1} we know that $\rY_{\rL',\mu,s}$ is a direct summand of $\mathrm{Res}_{\rL'}^{\rL}(\rY_{\rL,\tilde{s}})$. Observing that the unipotent radical of $\rM'$ is also the unipotent radical of $\rM$, we deduce directly from the definition that $r_{\rM'}^{\rL'} (\mathrm{Res}_{\rL'}^{\rL}(\rY_{\rL,\tilde{s}})) = \mathrm{Res}_{\rM'}^{\rM} (r_{\rM}^{\rL}(\rY_{\rL,\tilde{s}}))$, and Proposition 5.10 in \cite{Helm} states that $r_{\rM}^{\rL} (\rY_{\rL,\tilde{s}})= \rY_{\rM,\tilde{s}}$. The statements above, combining with the fact that parabolic restriction is exact and respects projectivity, derive that  $r_{\rM'}^{\rL'} \rY_{\rL',\mu,s}$ is a projective direct summand of $\mathrm{Res}_{\rM'}^{\rM}\rY_{\rM, \tilde{s}}$. As what we have mentioned, $[\tilde{s}] \cap \rM^{\ast}$ is empty or consists of one $\rM^{\ast}$-conjugacy class, so does $[s] \cap \rM'^{\ast}$. In the first case $\rY_{\rM, \tilde{s}}= 0$, whence $r_{\rM'}^{\rL} \rY_{\rL',\mu,s}=0$, so the result.

Now considering the second case: let $\tilde{s}' \in \rM^{\ast}$ and $[\tilde{s}']$ denote the $\rM^{\ast}$-conjugacy class equals to $[\tilde{s}] \cap \rM^{\ast}$. Let $\mu'$ denote the character $\mathrm{Res}_{\rU'_{\rL'}}^{\rU'_{\rM'}} \mu$, where $\rU'_{\rL'}$ and of $\rU'_{\rM'}$ denote the unipotent radical of $\rL'$ and $\rM'$ respectively, which is non-degenerate by definition. Corollary 15.15 in [Bon] gives an equation:
$$r_{\rM'}^{\rL'} \rY_{\rL', \mu, s} \otimes \bar{\rK}= \rY_{\rM',\mu',s'} \otimes \bar{\rK}.$$
which means the $\ell$-regular quotient of $r_{\rM'}^{\rL'} \rY_{\rL', \mu, s}$ is indecomposable, and by using the criterion of  \cite[lemma 5.11 ]{Geck} we conclude that $r_{\rM'}^{\rL'} \rY_{\rL', \mu, s}$ is indecomposable. Note that Corollary 15.11 in [Bon] tells that the sub-representation of $\mathrm{Res}_{\rM'}^{\rM}\rY_{\rM, \tilde{s}} \otimes \bar{\rK}$ corresponding to $[s']$ is without multiplicity, and the equation above says that the irreducible sub-representations corresponding to $[s']$ of $r_{\rM'}^{\rL'} \rY_{\rL', \mu, s} \otimes \bar{\rK}$ and $\rY_{\rM',\mu',s'} \otimes \bar{\rK}$ coincide, whence these two projective direct summands of $\mathrm{Res}_{\rM'}^{\rM}\rY_{\rM, \tilde{s}}$ coincide each other.

\end{proof}

We have finished the first step to prove Theorem \ref{theorem 1}. Remark \ref{remark 2.11} tells that the statement of step 2 is true for $\rL$, whence there only left the proposition below to finish our proof:

\begin{prop}
Let $\rY_{\rL', \mu, s}$, $\rY_{\rL,\tilde{s}}$, $\widetilde{\nu}$ be as in Proposition \ref{prop a}, then for any standard split Levi $\rM'$ of $\rL'$, we have $r_{\rM'}^{\rL'} \rY_{\rL',\mu, s}= \rY_{\rM', \mu, s}$ is cuspidal if and only if $r_{\widetilde{M}}^{\rL} \rY_{\rL, \tilde{s}}= \rY_{\rM, \tilde{s}}$ is cuspidal.
\end{prop}
 
\begin{proof}
Since $\rL' \hookrightarrow \rL$ is a bijection preserving partial order between standard Levi subgroups of $\rG$ and $\rG'$, the statement in the proposition is equivalent to say that for any split Levi $\rM'$ of $\rL'$, 
$$r_{\rM'}^{\rL'} \rY_{\rL',\mu, s} =0 \iff r_{\rM}^{\rL} \rY_{\rL, \tilde{s}}=0.$$
The proof of Proposition \ref{prop b} tells us
$$r_{\rM'}^{\rL'} \rY_{\rL',\mu, s} \hookrightarrow \mathrm{Res}_{\rM'}^{\rM} \rY_{\rM,\tilde{s}},$$
whence "$\Rightarrow$" is clear.

Now consider the other direction. Notice that $r_{\rM'}^{\rL'} \rY_{\rL',\mu,s}$ is an $\rO[\rM']$-lattice, and definition 5.9 in \cite{Geck} tells us that $r_{\rM'}^{\rL'} \rY_{\rL',\mu, s}=0$ if and only if its $\ell$-regular quotient $\pi_{l'} (r_{\rM'}^{\rL'}\rY_{\rL', \mu, s})=0.$ 
By definition $(\pi_{l'} (r_{\rM'}^{\rL'} \rY_{\rL', \mu, s})) \otimes \bar{\rK}$ is the sum of all simple $\bar{\rK}[\rM']$-submodules of $r_{\rM'}^{\rL'}( \rY_{\rL', \mu, s} \otimes \bar{\rK})$ which lie in the Lusztig series corresponding to $\ell$-regular semisimple $\rM'^{\ast}$-conjugacy classes. \cite[Corollary 15.15]{Bonn} states that, in fact $\bar{\rK}[\rM']$-module $(\pi_{l'} (r_{\rM'}^{\rL'} \rY_{\rL', \mu, s})) \otimes \bar{\rK}$ is the sum of all irreducible $\bar{\rK}[\rM']$-submodules of Gelfand-Graev representation $\mathrm{Ind}_{\rU'_{\rM'}}^{\rM'} \mu$ lying in the Lusztig series corresponding to $[s] \cap \rM'^{\ast}$, where $[s]$ denotes the $\rL'^{\ast}$-conjugacy class. We have now $(\pi_{l'} (r_{\rM'}^{\rL'} \rY_{\rL', \mu, s})) \otimes \bar{\rK}=0$ implies $[s] \cap \rM'^{\ast}=0$, which means $[\tilde{s}] \cap \rM^{\ast}=0$, whence $\rY_{\rM,{\tilde{s}}}=0$.
 
\end{proof}

\section{Maximal simple cuspidal $k$-types}
\label{chapter 4}

\subsection{Construction of cuspidal $k$-representations of $\rG'$}
\label{section 08}
From this section until the end of this manuscript, we assume that the field $F$ is non-archimedean locally compact, whose residue field is of characteristic $p(\neq l)$. Let $\rG'$ denote $\mathrm{SL}_n(\rF)$ and $\rG$ denote $\mathrm{GL}_n(\rF)$. Let $\mathrm{Rep}_{k}(\rG')$ denote the category of smooth $k$-representations of $\rG'$.

In this section, we want to prove that for any $k$-irreducible cuspidal representation $\pi'$ of $\rG'$, there exists an open compact subgroup $\tilde{J}'$ of $\rG'$ and an irreducible representation $\tilde{\lambda}'$ of $\tilde{J}'$ such that $\pi'$ is isomorphic to $\ind_{\tilde{J}'}^{\rG'}\tilde{\lambda}'$ (\ref{thm 16}, \ref{cor 21} and \ref{defn 22}).

\subsubsection{Types $(J,\lambda\otimes\chi\circ\det)$}
Let $(J,\lambda)$ be a maximal simple cuspidal $k$-type of $\rG$, and we need to check that the type $(J,\lambda\otimes\chi\circ\det)$ is also a maximal simple cuspidal $k$-type of $\rG$, which will be used in the proof of Proposition \ref{prop 1}. This has been proved in appendix of \cite{BuKuII} in the case of characteristic $0$, and by using the following two lemmas, we observe the same results for the case of characteristic $\ell$ by reduction modulo $\ell$.

\begin{defn}
\label{defn 1122}
Let $[\fA,n,0,\beta]$ be a simple stratum, and $\theta$ be a simple $k$-character or a simple $\obQ_{\ell}$ character of $H^{1}$, and $\eta$ the unique irreducible $k$-representation of $J^1$ which contains $\theta$, and $\kappa$ an $\beta$-extension of $\eta$ to $J$. Let $(J,\lambda)$ be a simple $k$-type or a simple $\obQ_{\ell}$-type of $\rG$. And also all the notations used: $\mathfrak{H}(\beta,\fA)$,  $\mathfrak{J}(\beta,\fA)$ are defined in \S 3.1.7 and \S 3.1.8 of \cite{BuKu}.
\end{defn}

\begin{prop} [Vign\'eras, IV.1.5 in \cite{V2}]
\label{prop 2'}
The reduction modulo $\ell$ of any maximal simple cuspidal  $\bar{\Q}_\ell$-type of $\rG$ is a maximal simple cuspidal $k$-type. And conversely, any maximal simple cuspidal $k$-type is the reduction modulo $\ell$ of a maximal simple cuspidal $\bar{\Q}_\ell$-type of $\rG$.
\end{prop}

\begin{lem}
\label{lem 3}
Let $K$ be a compact subgroup of a $p$-adic reductive group. Any $\obQ_{\ell}$-character of $K$ is $\ell$-integral, and reduction modulo $\ell$ gives a surjection from the set of $\obQ_{\ell}$-characters of $K$ to the set of $\bar{\mathbb{F}}_\ell$-characters of $K$
\end{lem}

\begin{proof}
Fix an isomorphism from $\mathbb{C}$ to $\bar{\mathbb{Q}}_\ell$ of fields and until the end of this proof, we identify the two fields through this isomorphism. Let $\chi_{\bC}$ be any $\bC$-character of $K$. The smoothness implies that there exists a finite field extension $\mathrm{E}\slash \Q_\ell$ such that $\chi_{\bC}$ is defined over $\mathrm{E}$ and we can find an $\rO_{\mathrm{E}}[K]$-lattice of $\chi$. Hence we can define reduction modulo $\ell$ for $\chi_{\bC}$ and denote is as $\bar{\chi}_{\bC}$. On the other hand, let $\chi_\ell$ be any $\bar{\mathbb{F}}_\ell$-character of $K$. It is clear that $\chi_\ell$ is defined over a finite field extension $\bar{\mathrm{E}}/\mathbb{F}_\ell$. Notice that the quotient group $K\slash \mathrm{Ker}(\chi_\ell)$ is a finite abelian group with order prime to $\ell$. Lemma 10 of section 14.4 in \cite{Serre} implies that $\chi_\ell$ is projective as $\bar{\mathrm{E}}[K\slash\mathrm{Ker}(\chi_\ell)]$-module. Then Proposition 42 of \cite{Serre} states that $\chi_\ell$ can be lift to $\rO_{\bar{\mathrm{E}}}$, where the fractional field $\mathrm{frac}(\rO_{\bar{\mathrm{E}}})$ is a finite field extension of $\Q_\ell$ and its residual field is isomorphic to $\bar{\mathrm{E}}$. Hence we finish the proof.
\end{proof}

\begin{cor}
\label{cor 4}
Let $\bar{\chi}$ be any $k^{\times}$-character of $\mathrm{F}^{\times}$, then it can be always lifted to a $\bar{\Q}_\ell$-character $\chi$ of $\mathrm{F}^{\times}$.
\end{cor}

\begin{proof}
We could write $F^{\times}\cong\mathbb{Z}\times\rO_{F}^{\times}$ and $\bar{\chi}$ is uniquely defined by $\bar{\chi}\vert_{\mathbb{Z}}$ and $\bar{\chi}\vert_{\rO_F^{\times}}$. The part $\bar{\chi}\vert_{\rO_F^{\times}}$ can be lift to a $\bar{\Q}_\ell$-character by lemma \ref{lem 3}. It is left to consider the restriction $\bar{\chi}\vert_{\mathbb{Z}}$, of which the image is a finite group of order prime to $\ell$. Thus we could find a finite field extension $K$ of $\Q_\ell$ such that there is an embedding from $\bar{\chi}(\mathbb{Z})$ to the quotient ring $\rO_K/ p_K$, where $p_K$ is the uniformizer of $\rO_K$.
\end{proof}

From now on we fix a continuous additive character $\psi_F$ from $F$ to $\mathbb{C}^{\times}$, which is null on $\mathfrak{p}_{F}$ (the unique maximal ideal of $\mathfrak{o}_F$) but non-null on $\mathfrak{o}_F$ (the ring of integers of $F$). Recall the equivalence (depending on the choice of $\psi_F$)
$$(\mathrm{U}^{[\frac{1}{2}n]+1}(\fA)/ \mathrm{U}^{n+1}(\fA))^{\wedge}\cong \mathfrak{P}^{-n}/\mathfrak{P}^{-([\frac{1}{2}n]+1)},$$
where $(\mathrm{U}^{[\frac{1}{2}n]+1}(\fA)/ \mathrm{U}^{n+1}(\fA))^{\wedge}$ denote the Pontrjagin dual. Let $\beta\in\mathfrak{P}^{-n}/\mathfrak{P}^{-([\frac{1}{2}n]+1)}$, we use $\psi_{\beta}$ to denote the character on $\mathrm{U}^{[\frac{1}{2}n]+1}(\fA)/ \mathrm{U}^{n+1}(\fA)$ induced through the equivalence above (or consult \S 1.6.6 of \cite{BuKu} for explicite definition). Let $\bar{\psi_{\beta}}$ denote the reduction modulo $\ell$ of $\psi_{\beta}$ according to the choice of $\psi_F$.

\begin{lem}
\label{lem 2'}
Let $(J,\lambda)$ be a maximal simple cuspidal $k$-type of $\rG$, if $(J,\lambda)$ is of level zero or $\beta\in F$, then $(J,\lambda\otimes\chi\circ\det)$ is also a maximal simple cuspidal $k$-type of $\rG$, where $\chi$ is any $k$-quasicharacter of $F^{\times}$. In particular, while $\chi$ is not trivial on $\rU^1(\fA)$. Let $n_0\ge1$ is the least integer such that $\chi\circ\det$ is trivial on $\rU^{n_0+1}(\fA)$, and $c\in \mathfrak{P}^{-n_0}$ such that $\chi\circ\det$ coincides with $\bar{\psi_c}$ on $\mathrm{U}^{[\frac{1}{2}n_0]+1}(\fA)$. Then 
$$\mathfrak{H}(\beta+c, \fA)=\mathfrak{H}(\beta,\fA)=\mathfrak{J}(\beta+c,\fA)=\mathfrak{J}(\beta,\fA)=\fA$$
\end{lem}

\begin{proof}
While $(J,\lambda)$ is of level zero, we only need to prove $\chi\circ\mathrm{det}$ is a simple character on $\rU^1(\mathfrak{A})$. While $\beta\in F$, we only need to prove the character $\theta=\bar{\psi_{\beta}}\otimes\chi\circ\det$ is a simple character on $\rU^1(\mathfrak{A})$. This is directly induced by the results in the appendix in \cite{BuKuII} for the complex case, because the definition of simple stratum in the case of characteristic $\ell$ is the same as the case of characteristic $0$. And the definition 2.2.2 of \cite{MS} implies that simple $k$-characters are reduction modulo $\ell$ of simple $\bC$-characters.
\end{proof}

\begin{lem}
\label{lem 2}
Let $[\mathfrak{A},n,0,\beta]$ be a simple stratum in $\mathrm{A}$ with $\beta\notin \mathrm{F}$, $n\geq1$. Let $c\in\mathrm{F}$, and $n_0=-\mathnormal{v}_{\fA}(c)$, $n_1=-\mathnormal{v}_{\fA}(\beta+c)$, \begin{enumerate}
\item The stratum $[\mathfrak{A},n,0,\beta+c]$ is a simple stratum in $\rA$, and we have $\mathfrak{H}(\beta+c, \fA)=\mathfrak{H}(\beta,\fA)$ and $\mathfrak{J}(\beta+c,\fA)=\mathfrak{J}(\beta,\fA)$.
\item Let $\chi_{\bC}$ be a $\bC$-quasicharacter of $\mathrm{F}^{\times}$ such that $\chi_{\bC}\circ\mathrm{det}$ agrees with $\psi_c$ on $\mathrm{U}^{[\frac{1}{n_0}]+1}(\fA)$. Then we have an equivalence of simple $\bC$-characters:
$$\mathnormal{C}_{\bC}(\fA,0,\beta+c)=\mathnormal{C}_{\bC}(\fA,0,\beta)\otimes\chi_{\bC}\circ\mathrm{det}.$$
\item Let $\chi_\ell$ be any $k$-quasicharacter of $F^{\times}$ such that $\chi_\ell\circ\det$ agrees with $\bar{\psi}_c$ on $\mathrm{U}^{[\frac{1}{n_0}]+1}(\fA)$. Then we have an equivalence of simple $k$-characters:
$$\mathnormal{C}_\ell(\fA,0,\beta+c)=\mathnormal{C}_\ell(\fA,0,\beta)\otimes\chi_\ell\circ\mathrm{det}.$$
\end{enumerate}
\end{lem}

\begin{proof}
The first two assertions are the lemma in appendix of \cite{BuKuII}, so we only need to proof the last assertion. Recall that we fixed a continuous additive character $\psi_{F}$ from $F$ to $\mathbb{C}^{\times}$. Lemma \ref{lem 3} implies that every simple $\bC$-character in $\mathnormal{C}_{\bC}(\fA,0,\beta+c)$ is $\ell$-integral and has a reduction modulo $\ell$. According to the definition 2.2.2 of \cite{MS}, the reduction modulo $\ell$ gives a bijection between simple $\bC$-characters $\mathnormal{C}_{\bC}(\fA,0,\beta+c)$ to $\mathnormal{C}_\ell(\fA,0,\beta+c)$. Notice that this bijection is dependent with the choice of $\psi_F$. Apparently,
$$\mathnormal{C}_\ell(\fA,0,\beta+c)=\overline{\mathnormal{C}_{\bC}(\fA,0,\beta)}\otimes\bar{\chi}_{\bC}\circ\det,$$
where $\bar{\chi}_{\bC}$ denote the reduction modulo $\ell$ of $\chi_{\bC}$, and $\overline{\mathnormal{C}_{\bC}(\fA,0,\beta)}$ denote the set of $k$-characters, which are reduction modulo $\ell$ of characters in $\mathnormal{C}_{\bC}(\fA,0,\beta)$. By definition $\overline{\mathnormal{C}_{\bC}(\fA,0,\beta)}=\mathnormal{C}_{l}(\fA,0,\beta)$, and hence
$$\mathnormal{C}_\ell(\fA,0,\beta+c)=\mathnormal{C}_\ell(\fA,0,\beta)\otimes\bar{\chi}_{\bC}\circ\mathrm{det}.$$
Applying Corollary \ref{cor 4} to $\chi_\ell$, there exists a $\bar{\Q}_\ell$-quasicharacter $\tau_{\bC}$ of $F^{\times}$, such that $\chi_\ell\circ\det$ is isomorphic to the reduction modulo $\ell$ of $\tau_{\bC}\circ\det$. Notice that simple characters in $\mathnormal{C}_{\bC}(\fA,0,\beta+c)$ are defined on $\mathrm{H}^1=\mathfrak{H}(\beta,\fA)\cap\rU^1(\fA)$, which is a pro-$p$-subgroup of $\rG$. The reduction modulo $\ell$ of $\tau_{\bC}\circ\det$ is isomorphic to $\bar{\psi}_c$ on $\mathrm{H}^1(\beta)\cap\mathrm{U}^{[\frac{1}{2}n_0]+1}(\fA)$, which implies that $\tau_{\bC}\circ\det$ is isomorphic to $\psi_c$ on $\mathrm{H}^1(\beta)\cap\mathrm{U}^{[\frac{1}{2}n_0]+1}(\fA)$. The assertion (1) and (2) tell that
$$\mathnormal{C}_{\bC}(\fA,0,\beta+c)=\mathnormal{C}_{\bC}(\fA,0,\beta+c)\otimes\chi_{\bC}^{-1}\otimes\tau_{\bC}\circ\det.$$
We deduce directly that
$$\mathnormal{C}_{l}(\fA,0,\beta+c)=\overline{\mathnormal{C}_{\bC}(\fA,0,\beta+c)}=\mathnormal{C}_{l}(\fA,0,\beta+c)\otimes\chi_\ell\circ\det,$$
as required.
\end{proof}

\begin{cor}
\label{cor 3}
Let $(J,\lambda)$ be a maximal simple cuspidal $k$-type of $\rG$, and $\chi$ a $k$-quasicharacter of $F^{\times}$. Then the $k$-type $(J,\lambda\otimes\chi\circ\det)$ is also a maximal simple cuspidal $k$-type.
\end{cor}

\begin{proof}
Let $(J,\lambda_{\bC})$ be an $\ell$-integral maximal simple cuspidal $\bC$-type of $\rG$, whose reduction modulo $\ell$ is isomorphic to $(J,\lambda)$. Let $\chi_{\bC}$ be a $\bC$-quasicharacter of $F^{\times}$ whose reduction modulo $\ell$ is isomorphic to $\chi$ (by \ref{cor 4}). Then by the appendix of \cite{BuKuII}, the $\ell$-integral type $(J,\lambda_{\bC}\otimes\chi_{\bC}\circ\det)$ is also maximal cuspidal simple. Thus its reduction modulo $\ell$ is maximal simple cuspidal $k$-type by Proposition \ref{prop 2'}. Let $c\in\mathrm{F}$ be the element corresponding to a $\bC$-lifting of $\chi$, and $\beta$ corresponding to a simple character $\theta$ (this is well-defined, because $\mathrm{H}^1(\beta)$ is pro-$p$) contained in $(J,\lambda_{\bC})$ (as in lemma \ref{lem 2'} or \ref{lem 2}), then the two lemmas above imply that the reduction modulo $\ell$ of $\theta_{\bC}\otimes\chi\circ\det$ is a simple character contained in $(J,\lambda\otimes\chi\circ\det)$. And $\mathrm{H}^1(\beta+c)=\mathrm{H}^1(\beta)$, where $\mathrm{H}^1=\mathfrak{H}(\beta+c, \fA)\cap \rU^1(\fA)$.
\end{proof}

\begin{rem}
\label{rem 33}
Let $(J_{\rM},\lambda_{\rM})$ be a maximal simple cuspidal $k$-type of $\rM$, where $\rM$ is a Levi subgroup of $\rG$. Then $\lambda_{\rM}\cong\lambda_1\otimes\cdots\otimes\lambda_r$ for some $r\in\mathbb{N}^{\ast}$, where $(J_i,\lambda_i)$ are maximal simple cuspidal $k$-types of $\mathrm{GL}_{n_i}(F)$. Hence for any $k$-quasicharacter $\chi$ of $F^{\times}$, then new $k$-type $(J_{\rM},\lambda_{\rM}\otimes\chi\circ\det)$ is also maximal simple cuspidal of $\rM$.
\end{rem}

\subsubsection{Intertwining and weakly intertwining}
In this section, for any closed subgroup $H$ of $\rG$, we always use $H'$ to denote its intersection with $\rG'$. Let $\rM$ denote any Levi subgroup of $\rG$.

\begin{prop}
\label{prop 0.1}
\label{Lprop 0.1}
Let $K$ be a compact subgroup of $\rM$, and $\rho$ be an irreducible $k$-representation of $K$. The restriction $\res_{K'}^K\rho$ is semisimple. 
\end{prop}

\begin{proof}
Let $O$ denote the kernel of $\rho$, which is a normal subgroup of $K$. The subgroup $O\cdot K'$ is also a compact open normal subgroup of $K$, hence with finite index in $K$. We deduce that the restriction $\res_{O\cdot K'}^{K}\rho$ is semisimple by Clifford theory, furthermore the restriction $\res_{K'}^K \rho$ is semisimple.
\end{proof}

\begin{prop}
\label{prop 0.2}
\label{Lprop 0.2}
Let $K$ be a compact open subgroup of $\rM$, $\rho$ an irreducible smooth representation of $K$, and $\rho'$ an irreducible component of the restriction $\res_{K'}^K\rho$. Let $\bar{\rho}$ be an irreducible representation of $K$ such that $\res_{K'}^{K}\bar{\rho}$ also contains $\rho'$. Then there exists a $k$-quasicharacter $\chi$ of $F^{\times}$ such that $\rho\cong\bar{\rho}\otimes\chi\circ \mathrm{det}$.
\end{prop}

\begin{proof}

Let $U$ be any pro-$p$ normal subgroup of $K$ contained in the kernel of $\rho$, hence with finite index. Let's consider $\Ind_{U'}^{U}(1)$, which is semisimple, thus by lemma of Schur it is a direct sum of characters in the form of $\chi\circ\mathrm{det}\vert_U$. Since $\chi$ can be extended to a quasicharacter of $F^{\times}$, and we note the extended quasicharacter as $\chi$ as well, then we write $\chi\circ\mathrm{det}\vert_U$ as $(\chi\circ\mathrm{det})\vert_U$. The fact that $\res_{U'}^{K}\rho$ contains the trivial character induces the same property for $\res_{U'}^K\bar{\rho}$. By Frobenius reciprocity, we know that $\res_{U}^K\bar{\rho}$ contains a character in the form of $(\chi\circ\mathrm{det})\vert_U$, and the irreducibility implies that it is in fact a multiple of this character. We can hence assume that $\bar{\rho}$ is trivial on $U$.

By the Clifford theory, the restriction of $\rho$ (resp. $\bar{\rho}$) to $K'U$ is semisimple. Hence $\Hom_{K'U}(\rho,\bar{\rho})\neq0$. Applying Frobenius reciprocity, we see that $\rho$ is a subrepresentation of $\ind_{K'U}^{K}\res_{K'U}^K\bar{\rho}$, which is equivalent to $\bar{\rho}\otimes\ind_{K'U}^{K}(1)$ by 5.2 d, chapitre I, \cite{V1}. In fact, the Jordan-H$\ddot{\mathrm{o}}$lder factors are $\bar{\rho}\otimes\chi\circ\mathrm{det}$: Any irreducible factor $\tau$ of $\ind_{K'U}^{K}(1)$ can be view an irreducible $k$-representation of the quotient group $K\slash K'U$, which is isomorphic to a subgroup of the finite abelian group $\mathcal{O}_{F}^{\times}\slash\det{U}$, where $\mathcal{O}_{F}$ indicates the ring of integers of $F$. Hence $\tau$ must be a $k$-character of $K\slash K'U$ and can be extended to $F^{\times}$, since we can first extend $\tau$ to $\mathcal{O}_F^{\times}\slash\det{U}$ (hence to $\mathcal{O}_F^{\times}$) and $F^{\times}\cong\mathcal{O}_F^{\times}\times\mathbb{Z}$. We denote this extension by $\tilde{\tau}$. It is clear that $\tau\cong\tilde{\tau}\circ\det\vert_{K}$.

\end{proof}

\begin{defn}
\label{prepdefn 01}
Let $\rG$ be a locally profinite group, and $R$ be an algebraically closed field. Let $K_i$ be an open compact subgroup of $\rG$ for $i=1,2$, and $\rho_i$ a $R$-representation of $K_i$. Define $i_{K_1,K_2}x(\rho_2)$ to be the induced $R$-representation 
$$\ind_{K_1\cap x(K_2)}^{K_1}\res_{K_1\cap x(K_2)}^{x(K_2)}x(\rho_2),$$
where $x(\rho_2)$ is the conjugation of $\rho_2$ by $x$. 
\begin{itemize}
\item We say an element $x\in\rG$ weakly intertwines $\rho_1$ with $\rho_2$, if $\rho_1$ is an irreducible subquotient of $i_{K_1,K_2}x(\rho_2)$. And $\rho_1$ is weakly intertwined with $\rho_2$ in $\rG$, if $\rho_1$ is isomorphic to a subquotient of $\ind_{K_2}^{\rG}\rho_2$. We denote $\mathrm{I}_{\rG}^{w}(\rho_1,\rho_2)$ the set of elements in $\rG$, which weakly intertwines $\rho_1$ with $\rho_2$. When $\rho_1=\rho_2$, we abbreviate $\mathrm{I}_{\rG}^{w}(\rho_1,\rho_2)$ as $\mathrm{I}_{\rG}^{w}(\rho_1)$.
\item We say the element $x\in\rG$ intertwines $\rho_1$ with $\rho_2$, if the $\Hom$ set
$$\Hom_{kK_1}(\rho_1,i_{K_1,K_2}x(\rho_2)) \neq 0.$$
Representation $\rho_1$ is intertwined with $\rho_2$ in $\rG$, if the $\Hom$ set
$$\Hom_{k\rG}(\ind_{K_1}^{\rG}\rho_1,\ind_{K_2}^{\rG}\rho_2)\neq 0.$$
We denote $\mathrm{I}_{\rG}(\rho_1,\rho_2)$ the set of elements in $\rG$, which intertwine $\rho_1$ with $\rho_2$. When $\rho_1=\rho_2$, we abbreviate $\mathrm{I}_{\rG}(\rho_1,\rho_2)$ as $\mathrm{I}_{\rG}(\rho_1)$.
\end{itemize}
When $\rho_1$ is irreducible, we deduce directly from Mackey's decomposition formulae that $\rho_1$ is (weakly) intertwined with $\rho_2$ in $\rG$ if and only if there exists an element $x\in\rG$, such that $x$ (weakly) intertwines $\rho_1$ with $\rho_2$.
\end{defn}

\begin{prop}
\label{prop 0.3}
\label{Lprop 0.3}
For $i= 1,2$, let $K_i$ be a compact open subgroup of $\rM$, and $\rho_i$ an irreducible representation of $K_i$ and $\rho_i'$ be an irreducible component of $\res_{K_i'}^{K_i}\rho_i$. Let $x\in\rM$ that weakly intertwines $\rho_1'$ with $\rho_2'$. Then there exists a $k$-quasicharacter $\chi$ of $F^{\times}$ such that $x$ weakly intertwines $\rho_1$ with $\rho_2\otimes\chi\circ\mathrm{det}$.
\end{prop}

\begin{proof}
By Mackey's decomposition formula, $i_{K_1',K_2'}x(\rho_2')$ is a subrepresentation of $i_{K_1,K_2}x(\rho_2)$. Since $i_{K_1,K_2}x(\rho_2)$ has finite length, the uniqueness of Jordan-H\"older factors implies that there exists an irreducible subquotient of $i_{K_1,K_2}x(\rho_2)$, whose restriction to $K_1'$ contains $\rho_1'$ as a direct components. By \ref{prop 0.2}, this irreducible subquotient is isomorphic to $\rho_1\otimes\chi\circ\mathrm{det}$, where $\chi$ is a quasicharacter. By definition, this means $\rho_1$ is weakly intertwined with $\rho_2\otimes\chi^{-1}\circ\mathrm{det}$ by $x$.
\end{proof}

Now we begin to consider the maximal simple cuspidal $k$-types of $\rG=\mathrm{GL}_n(F)$.
\begin{prop}
\label{prop 1}
Let $(J,\lambda)$ be a maximal simple cuspidal $k$-type of $\rG$, and $\chi$ a $k$-quasicharacter of $F^{\times}$. If $(J,\lambda\otimes\chi\circ\det)$ is weakly intertwined with $(J,\lambda)$, then they are intertwined. And there exists an element $x\in\rU(\fA)$ such that $x(J)=J$ and $x(\lambda)\cong\lambda\otimes\chi\circ\det$.
\end{prop}

\begin{proof}
There is a surjection from $\res_{H^1}^J\lambda$ to $\theta_1$. By Frobenius reciprocity, there is an injection from $\lambda$ to $\ind_{H^1}^J\theta_1$, and exactness of the functors ensure that there exists an injection: $\res_{H^1}^G\ind_J^G\lambda\hookrightarrow \res_{H^1}^G\ind_{H^1}^G\theta_1$. Whence, by hypothesis, $\res_{H^1}^J\lambda\otimes\chi\circ\mathrm{det}$ is a subquotient of $\res_{H^1}^G\ind_{H^1}^G\theta_1$. After Corollary \ref{cor 3}, the groups $\mathrm{H}^1(\beta+c)=\mathrm{H}^1(\beta)$. Hence $\res_{\mathrm{H}^1}^{\rG}(\lambda\otimes\chi\circ\mathrm{det})$ is a multiple of $\theta_2$, from which we deduce that $\theta_2$ is a subquotient of $\res_{\mathrm{H}^1}^{\rG}\ind_{\mathrm{H}^1}^{\rG}\theta_1$.

Notice that $\mathrm{H}^1$ is a prop-$p$ group, and any smooth representation of $\mathrm{H}^1$ is semisimple. It follows that $\theta_2$ is a sub-representation of $\res_{\mathrm{H}^1}^{\rG}\ind_{\mathrm{H}^1}^{\rG}\theta_1$, which is equivalent to say that $\theta_2$ is intertwined with $\theta_1$ in $\rG$. Let $i=1,2$ and $\theta_{i_{\bC}}$ be $\bC$-simple characters whose reduction modulo $\ell$ is isomorphic to $\theta_i$, then $\theta_{1_{\bC}}$ is intertwined with $\theta_{2_{\bC}}$ in $\rG$ cause $\mathrm{H}^1$ is pro-$p$. It follows that the nonsplit fundamental strata $[\fA,n,n-1,\beta]$ and the nonsplit fundamental strata $[\fA,m,m-1,\beta+c]$ are intertwined. We deduce that $n=m$ by 2.3.4 and 2.6.3 of \cite{BuKu}. Then we apply Theorem 3.5.11 of \cite{BuKu}: there exists $x\in\mathrm{U}(\fA)$ such that $x(\mathrm{H}^1)=\rH^1$, $\nC(\fA,0,\beta)=\nC(\fA,0,x(\beta+c))$ and $x(\theta_{2_{\bC}})=\theta_{1_{\bC}}$, hence $x(\theta_2)=\theta_1$. In particular, $x(J)$ is a subset of $\mathcal{I}_{\mathrm{U}(\fA)}(\theta_1)$. Meanwhile, the 2.3.3 of \cite{MS} and 3.1.15 of \cite{BuKu} implies that $\mathcal{I}_{\rG}(\theta_1)\cap\mathrm{U}(\fA)=J$, then $x(J)=J$. Proposition 2.2 of \cite{MS} shows the uniqueness of $\eta_1$, hence $x(\eta_2)\cong\eta_1$. From  \cite[Corollary 8.4]{V3} we know that the $\eta_1$-isotypic part of $\res_J^{\rG}\ind_J^{\rG}(\lambda)$ can be viewed as a representation of $J$, which is a direct factor of $\res_J^{\rG}\ind_J^{\rG}(\lambda)$ and is multiple of $\lambda$ when $(J,\lambda)$ is maximal simple. Since $x(\lambda\otimes\chi\circ\det)$ could only be a subquotient of the $\eta_1$-isotypic part of $\res_J^{\rG}\ind_J^{\rG}(\lambda)$ and $\ind_J^{\rG}\lambda\cong\ind_J^{\rG}x(\lambda)$, we deduce from above that $\Hom_{kJ}(\lambda\otimes\chi\circ\det, \res_J^{\rG}\ind_J^{\rG}\lambda)\neq0$.
\end{proof}

\begin{cor}
\label{cor 5}
For any $g\in\rG$, if $g$ weakly intertwines $(J,\lambda\otimes\chi\circ\det)$ and $(J,\lambda)$, then $g$ intertwines $(J,\lambda\otimes\chi\circ\det)$ and $(J,\lambda)$.
\end{cor}

\begin{proof}
By Mackey's decomposition formula $i_{J,g(J)}g(\lambda)$ is a direct factor of $\res_J^{\rG}\ind_J^{\rG}(\lambda)$. On the other hand, we notice that the $\ind_J^{\rG}\lambda$ is isomorphic to $x(\ind_J^{\rG}\lambda)$ as $\rG$-representation, so they are equivalent after restricting to $J$. Hence the $x(\eta_1)$-isotypic part $(\res_J^{\rG}\ind_J^{\rG}\lambda)^{x(\eta_1)}$ is isomorphic to $x(\res_J^{\rG}\ind_J^{\rG}\lambda)^{x(\eta_1)}$ as $J$ representation. The later one is isomorphic to $x((\res_J^{\rG}\ind_J^{\rG}\lambda)^{\eta_1})$, which is a multiple of $x(\lambda)$. In the proof of \ref{prop 1}, there exists $x\in\mathrm{U}(\fA)$ such that $x(\eta_1)\cong\eta_2$, $x(\lambda)\cong\lambda\otimes\chi\circ\det$. By hypothesis $\lambda\otimes\chi\circ\det^{x(\eta_1)}$ is a subquotient of $i_{J,g(J)}g(\lambda)$, hence a subquotient of $(i_{J,g(J)}g(\lambda))^{x(\eta)}$. And $(i_{J,g(J)}g(\lambda))^{x(\eta)}$ is a sub-representation of $(\res_J^{\rG}\ind_J^{\rG}\lambda)^{x(\eta)}$, whence a multiple of $x(\lambda)$ as well. So $\lambda\otimes\chi\circ\det^{x(\eta_1)}$ is a sub-representation of $(i_{J,g(J)}g(\lambda))^{x(\eta)}$. We finish the proof.
\end{proof}

\subsubsection{Decomposition of $\res_{J}^{\rG}\ind_{J}^{\rG}\lambda$}
In this section, we need to do some computation to obtain the decomposition in \ref{thm 7}, which plays a key role in the proof of Proposition \ref{prop 14}. And this consists half of the proof of Theorem \ref{thm 16}.

\begin{thm}
\label{thm 7}
Let $(J,\lambda)$ be a maximal simple cuspidal $k$-type of $\rG$. There exists an integer $m$ and a decomposition:
$$\res_J^{\rG}\ind_J^{\rG}\lambda\cong(\oplus_{i=1}^{m}\ x_i(\Lambda(\lambda)))\oplus W$$
where $x_i\in\mathrm{U}(\fA)$, and $x_1=1$. The representation $\Lambda(\lambda)$ is semisimple, and a multiple of $\lambda$. For each $x_i$, the representation $x_i(\Lambda(\lambda))$ is the $x_i$-conjugation of $\Lambda(\lambda)$. The elements $x_i$'s satisfy that $x_i(\eta)\ncong x_j(\eta)$ if $i\neq j$ (see Definition \ref{defn 1122} for $\eta$), and let $\lambda'$ be any irreducible sub-representation of $\res_{J'}^{J}\lambda$, then $\lambda'$ is not equivalent to any irreducible subquotient of $\res_{J'}^{J}W$.
\end{thm}

\begin{rem}
From now on, let $(J,\lambda)$ be any maximal simple cuspidal $k$-type of $\rG$. We always use $\Lambda_{\lambda}$ to denote $\oplus_{i=1}^{m}\ x_i(\Lambda(\lambda))$, where $\Lambda(\lambda)$ has been defined in Theorem \ref{thm 7}, and we could write the decomposition in Theorem \ref{thm 7} as:
$$\res_J^{\rG}\ind_J^{\rG}\lambda\cong\Lambda_{\lambda}\oplus W.$$
\end{rem}

To prove Theorem \ref{thm 7}, we need the following lemmas:
\begin{lem}
\label{lem 8}
Let $K_1,K_2$ be two compact open subgroups of $\rG$ such that $K_1\subset K_2$. Then the compact induction $\ind_{K_1}^{K_2}$ respect infinite direct sum.
\end{lem}

\begin{proof}
Let $I$ be an index set, and $(V_i,\pi_i)$ be $k$-representations of $K_1$. Define $\pi=\oplus_{i\in I}\pi_i$. By definition of compact induction, the representation space of $\ind_{K_1}^{K_2}\pi$ are the smooth vectors of the $k$-vector space consisting with function $f:K_2\rightarrow V$ such that $f(hg)=\pi (h)f(g)$, where $h\in K_1$, $g\in K_2$, and $K_2$ acts as right transition. Notice first that every function satisfied the condition above is smooth. In fact, the quotient group $K_1/ K_2$ is finite, of which let $a_1,\ldots,a_m$ be a set of representatives in $K_2$. Then there is an bijection from the vector space, consisting of the functions on $K_2$ verified the condition above, to $V^m$, which is sending $f$ to $f(a_1),\ldots,f(a_m)$. Now let $f$ be any such function on $K_2$. For any $j\in \{1,\ldots,m \}$, there exists an open subgroup $H_j\subset K_1$ which stabilizes $v_j$. Let $g\in K_2$, the value $(a_j^{-1}(g) (f))(a_j)=f(a_j)$. Hence the open compact subgroup $H=\cap_{j=1}^m a_j^{-1}(H_j)$ stablizes $f$, so $f$ is smooth. Notice that $\oplus_{j=1}^m(\oplus_{i\in I} V_i)\cong \oplus_{i\in I}(\oplus_{j=1}^m V_i)$ as vector spaces, which the result follows.
\end{proof}

\begin{lem}
\label{lem 9}
Let $K$ be a compact open subgroup of $\rM$, where $\rM$ is a Levi subgroup of $\rG$, and $K'=K\cap\rG'$. Let $\pi$ be a $k$-representation of $K$. If $\tau'$ is an irreducible subquotient of the restricted representation $\res_{K'}^K\pi$, then there exists an irreducible subquotient $\tau$ of $\pi$, such that $\tau'$ is an irreducible direct component of $\res_{K'}^{K}\tau$.
\end{lem}

\begin{proof}
Let $H$ be a pro-$p$ open compact subgroup of $K$. The representation $\res_H^K \pi$ is semisimple, which can be written as $\oplus_{i\in I}\pi_i$, where $I$ is an index set. There is an injection from $\pi$ to $\ind_H^K\res_H^K \pi$, and the lemma \ref{lem 8} implies that $\ind_H^K\res_H^K \pi\cong\oplus_{i\in I}\ind_H^K \pi_i$. Notice that for each $i\in I$, the representation $\ind_H^K\pi_i$ has finite length. Let $W'$, $V'$ be two sub-representations of $\pi'=\res_{K'}^{K}\pi$, such that $\tau'\cong W'/V'$. When $\tau'$ is non-trivial, there exists $x\in W'$ such that $x\notin V'$. Since $\ind_H^K\res_H^K \pi$ is isomorphic to a direct sum of $\ind_H^K \pi_i$, there exists a finite index set $\{ i_1,\ldots,i_m\} \subset I$, where $m\in\mathbb{N}^{\ast}$, such that $x\in \oplus_{i_1,\ldots,i_m} \ind_H^K\pi_i$. We have:
$$0\neq(W'\cap \oplus_{i_1,\ldots,i_m} \ind_H^K\pi_i) / (V' \cap \oplus_{i_1,\ldots,i_m} \ind_H^K\pi_i )\hookrightarrow W'\slash V',$$
Since $W'\slash V'$ is irreducible, the injection above is an isomorphism, and we conclude that
$$(W'\cap \oplus_{i_1,\ldots,i_m} \ind_H^K\pi_i) / (V' \cap \oplus_{i_1,\ldots,i_m} \ind_H^K\pi_i )\cong W'/V'\cong \tau'.$$
Since the restricted representation $\res_{K'}^K \oplus_{i_1,\ldots,i_m} \ind_H^K\pi_i$ has finite dimension hence finite length, by the uniqueness of Jordan-H\"older factors, there exists an irreducible subquotient of $\oplus_{i_1,\ldots,i_m} \ind_H^K\pi_i$, whose restriction to $K'$ is semisimple (by Proposition \ref{prop 0.1}) and containing $\tau'$ as a subrepresentation. 
\end{proof}

Now we look back Theorem \ref{thm 7}.
\begin{proof} of \ref{thm 7}:

By \cite[Corollary 8.4]{V3}, we can decompose $\res_J^{\rG}\ind_J^{\rG}\lambda\cong \Lambda(\lambda)\oplus W_1$, where any irreducible subquotient of $W_1$ is not isomorphic to $\lambda$. Let $\lambda'$ be an irreducible subrepresentation of the semisimple $k$-representation $\res_{J'}^{J}\lambda$. If $\lambda'$ is an irreducible subquotient of $\res_{J'}^{J}W_1$, by Lemma \ref{lem 9} and Propositon \ref{prop 0.2}, there exists a $k$-quasicharacter $\chi$ of $F^{\times}$ such that $\lambda\otimes\chi\circ\det$ is an irreducible subquotient of $W_1$. This follows that $\lambda\otimes\chi\circ\det$ is weakly intertwined with $\lambda$. By Proposition \ref{prop 1}, they are intertwined and there exists $x\in\mathrm{U}(\fA)$ such that $\lambda\otimes\chi\circ\det\cong x(\lambda)$. The fact that $\lambda\otimes\chi\circ\det$ is a subquotient of $W_1$ implies that  $x(\eta)\ncong\eta$. As in the proof of Corollary \ref{cor 5}, we have:
$$(\Lambda(\lambda))^{x(\eta)}\oplus W_1^{x(\eta)}\cong(\res_J^{\rG}\ind_J^{\rG}\lambda)^{x(\eta)}\cong x((\res_J^{\rG}\ind_J^{\rG}\lambda)^{\eta}),$$
and the later one is isomorphic to $x(\Lambda(\lambda))$, which is a direct sum of $x(\lambda)$. Since $x(\eta)\ncong\eta$, we have $(\Lambda(\lambda))^{x(\eta)}=0$. As for $W_1$, thus we can decompose $W_1$ as $W_1^{x(\eta)}\oplus W_2$. Hence $W_1^{x(\eta)}\cong x(\Lambda(\lambda))$. Now we obtain an isomorphism:
$$\res_J^{\rG}\ind_J^{\rG}\lambda\cong\Lambda(\lambda)\oplus x(\Lambda(\lambda))\oplus W_2,$$
where $W_2^{x(\eta)}=0$ and $W_2^{\eta}=0$. This implies any irreducible subquotient of $W_2$ is not isomorphic to $\lambda$ neither $x(\lambda)$. If $\lambda'$ is an irreducible subquotient of $W_2$, we can repeat the steps above, then find a $k$-quasicharacter $\chi_2$, an element $x_2\in\mathrm{U}(\fA)$, and decompose $W_2$ as $x_2(\Lambda(\lambda))\oplus W_3$, where any $W_3^{x_2(\eta)}=0$. Furthermore, any irreducible representation of $J$, whose restriction to $J'$ contains $\lambda'$ as a subrepresentation, is $\mathrm{U}(\fA)$-conjugate to $\lambda$. The quotient group $\rU(\fA)/ J$ is finite, hence the set of irreducible representations $\{ x(\lambda) \}_{x\in\rU(\fA)}$ is finite, which means after repeat the steps above for finite times, we could obtain the decomposition as required.
\end{proof}

\subsubsection{Projective normalizer $\tilde{J}$ and its subgroups}

Now we will recall one definition and some propositions given by Bushnell and Kutzko in \cite{BuKuII} when they consider the $\obQ_{\ell}$-representations of $\rG'$.
\begin{defn}[Bushnell,Kutzko]
\label{defn 11}
We define the projective normalizer $\tilde{J}=\tilde{J}(\lambda)$ of $(J,\lambda)$. Let $\fA$ be the principal order attached to $(J,\lambda)$. Then define $\tilde{J}$ to be the group of all $x\in \mathrm{U}(\fA)$ such that:
\begin{itemize}
\item $xJx^{-1}=J$, and
\item there exists a $k$-quasicharacter $\chi$ of $F^{\ast}$ such that $x(\lambda)\cong\lambda\otimes\chi\circ\det$.
\end{itemize}
\end{defn}

\begin{prop}
\label{prop 12}
Let $(J,\lambda)$ be a simple type in $\rG$ as in definition \ref{defn 11}, and $\chi$ be a $k$-quasicharacter of $F^{\ast}$. The following are equivalent:
\begin{enumerate}
\item $\lambda\cong\lambda\otimes\chi\circ\det,$
\item $\chi\circ\det\vert_{J^1}$ is trivial and $\sigma\otimes\chi\circ\det\vert_{\mathrm{U}(\mathfrak{B})}\cong\sigma$,
\item $\chi\circ\det\vert_{J^1}$ is trivial, and $\lambda$, $\lambda\otimes\chi\circ\det$ are intertwined in $\rG$.
\end{enumerate}
\end{prop}

\begin{proof}
The proof in Proposition 2.3 \cite{BuKuII} still works in our case, and we write it here to ensure it in the modulo $\ell$ case. We prove this proposition in the order of $(2)\rightarrow (1)\rightarrow (3) \rightarrow (2)$. 

Since $J/ J'\cong \rU(\mathfrak{B}/\rU^1(\mathfrak{B}))$, the implication $(2)\rightarrow (1)$ is trivial. Now let us assume $\lambda$ is equivalent to $\lambda\otimes\chi\circ\det$. Restricting to $H^1$, we see that the simple character $\theta\cong\theta\otimes\chi\circ\det\vert_{H^1}$, which implies $\chi\circ\det\vert_{H^1}$ is trivial. Now assume $(3)$ holds. Proposition \ref{prop 1} gives an element $x\in \mathrm{U}(\fA)$ such that $x(J)=J$ and $x(\lambda)=\lambda\otimes\chi\circ\det$. We have $\chi\circ\det$ is trivial on $J^1$. Combining this fact with the uniqueness of $\eta$ corresponding to any fixed simple character of $H^1$ (see Proposition 2.2 of \cite{MS}), we have $x(\eta)\cong\eta$. In particular, $x\in \mathcal{I}_{\rG}(\theta)=J^1 B^{\times}J^1$, by IV.1.1 in \cite{V2}, hence $x\in J^1B^{\ast}J^1\cap\rU(\fA)=J$. Whence $\lambda\cong x(\lambda)\cong\lambda\otimes\chi\circ\det$. We therefore have $\kappa\otimes\sigma\cong\kappa\otimes\sigma\otimes\chi\circ\det$, where $\kappa$ is a $\beta$-extension of $\eta$ to $J$.

As indicate in the proof of Bushnell and Kutzko, from now on, we use the technique in Proposition 5.3.2 of \cite{BuKu}: Let $X$ denote the representation space of $\kappa$ and $Y$ the representation space of $\sigma$, which can be identified with the representation space of $\sigma\otimes\chi\circ\det$. Let $\phi$ be the isomorphism between $\kappa\otimes\sigma$ and $\kappa\otimes\sigma\otimes\chi\circ\det$. We may write $\phi$ as $\sum_{j}S_j\otimes T_j$ where $S_j\in\mathrm{End}_k(X)$ and $T_j\in\mathrm{End}_k(Y)$, and where $\{ T_j \}$ are linearly independent. Let $g\in J^1$, we have $\kappa\otimes\sigma(g)\circ\phi =\phi\circ(\kappa\otimes\sigma\otimes\chi\circ\det)(g)$. Since $J^1\subset \mathrm{ker}(\sigma)= \mathrm{ker}(\sigma\otimes\chi\circ\det)$, this relation reads:
$$(\eta(g)\otimes 1)\circ \sum_{j}S_j\otimes T_j=(\sum_j S_j\otimes T_j)\circ(\eta(g)\otimes1),$$
which is equivalent to say that:
$$\sum_j(\eta(g)\circ S_j-S_j\circ\eta(g))\otimes T_j=0.$$
The linearly independence of $T_j$ implies that $S_j\in\mathrm{End}_{kJ^1}(\eta)=k^{\ast}$, by the lemma of Schur. Hence $\phi=\mathrm{1} \otimes \sum_{j}S_j\cdot T_j$. Now note $T= \sum_{j}S_j\cdot T_j$ and take $g\in J$, the morphism relation reads:
$$(\kappa(g)\otimes \sigma(g))\circ(1\otimes  T)=\kappa(g)\otimes(\sigma(g)\circ T)=\kappa(g)\otimes (T\circ\sigma\otimes\chi\circ\det(g))$$
$$=(1\otimes T)\circ (\kappa(g)\otimes\sigma\otimes\chi\circ\det(g)),$$
which says $T\in\mathrm{Hom}_{kJ}(\sigma,\sigma\otimes\chi\circ\det)\neq 0$. We finish the proof.
\end{proof}

\begin{cor} [Bushnell,Kutzko]
\label{cor 13}
Let $x\in\tilde{J}(\lambda)$, and let $\chi$ be a quasicharacter of $F^{\ast}$ such that $x(\lambda)\cong\lambda\otimes\chi\circ\det$. Then:
\begin{enumerate}
\item the map $x \mapsto \chi\circ\det\vert_{J^1}$ is an injective homomorphism $\tilde{J}/J\rightarrow (\det(J^1))^{\wedge}$. The later one denotes the dual group of the subgroup $\det(J^1)$ of $F^{\times}$;
\item $\tilde{J}/ J$ is a finite abelian $p$-group, where $p$ is the residual characteristic of $F$.
\end{enumerate}
\end{cor}

\begin{proof}
For (1). Let $x\in\tilde{J}$. Suppose there exist two $k$-quasicharacter $\chi_1,\chi_2$ of $F^{\ast}$, such that $x(\lambda)\cong \lambda\otimes\chi_1\circ\det$ and $\lambda\otimes\chi_1\circ\det\cong\lambda\otimes\chi_2\circ\det$. This is equivalent to say that
$$\lambda\cong\lambda\otimes(\chi_1\circ\det)\otimes(\chi_2^{-1}\circ\det)\cong\lambda\otimes(\chi_1\otimes\chi_2^{-1})\circ\det.$$
The equivalence between (1) and (2) of Proposition \ref{prop 12} implies that $(\chi_1\otimes \chi_2^{-1})\circ\det\vert_{J^1}$ is trivial. Hence $\chi_1\circ\det\vert_{J^1}\cong\chi_2\circ\det\vert_{J^1}$. So the map is well defined, and is clearly a morphism between groups. Now suppose that $x\in\tilde{J}$ and $\chi$ is a $k$-quasicharacter of $F^{\times}$ which is trivial on $\det(J^1)$, such that $x(\lambda)\cong\lambda\otimes\chi\circ\det$. As in the Proposition \ref{prop 12}, the equivalence of conditions means that $\lambda\cong\lambda\otimes\chi\circ\det$. Thus $x$ intertwined $\lambda$ to itself. Whence the element $x$ belongs to $J B^{\times} J\cap\rU(\fA)=J$.

For (2). Since $J^1$ is a pro-$p$ group, this is induced directly from (1).
\end{proof}

\subsubsection{Two conditions for irreducibility}
\label{section 07}
In this section, let $(J,\lambda)$ be any maximal simple cuspidal $k$-type of $\rG$. We will construct a compact subgroup $M_{\lambda}$ of $\rG'$ and a family of irreducible representations $\lambda_{M_{\lambda}}'$ of $M_{\lambda}$, such that the induced representation $\ind_{M_{\lambda}}^{\rG'}\lambda_{M_{\lambda}}'$ is irreducible and cuspidal (Theorem \ref{thm 16}). And in the next section, we will see that any irreducible cuspidal $k$-representation $\pi'$ of $\rG'$ can be constructed in this manner.

To check the irreducibility of this induced representation, we only need to calculate its intertwining set in $\rG'$, when considering representations in characteristic $0$, but this is not sufficient in the case of modulo $\ell$. As noted in lemma 4.2 in article \cite{V3}, Vign\'eras presents a criterium of irreducibility in modulo $\ell$ cases:

\begin{lem} [criterium of irreducibility by Vign\'eras]
\label{lem 19}
Let $K$ be an open compact subgroup of $\rG'$, and $\pi'$ be a $k$-irreducible representation of $K$. The induced representation $\ind_K^{\rG'}\pi'$ is irreducible, when
\begin{enumerate}
\item $\mathrm{End}_{k\rG'}(\ind_{K}^{\rG'}\pi')=k$,
\item for any $k$-irreducible representation $\nu$ of $\rG'$, if $\pi'$ is contained in $\res_{K}^{\rG'}\nu$ then there is a surjection which maps $\res_{K}^{\rG'}\nu$ to $\pi'$.
\end{enumerate}
\end{lem}

As in 8.3 chapter I of \cite{V1}, the first criterion of irreducibility is equivalent to say that the intertwining set $\mathrm{I}_{\rG'}(\pi')=K$. 

\begin{cor}
\label{cor 99}
Let $(J,\lambda)$ be a maximal simple cuspidal $k$-type in $\rG$. The induced $k$-representation $\ind_{J}^{\tilde{J}}\lambda$ is irreducible.
\end{cor}

\begin{proof}
Lemma \ref{lem 19} can be applied in this case after changing $\rG'$ to any locally pro-finite group. First, we calculate $\mathrm{End}_{k\tilde{J}}(\ind_{J}^{\tilde{J}}\lambda)$, which equals to $k$ since the intertwining group $\mathrm{I}_{\tilde{J}}(\lambda)=J$. Now we consider the second condition. Let $\nu$ be an irreducible $k$-representation of $\tilde{J}$, such that 
$$\lambda\hookrightarrow\res_{J}^{\tilde{J}}\lambda.$$
By Frobenius reciprocity and the exactness of functors $\ind$ and $\res$, we have a surjection:
$$\res_{J}^{\rG}\ind_{J}^{\rG}\lambda\rightarrow\res_{J}^{\tilde{J}}\nu.$$
The $(J^1,\kappa)$-isotypic part $\nu^{\kappa}$ of $\res_{J}^{\tilde{J}}\nu$ is a direct component as $J$ representation, and $\nu^{\kappa}$ is a quotient of the $(J^1,\kappa)$-isotypic part $\lambda^{\kappa}$ of $\lambda$ as $J$ representation. The later one is a multiple of $\lambda$ by Corollary $8.4$ of \cite{V3}. Hence the surjection required in the second condition of Lemma \ref{lem 19} exists.
\end{proof}

\begin{thm}
\label{thm 10}
Let $\lambda'$ be a subrepresentation of $\res_{J'}^J\lambda$. Then $\lambda'$ verifies the second condition of irreducibility. This is to say that for any irreducible representation $\pi'$ of $\rG'$, if there is an injection: $\lambda'\hookrightarrow \res_{J'}^{\rG'}\pi'$, then there is a surjection: $\res_{J'}^{\rG'}\pi'\twoheadrightarrow \lambda'.$
\end{thm}

\begin{proof}
Since $J$ is open, every double coset $\rG' g J$ is open and closed, hence we could apply Mackey's decomposition formula:
$$\res_{\rG'}^{\rG}\ind_J^{\rG}\lambda \cong \oplus_{a\in J\backslash \rG/\rG'} \ind_{\rG'\cap a(J)}^{\rG'}\res_{\rG'\cap a(J)}^{a(J)}a(\lambda).$$
We take $a=1$, then $\ind_{J'}^{\rG'}\res_{J'}^J\lambda$ is a direct factor of $\res_{\rG'}^{\rG}\ind_J^{\rG}\lambda$. The hypothesis $\lambda'\hookrightarrow\res_{J'}^{\rG'}\pi'$ implies a surjection from $\ind_{J'}^{\rG'}\lambda'$ to $\pi'$ by Frobenius reciprocity. Since $\res_{J'}^J\lambda$ is semisimple with finite length by Proposition \ref{prop 0.1} and the functor $\ind_{J'}^{\rG'}$ respects finite direct sum, we have an surjection:
$$\ind_{J'}^{\rG'}\res_{J'}^J\lambda\twoheadrightarrow \pi',$$
hence we obtain a surjection:
$$\res_{\rG'}^{\rG}\ind_J^{\rG}\lambda\twoheadrightarrow\pi'.$$
Now consider the surjection:
$$\iota:\res_{J'}^{\rG}\ind_{J}^{\rG}\lambda\twoheadrightarrow \res_{J'}^{\rG'}\pi'.$$
Meanwhile, by Theorem \ref{thm 7}, we could decompose $\res_J^{\rG}\ind_J^{\rG}\lambda\cong \Lambda_{\lambda}\oplus W$. We have $\Lambda_{\lambda}\oplus W/\mathrm{ker}{\iota}\cong\res_{J'}^{\rG'}\pi'$. If the image of the injection $\lambda'\hookrightarrow \Lambda_{\lambda}\oplus W/\mathrm{ker}(\iota)$ is contained in $W+\mathrm{ker}(\iota)/\mathrm{ker}(\iota)$, then $\lambda'$ is an irreducible subquotient of $W$, which is contradicted with Theorem \ref{thm 7}. Hence the image of the composed morphism: 
$$\lambda'\hookrightarrow \Lambda_{\lambda}\oplus W/\mathrm{ker}(\iota)\twoheadrightarrow \Lambda_{\lambda}\oplus W/(W+\mathrm{ker}(\iota))\cong\Lambda_{\lambda}/(\Lambda_{\lambda}\cap(W+\mathrm{ker}(\iota)))$$
is non-trivial. Since $\Lambda_{\lambda}/(\Lambda_{\lambda}\cap(W+\mathrm{ker}(\iota)))$ is a quotient of $\Lambda_{\lambda}$, and the functor $\res_{J'}^J$ maps any irreducible representation of $J$ to a semisimple representation with finite length of $J'$, the representation $\res_{J'}^J\Lambda_{\lambda}$ is semisimple with finite length of $J'$. So does the quotient $\Lambda_{\lambda}/ (\Lambda_{\lambda}+\mathrm{ker}(\iota))$, of which $\lambda'$ is an irreducible direct component. This implies a surjection: $\res_{J'}^{\rG'}\pi'\twoheadrightarrow \lambda'$.
\end{proof}

In the theorem above, we proved that $\lambda'$ verifies the second condition of irreducible criterium of irreducibility in lemma \ref{lem 19}. Unfortunately, $(J',\lambda')$ does not satisfies the first condition. This is also false for representations of characteristic $0$. A natural idea is to construct a open compact subgroup of $\rG'$, which is bigger than $J'$. In the case of characteristic $0$, Bushnell and Kutzko calculated in \cite{BuKuII}. This group is $\tilde{J}'_{\bC}=\tilde{J}_{\bC}\cap\rG'$, the intersection of projective normalizer of a $\obQ_{\ell}$-maximal cuspidal simple type and $\rG'$. We will see in Proposition \ref{prop 24} and definition \ref{defn 22} that this group is $\tilde{J'}=\tilde{J}\cap\rG'$ in the case of modulo $\ell$.

\begin{prop}
\label{prop 17}
Let $L$ be any subgroup of $\tilde{J}'=\tilde{J}\cap\rG'$ such that $J'\subset L\subset \tilde{J}'$, and $\lambda'$ an irreducible subrepresentation of $\lambda\vert_{J'}$. Then the induced representation $\ind_{J'}^{L}\lambda'$ is semisimple.
\end{prop}

\begin{proof}
By Mackey's decomposition formula, the induced representation $\ind_{J'}^{\tilde{J}'}\lambda'$ is a subrepresentation of $\res_{\tilde{J}'}^{\tilde{J}}\ind_J^{\tilde{J}}\lambda$. Applying Mackey's decomposition formula, we have
$$\res_{\tilde{J}'}^{\tilde{J}}\ind_J^{\tilde{J}}\lambda\cong\oplus_{g\in J\backslash\tilde{J}}\res_{J'}^{J}g(\lambda),$$
since $\tilde{J}$ normalises $J$ and $J'$. Hence $\res_{\tilde{J}'}^{\tilde{J}}\ind_J^{\tilde{J}}\lambda$ is semisimple by Proposition \ref{prop 0.1} and the fact that $\ind_J^{\tilde{J}}\lambda$ is irreducible. Since $L$ is a normal open subgroup of $\tilde{J'}$, the index of $L$ in $\tilde{J}'$ is finite. Hence the restricted representation $\res_{L}^{\tilde{J'}}\ind_{J'}^{\tilde{J}'}\lambda'$ is semisimple by Clifford theory, of which $\ind_{J'}^{L}\lambda'$ is a subrepresentation. Now we obtain the result.
\end{proof}

\begin{prop}
\label{prop 14}
Let $\lambda'$ be an irreducible subrepresentation of $\res_{J'}^J\lambda$, and $\lambda_L'$ an irreducible subrepresentation of $\ind_{J'}^{L}\lambda'$. Then $\lambda_L'$ verifies the second condition of irreducibility. This is to say that for any irreducible representation $\pi'$ of $\rG'$, if there is an injection $\lambda_L' \hookrightarrow \res_{L'}^{\rG'}\pi'$, then there exists a surjection $\res_{L'}^{\rG'}\pi'\twoheadrightarrow \lambda_{L}'$, where $L'=L\cap\rG'$.
\end{prop}

\begin{proof}
We have proved in Proposition \ref{prop 17} that $\ind_{J'}^L\lambda'$ is semisimple. Hence the injection from $\lambda_L'$ to $\res_L^{\rG'}\pi'$ induces a non-trivial homomorphism $\ind_{J'}^L \lambda'\twoheadrightarrow \res_L^{\rG'}\pi'$. By Frobenius reciprocity, we obtain an injection from $\lambda'$ to $\res_{J'}^{\rG'}\pi'$. Thus there exists a non-trivial homomorphism $\res_{J'}^{J}\lambda\rightarrow\res_{J'}^{\rG'}\pi'$. After applying Frobenius reciprocity and the exactness of the functor $\res_{J'}^{\rG'}$, we obtain a surjection:
$$\res_{J'}^{\rG'}\ind_{J'}^{\rG'}\res_{J'}^{J}\lambda\twoheadrightarrow\res_{J'}^{\rG'}\pi'.$$
By Mackey's decomposition formula, the $k$-representation $\ind_{J'}^{\rG'}\res_{J'}^{J}\lambda$ is a direct component of $\res_{\rG'}^{\rG}\ind_{J}^{\rG}\lambda$, and combining this fact with the exactness of functor $\res_{J'}^{\rG'}$, the $k$-representation $\res_{J'}^{\rG'}\ind_{J'}^{\rG'}\res_{J'}^{J}\lambda$ is a direct component of $\res_{J'}^{\rG}\ind_{J}^{\rG}\lambda$. Hence the surjection above implies a non-trivial homomorphism:
$$\res_{J'}^{\rG}\ind_{J}^{\rG}\lambda\twoheadrightarrow \res_{J'}^{\rG'}\pi'.$$
By Proposition \ref{thm 7}, the left hand side isomorphic to $\res_{J'}^J\Lambda_{\lambda}\oplus\res_{J'}^J W$, where any irreducible subquotient of $\res_{J'}^J W$ is not isomorphic to any irreducible subrepresentation of $\res_{J'}^J\lambda$. Now we obtain an equivalence of $\res_{J'}^{\rG'}\pi'$ with $(\res_{J'}^J\Lambda_{\lambda}\oplus\res_{J'}^J W )/ K$, where $K$ (a $k$-representation of $J'$) is the kernel of the surjection above.

We have:
$$\lambda_L'\hookrightarrow\res_{L}^{\rG'}\pi'\hookrightarrow\ind_{J'}^L\res_{J'}^{\rG'}\pi',$$
and the last factor is isomorphic to $\ind_{J'}^{L}((\res_{J'}^{J}\Lambda_{\lambda}\oplus\res_{J'}^J W)/K)$. We note this composed homomorphism from $\lambda_L'$ to $\ind_{J'}^{L}((\res_{J'}^{J}\Lambda_{\lambda}\oplus\res_{J'}^J W)/K)$ as $\tau$.

Since the functor $\ind_{J'}^L$ is exact, the right side is isomorphic to $(\ind_{J'}^L\res_{J'}^J \Lambda_{\lambda}\oplus\ind_{J'}^L\res_{J'}^J W) / \ind_{J'}^L K$. And we consider the representation $\ind_{J'}^L(\res_{J'}^J W + K)/\ind_{J'}^L K$, which is isomorphic to $\ind_{J'}^L((\res_{J'}^JW+ K)/K)$. 
We assume the image $\tau(\lambda_L')$ in $\ind_{J'}^{L}((\res_{J'}^{J}\Lambda_{\lambda}\oplus\res_{J'}^J W)/K)$ is contained in $\ind_{J'}^L((\res_{J'}^J W +K)/K)$. Then $\tau$ is a non-trivial morphism from $\lambda_{L}'$ to $\ind_{J'}^L((\res_{J'}^J W +K)/K)$. By Frobenius reciprocity, we deduce a non-trivial morphism from $\res_{J'}^L\lambda_L'$ to $(\res_{J'}^J W +K)/K$. 
Notice that
$$\res_{J'}^L \lambda_L'\hookrightarrow\res_{J'}^{\tilde{J}}\ind_J^{\tilde{J}}\lambda\cong\oplus_{a\in J\backslash\tilde{J}/ J'}\res_{J'}^J a(\lambda),$$
and by definition of $\tilde{J}$, the representation $a(\lambda)\cong\lambda\otimes\chi\circ\det$, for some $k$-quasicharacter $\chi$ of $F^{\times}$. Hence
$$\oplus_{a\in J\backslash\tilde{J}/ J'}\res_{J'}^J a(\lambda)\cong\oplus_{J\backslash\tilde{J}/ J'} \res_{J'}^J \lambda.$$
Thus there exists an irreducible direct component $\lambda''$ of $\res_{J'}^J\lambda$, from which there is an injective morphism to $(\res_{J'}^JW+K)/K\cong \res_{J'}^JW/(\res_{J'}^JW\cap K)$.
Hence $\lambda''$ is isomorphic to a subquotient of $\res_{J'}^J W$. This is contradicted to Theorem \ref{thm 7}. So the image $\tau(\lambda_L')$ is not contained in $\ind_{J'}^L((\res_{J'}^J W+K)/K)$. We deduce that the composed map:
$$\lambda_L'\hookrightarrow\res_L^{\rG'}\pi'\rightarrow\ind_{J'}^{L}(\res_{J'}^J(\Lambda_{\lambda}\oplus W)/K)/\ind_{J'}^L((\res_{J'}^J W+K)/K),$$
is non-trivial. The right hand side factor:
$$ \ind_{J'}^{L}(\res_{J'}^J(\Lambda_{\lambda}\oplus W)/K)/\ind_{J'}^L((\res_{J'}^J W+K)/K)$$
$$\cong\ind_{J'}^L(\res_{J'}^J(\Lambda_{\lambda}\oplus W) / (\res_{J'}^J W+K)).$$
Notice that $\res_{J'}^J(\Lambda_{\lambda}\oplus W)/(\res_{J'}^JW+K)$ is a quotient, hence isomorphic to a subrepresentation of $\res_{J'}^J\Lambda_{\lambda}$. The representation $\res_{J'}^J\Lambda_{\lambda}$ is semisimple, with irreducible direct components in the form of $x(\lambda')$, where $x\in\rU(\fA)$. Furthermore, since $\ind_{J'}^{L}x(\lambda')$ is a direct component of $\res_{L}^{\tilde{J}'}\ind_{J'}^{\tilde{J}'}x(\lambda)$, as in the proof of Proposition \ref{prop 17}, it is semisimple. After lemma \ref{lem 8}, we deduce that $\ind_{J'}^{L}\res_{J'}^J\Lambda_{\lambda}$ is semisimple, and so is the subrepresentation $\ind_{J'}^L(\res_{J'}^J(\Lambda_{\lambda} \oplus W)/((\res_{J'}^J W)+K))$, of which $\lambda_L'$ is a direct factor. Hence we finish the proof.
\end{proof}

\begin{defn}
\label{defn 15}
Let $(J,\lambda)$ be a maximal simple cuspidal $k$-type of $\rG$, and $\lambda'$ be any irreducible subrepresentation of $\res_{J'}^J\lambda$. Define $M_{\lambda}$ to be the subgroup of $\tilde{J}'$ consisting with all the elements $x\in\tilde{J}'$, such that $x(\lambda')\cong\lambda'$.
\end{defn}

\begin{rem}
Since $M_{\lambda}$ normalizes $J$ and the intersection $M_{\lambda}\cdot J\cap \rG'$ equals to $M_{\lambda}$, we deduce that $J$ normalizes $M_{\lambda}$. Notice that irreducible subrepresentations of $\res_{J'}^J\lambda$ are $J$-conjugate. Hence the group $M_{\lambda}$ depends only on $\lambda$.
\end{rem}

We will prove at the end of this section, that the couple $(M_{\lambda},\lambda_{M_{\lambda}}')$ verifies the two criterium of irreducibility. The first criterion has been checked in Proposition \ref{prop 14}. And we will calculate its intertwining group in $\rG'$ in two steps. First is to prove $\mathrm{I}_{\rG'}(\ind_{J'}^{\rU(\fA)'}\lambda')\subset\rU(\fA)'$ (Proposition \ref{prop 18}), and then prove that $\mathrm{I}_{\rU(\fA)'}\lambda_{M_{\lambda}}'=M_{\lambda}$ (Theorem \ref{thm 16}).

\begin{prop}
\label{prop 18}
Let $\lambda'$ be an irreducible subrepresentation of $\res_{J'}^J\lambda$, then the intertwining set $\mathrm{I}_{\rG'}(\ind_{J'}^{\rU(\fA)'}\lambda')$ is contained in $\rU(\fA)'$.
\end{prop}

\begin{proof}
Let $\tau$ denote the irreducible representation $\ind_{J}^{\rU(\fA)}\lambda$. The induced representation $\ind_{J'}^{\rU(\fA)'}\lambda'$ is a subrepresentation of $\res_{\rU(\fA)'}^{\rU(\fA)}\tau$, thus is semisimple with finite direct components. We write $\ind_{J'}^{\rU(\fA)'}\lambda'$ as $\oplus_{i\in I}\tau_i'$, where $\tau_i'$ are irreducible direct components and $I$ is a finite index set. Let $g\in\rG$, we have an equality:
$$\mathrm{I}_g(\ind_{J'}^{\rU(\fA)'}\lambda')=\bigcup_{i\in I,j\in I}\mathrm{Hom}(\tau_i',i_{\rU(\fA)',g(\rU(\fA)')}\tau_j').$$
Hence we have:
$$\mathrm{I}_{\rG'}(\ind_{J'}^{\rU(\fA)'}\lambda')= \bigcup_{i\in I,j\in I}\mathrm{I}_{\rG'}(\tau_i',\tau_j').$$ 

Now we assume that $g\in\rG'$ intertwines $\tau_i'$ with $\tau_j'$. Since $\tau_i'$ and $\tau_j'$ are direct components of $\res_{\rU(\fA)'}^{\rU(\fA)}\tau$, by Proposition \ref{prop 0.3}, there exists a $k$-quasicharacter $\chi$ of $F^{\times}$ such that $g$ weakly intertwines $\tau$ with $\tau\otimes\chi\circ\det$. This implies that $\tau$ is a subquotient of $i_{\rU(\fA),g(\rU(\fA))}\tau\otimes\chi\circ\det$, hence the restriction $\res_{J}^{\rU(\fA)}\tau$ is a subquotient of $\res_{J}^{\rU(\fA)}i_{\rU(\fA),g(\rU(\fA))}\tau\otimes\chi\circ\det$. Since $\lambda$ is a subrepresentation of $\res_{J}^{\rU(\fA)}\tau$, it is hence a subquotient of $\res_{J}^{\rU(\fA)}i_{\rU(\fA),g(\rU(\fA))}\tau\otimes\chi\circ\det$. We have
$$\res_{J}^{\rU(\fA)}i_{\rU(\fA),g(\rU(\fA))}\tau\otimes\chi\circ\det$$
$$=\res_{J}^{\rU(\fA)}\ind_{\rU(\fA)\cap g(\rU(\fA))}^{\rU(\fA)}\res_{\rU(\fA)\cap g(\rU(\fA))}^{g(\rU(\fA))}\ind_{g(J)}^{g(\rU(\fA))}g(\lambda)\otimes\chi\circ\det$$

Applying Mackey's decomposition formula two times to the later factor, we obtain that $\res_{J}^{\rU(\fA)}i_{\rU(\fA),g(\rU(\fA))}(\tau\otimes\chi\circ\det)$ is isomorphic to a finite direct sum, whose direct components are in the form of $\ind_{J\cap y(J)}^{J}\res_{J\cap y(J)}^{y(J)}y(\lambda\otimes\chi\circ\det)$, where $y\in \rU(\fA) g\rU(\fA)$. More precisely,
$$\res_{J}^{\rU(\fA)}i_{\rU(\fA),g(\rU(\fA))}(\tau\otimes\chi\circ\det)$$
$$=\bigoplus_{\beta\in \rU(\fA)\cap\alpha g(J)\backslash\rU(\fA)\slash J}\bigoplus_{\alpha\in g(J)\backslash g(\rU(\fA))\slash\rU(\fA)\cap g(\rU(\fA))}\ind_{J\cap y (J)}^{J}\res_{J\cap y (J)}^{\beta(\rU(\fA))\cap y (J)}y (\lambda)\otimes\chi\det$$
where $y=\beta\alpha g$.

After the uniqueness of Jordan-H\"older factors, the representation $\lambda$ is weakly intertwined with $\lambda\otimes\chi\circ\det$ by some $y\in \rU(\fA)g \rU(\fA)$. Hence $y$ intertwines $\lambda$ with $\lambda\otimes\chi\circ\det$ by Corollary \ref{cor 5}, and there exists $x\in\rU(\fA)$ such that $x(\lambda\otimes\chi\circ\det)\cong\lambda$ by Proposition \ref{prop 1}. The element $yx^{-1}$ intertwines $\lambda$ to itself, and hence lies in $E^{\times} J$. Therefore $g\in \rU(\fA)\mathrm{E}^{\times}J \rU(\fA)\cap\rG'$. Furthermore, we have $\rU(\fA)\mathrm{E}^{\times}J \rU(\fA)\cap\rG=\rU(\fA)'$, because $E^{\times}$ normalises $\rU(\fA)$ and for any $e\in E^{\times}$, $\det(e)\in\mathfrak{o}_F^{\times}$ if and only $e\in\mathfrak{o}_E^{\times}$, 
where $\mathfrak{o}_F$,$\mathfrak{o}_E$ denote the ring of integers of $F$,$E$ respectively. Hence for any $a\in\rU(\fA)$, $\det(ea)=1$ if and only if $ea\in \rU(\fA)\cap\rG'$. From which, we deduce that $\mathrm{I}_{\rG'}(\ind_{J'}^{\rU(\fA)'}\lambda')=\rU(\fA)'$. 
\end{proof}

\begin{lem}
\label{lem 17}
Let $\lambda'$ be an irreducible component of $\res_{J'}^{J}\lambda$, and let $x\in\rU(\fA)'$ intertwines $\lambda'$. Then $x\in\tilde{J}'$.
\end{lem}

\begin{proof}
If $x\in\rU(\fA)'$ intertwines $\lambda'$, then by Proposition \ref{prop 0.3} the element $x$ weakly intertwines $\lambda$ with $\lambda\otimes\chi\circ\det$ for some quasicharacter $\chi$ of $F^{\times}$. Then Corollary \ref{cor 5} implies that $x$ intertwines $\lambda$ with $\lambda\otimes\chi\circ\det$, and Proposition \ref{prop 1} implies that there exists an element $y\in\rU(\fA)$ such that $y(J)=J$ and $y(\lambda)\cong\lambda\otimes\chi\circ\det$. By definition of $\tilde{J}$, this element $y$ is clearly contained in $\tilde{J}$. The element $xy^{-1}$ therefore intertwines $\lambda$, 
 \S IV.$1.1$ in \cite{V2} says that $x\in E^{\times} J y\cap\rU(\fA)'$. However, $E^{\times}J\cap \rU(\fA)=J$ and $y\in\rU(\fA)$. We deduce that $x\in Jy\cap\rU(\fA)'\subset\tilde{J}'$.
\end{proof}

\begin{thm}
\label{thm 16}
Let $\lambda_{M_{\lambda}}'$ be an irreducible subrepresentation of $\ind_{J'}^{M_{\lambda}}\lambda'$. Then the induced representation $\ind_{M_{\lambda}}^{\rG'}\lambda_{M_{\lambda}}'$ is irreducible and cuspidal.
\end{thm}

\begin{proof}
We only need to verify that $(M_{\lambda},\lambda_{M_{\lambda}}')$ satisfies the two conditions of irreducibility. We have proved in Proposition \ref{prop 14} that $(M_{\lambda},\lambda_{M_{\lambda}}')$ verifies the second condition. It is left only to prove the intertwining set of $\lambda_{M_{\lambda}}'$ in $\rG'$ equals to $M_{\lambda}$, i.e. $\mathrm{I}_{\rG'}(\lambda_{M_{\lambda}}')=M_{\lambda}$. In Lemma \ref{lem 17}, we have proved that $\mathrm{I}_{\rU(\fA)'}\lambda'\subset \tilde{J'}$. Since $\tilde{J}'$ normalizes $J'$, then $x\in \mathrm{I}_{\rU(\fA)}\lambda'$ verifying $\mathrm{Hom}_{J'}(\lambda',x(\lambda)')\neq 0$, which is equivalent to say that $x(\lambda')\cong\lambda$. Hence $\mathrm{I}_{\rU(\fA)}\lambda'\subset M_{\lambda}$. By the Proposition 3 in 8.10, chapter I of \cite{V1}, let $g\in\rG'$ and $X$ a finite set of $\rG'$ such that $M_{\lambda}gM_{\lambda}=\cup_{x\in X}J'xJ'$, then there is an $k$-isomorphism:
\begin{equation}
\label{equa 1}
\mathrm{I}_{g^{-1}}(\ind_{J'}^{M_{\lambda}} \lambda')\cong\oplus_{j\in X}\mathrm{I}_{(gj^{-1})}(\lambda').
\end{equation}
Furthermore, we have:
\begin{equation}
\label{equa 2}
\mathrm{I}_{\rU(\fA)'}(\ind_{J'}^{M_{\lambda}} \lambda')=M_{\lambda},
\end{equation}
Hence $\mathrm{I}_{\rU(\fA)'}(\lambda_{M_{\lambda}}')=M_{\lambda}$ follows by the inclusion:
$$\mathrm{I}_{\rU(\fA)'}(\lambda_{M_{\lambda}}')\subset\mathrm{I}_{\rU(\fA)'}(\ind_{J'}^{M_{\lambda}} \lambda').$$
Whence, there left to prove that
$$\mathrm{I}_{\mathrm{\rG'}}(\lambda_{M_{\lambda}}')\subset \rU(\fA)'.$$
Notice that $\ind_{M_{\lambda}}^{\rU(\fA)'}\lambda_{M_{\lambda}}'$ is a subrepresentation of $\res_{\rU(\fA)'}^{\rU(\fA)}\tau$, where $\tau=\ind_{J}^{\rU(\fA)}\lambda$ (as in the proof of Proposition \ref{prop 18}). We have:
$$\mathrm{I}_{\rG'}(\ind_{M_{\lambda}}^{\rU(\fA)}\lambda_{M_{\lambda}}')\subset\mathrm{I}_{\rG'}(\res_{\rU(\fA)'}^{\rU(\fA)}\tau),$$
since $\res_{\rU(\fA)'}^{\rU(\fA)}\tau$ is semisimple. We obtain then
\begin{equation}
\label{equa 3} 
\mathrm{I}_{\rG'}(\ind_{M_{\lambda}}^{\rU(\fA)}\lambda_{M_{\lambda}}')\subset\rU(\fA)'
\end{equation}
by Proposition \ref{prop 18}. Now use one more time Proposition 3 in 8.10, chapter I of \cite{V1} as equation (\ref{equa 1}) and equation (\ref{equa 2}): Let $h\in\rG'$ and $Y$ a finite set of $\rG'$ such that $\rU(\fA)h\rU(\fA)=\cup_{y\in Y}M_{\lambda}yM_{\lambda}$, then there is an $k$-isomorphism:
$$\mathrm{I}_{h^{-1}}(\ind_{M_{\lambda}}^{\rU(\fA)'} \lambda_{M_{\lambda}}')\cong\oplus_{s\in Y}\mathrm{I}_{(hs^{-1})}(\lambda_{M_{\lambda}}').$$
Hence we have:
$$\mathrm{I}_{\rG'}(\lambda_{M_{\lambda}}')\subset\mathrm{I}_{\rG'}(\ind_{M_{\lambda}}^{\rU(\fA)'}\lambda_{M_{\lambda}}').$$
Combining with the equation \ref{equa 3}, we deduce the result.
\end{proof}

\subsubsection{Cuspidal $k$-representations of $\rG'$}
Let $\rM$ denote a Levi subgroup of $\rG$, and $\rM'=\rM\cap\rG'$. In this section, we consider the restriction functor $\res_{\rM'}^{\rM}$, which has been studied by Tadi$\acute{\mathrm{c}}$ in \cite{TA} for representations with characteristic $0$. In his article, he proved that any irreducible complex representation of $\rM'$ is contained in an irreducible complex representation of $\rM$, and they are cuspidal simultaneously. His method can be adapted for the case of modulo $\ell$.

\begin{prop}
\label{propver13.1}
Let $K$ be a locally pro-finite group, and $K'\subset K$ is a closed normal subgroup of $K$ with finite index. Let $(\pi,V)$ be an irreducible $k$-representation of $K$, then the restricted representation $\res_{K'}^{K}\pi$ is semisimple with finite length.
\end{prop}

\begin{proof}
The proof is the same as \S 6.12,II in \cite{V1}. We repeat it again is to check that we can drop the condition that $[K:K']$ is inversible in $k$.

The restricted representation $\res_{K'}^{K}\pi$ is finitely generated, hence has an irreducible quotient. Let $V_0$ be the sub-representation such that $V\slash V_0$ is irreducible. Let $\{k_1,...,k_m\},m\in\mathbb{N}$ be a family of representatives of the quotient $K\slash K'$. Now we consider the kernel of the non-trivial projection from $\res_{K'}^{K}\pi$ to $\oplus_{i=1}^{m}V\slash k_i(V_0)$, which is $K$-stable, hence equals to $0$ since $\pi$ is irreducible. We deduce that $\res_{K'}^{K}\pi$ is a sub-representation of $\oplus_{i=1}^{m}V\slash k_i(V_0)$ hence is semisimple.
\end{proof}

\begin{prop}
\label{prop 6}
Let $\pi$ be any irreducible $k$-representation of $\rM$, then the restriction $\res_{\rM'}^{\rM}\pi$ is semisimple with finite length, and the direct components are $\rM$-conjugate. Conversely, let $\pi'$ be any irreducible $k$-representation of $\rM'$, then there exists an irreducible representation $\pi$ of $\rM$, such that $\pi'$ is a direct component of $\res_{\rM'}^{\rM}\pi$.
\end{prop}

\begin{proof}
For the first part of this proposition. The method of Silberger in \cite{Si} when $\ell=0$ can be generalised to our case that $\ell$ is positive. We first assume that $\pi$ is cuspidal. Let $Z$ denote the center of $\rM$, and the quotient $\rM\slash Z\rM'$ is compact. Since for any vector $v$ in the representation space of $\pi$ the stabiliser $\mathrm{Stab}_{\rM}(v)$ is open, the image of $\mathrm{Stab}_{\rM}(v)$ has finite index in the quotient group $\rM\slash Z\rM'$. Combining with Schur's lemma, the restricted $k$-representation $\res_{\rM'}^{\rM}\pi$ is finitely generated. By \S 2.7,II in \cite{V1} the restricted representation is $Z'=Z\cap\rM'$-compact.

Let $(v_1,...,v_m),m\in\mathbb{N}$ be a family of generators of the representation space of $\res_{\rM'}^{\rM}\pi$. For any compact open subgroup $K$ of $\rM'$, we want to prove the space $V^{K}$ is finitely dimensional. We could always assume that $K$ stabilises $v_i,i=1,...,m$, and consider the map 
$$\alpha_i:g\hookrightarrow e_Kgv_i,i=1,...m,$$
where $e_K$ is the idempotent associated $K$ in the Heck algebra of $\rM'$. Apparently, the space $V^{K}$ is generated by $\mathcal{L}=\{e_Kgv_i,g\in\rM',i=1,...,m\}$. If the dimension of $V^{K}$ is infinite, we can choose a infinite subset $\mathcal{L'}$ of $\mathcal{L}$ which forms a basis of $V^{K}$, especially there exists $i_0\in\{1,...,m\}$ such that $\rM_{i_0}'=\{g\in\rM',e_Kgv_{i_0}\in\mathcal{L'}\}$ is an infinite set. In particular, cosets $gK,g\in\rM_{i_0}'$ are disjoint since $K$ stabilizes $v_{i_0}$. Furthermore, since the center $Z'$ of $\rM'$ acts as a character on $\res_{\rM'}^{\rM}\pi$, which means $Z'$ stabilises each $v_i$, the images of cosets $gK,g\in\rM_{i_0}'$ are disjoint in the quotient $\rM'\slash Z'$. Let $v_{i_0}^{\ast}$ be an $k$-linear form of $V^{K}$ which equals to $1$ on the set $\mathcal{L}'$. The above analysis implies that the image of the support of coefficient $\langle v_{i_0}^{\ast}e_K,gv_{i_0}\rangle=\langle v_{i_0}^{\ast},e_Kgv_{i_0}\rangle$ in $\rM'\slash Z'$ contains infinite disjoint cosets $gK,g\in\rM_{i_0}'$, which contradicts with the assumption that $\res_{\rM'}^{\rM}\pi$ is $Z'$-compact. We conclude that $\res_{\rM'}^{\rM}\pi$ is finitely generated and admissible, hence has finite length.

Now we come back to the general case:$\pi$ is irreducible representation of $\rM$. We first prove that $\res_{\rM'}^{\rM}\pi$ has finite length, then we prove it is semisimple. For the first part, it is sufficient to prove the restricted representation $\res_{\rM'}^{\rM}\pi$ is finitely generated and admissible. Let $(\rL,\sigma)$ be a cuspidal pair in $\rM$ such that $\pi$ is a sub-representation of $i_{\rL}^{\rM}\sigma$. Applying Theorem \ref{thm 5.2}, we have $\res_{\rM'}^{\rM}i_{\rL}^{\rM}\sigma\cong i_{\rL'=\rL\cap\rM'}^{\rM'}\res_{\rL'}^{\rL}\sigma$. We have proved that $\res_{\rL'}^{\rL}\sigma$ is admissible and finitely generated. Since normalised parabolic induction $i_{\rL'}^{\rM'}$ respect admissibility and finite generality, the $k$-representation $\res_{\rL'}^{\rL}\sigma$ is also admissible and finitely generated, and hence has finite length. So does its sub-representation $\res_{\rM'}^{\rM}\pi$. For the semi-simplicity, let $W$ be an irreducible sub-representation of $\res_{\rM'}^{\rM}\pi$, of which $gW$ is also an irreducible sub-representation for $g\in\rM$. Let $W'=\sum_{g\in\rM'}g(W)$, which is a semisimple (by the equivalence condition in \S A.VII. of \cite{Re}) sub-representation of $\res_{\rM'}^{\rM}\pi$. Obviously, $\rM$ stabilises $W'$, hence $W'=\res_{\rM'}^{\rM}\pi$ by the irreducibility of $\pi$.

Now we consider the second part of this proposition, and apply the proof of Proposition \S 2.2 in \cite{TA} in our case. Let $\pi'$ be any irreducible $k$-representation of $\rM'$, and $\mathrm{S}$ the subgroup of $\mathrm{Z}$ generated by the scalar matrix $\varpi_{F}$, where $\varpi_{F}$ is the uniformizer of the ring of integers of $\fo_{F}$. It is clear that the intersection $\mathrm{S}\cap\rM'=\{ \mathds{1}\}$. Hence we could let $\tilde{\pi}$ denote the extension of $\pi$ to $\mathrm{S}\rM'$, where $\mathrm{S}$ acts as identity. The quotient group $\rM/\mathrm{S}\rM'$ is compact, hence the induced representation $\ind_{\mathrm{S}\rM'}^{\rM}\tilde{\pi}$ is admissible (see the formula in \S I,5.6 of \cite{V1}). For any $\rM$-subrepresentation $\tau$ of $\ind_{\mathrm{S}\rM'}^{\rM}\tilde{\pi}$, there is a surjective morphism from $\res_{\mathrm{S}\rM'}^{\rM}\tau$ to $\tilde{\pi}$, defined as $f\mapsto f(\mathds{1})$. This induced a surjective morphism from $\res_{\rM'}^{\rM}\tau$ to $\pi'$.

Now let $\pi_1$ be a finitely generated subrepresentation of $\ind_{\mathrm{S}\rM'}^{\rM}\tilde{\pi}$. Since $\pi_1$ is finite type and admissible, it has finite length containing an irreducible subrepresentation noted as $\pi$. And there is a surjective morphism from $\res_{\rM'}^{\rM}\pi$ to $\pi'$. Combining this with the first part above, the representation $\pi$ is the one we want.
\end{proof}

\begin{cor}
\label{cor 20}
Let $\pi$ be an irreducible $k$-representation of $\rM$. If the restricted representation $\res_{\rM'}^{\rM}\pi$ contains an irreducible cuspidal $k$-representation of $\rM'$, then $\pi$ is cuspidal. This is to say that any cuspidal $k$-representation of $\rM'$ is a subrepresentation of $\res_{\rM'}^{\rM}\pi$ for some cuspidal $k$-representation $\pi$ of $\rM$.
\end{cor}

\begin{proof}
For the first part above, we know that the direct components of $\res_{\rM'}^{\rM}\pi$ are $\rM$-conjugate by Proposition \ref{prop 6}. Let $\mathrm{P}'=\mathrm{L}'\cdot\rU$ be any proper parabolic subgroup of $\rM'$ and $\mathrm{P}=\mathrm{L}\cdot\rU$ the proper parabolic subgroup of $\rM$ such that $\mathrm{P}\cap\rM'=\mathrm{P}'$ and $\mathrm{L}\cap \rM'=\mathrm{L}'$. Let $\pi_0'$ be any direct component of $\res_{\rM'}^{\rM}\pi\cong\oplus_{i\in I}\pi_i'$, where $I$ is a finite index set and $\pi_i'$ are irreducible representations of $\rM'$, and for each $i\in I$ let $a_i\in\rM$ such that $\pi_i'\cong a_i(\pi')$. In particular, we could assume that $\{a_i\}_{i\in I}$ is a subset of $\mathrm{L}$. We have:
$$\res_{\mathrm{L'}}^{\mathrm{L}}r_{\mathrm{L}}^{\rM}\pi\cong\oplus_{i\in I} r_{\mathrm{L'}}^{\rM'}\pi_i'.$$
Meanwhile, since the unipotent radical $\mathrm{U}$ is normal in $\mathrm{L}$, we deduce that:
$$r_{\mathrm{L'}}^{\rM'}\pi_i'\cong a_i(r_{\mathrm{L'}}^{\rG'}\pi')\cong 0.$$
Hence $\pi$ is cuspidal as required.
\end{proof}

\begin{cor}
\label{cor 21}
For any irreducible cuspidal $k$-representation $\pi'$ of $\rG'$, there exists a maximal simple cuspidal $k$-type $(J,\lambda)$ of $\rG$, and $M_{\lambda}$ as in definition \ref{defn 15}. There exists a direct component $\lambda_{M_{\lambda}}'$ of $\ind_{J'}^{M_{\lambda}}\res_{J'}^J\lambda$ such that $\pi'$ is isomorphic to the induced representation $\ind_{M_{\lambda}}^{\rG'} \lambda_{M_{\lambda}}'$.
\end{cor}

\begin{proof}
Applying Corollary \ref{cor 20}, let $\pi$ be an irreducible cuspidal $k$-representation of $\rG'$ which contains $\pi'$ as a sub-$\rG'$-representation. Let $(J_0,\lambda_0)$ be the maximal simple cuspidal $k$-type of $\rG$ corresponding to $\pi$, and $(M_{\lambda_0}, \lambda_{M_{\lambda_0}}')$ as in Theorem \ref{thm 16}. We know that $\pi$ is isomorphic to $\ind_{E^{\times}J_0}^{\rG}\Lambda_0$, where $\Lambda_0$ is an extension of $\lambda$ to $E^{\times}J_0$, and the intersection $E^{\times}J_0\cap\rG'=J_0'$. Then after applying Mackey's decomposition formula to $\res_{\rG'}^{\rG}\pi$, we obatin that of which the representation $\ind_{J_0'}^{\rG'}\res_{J_0'}^J\lambda_0$ is a subrepresentation. Hence, $\ind_{M_{\lambda_0}}^{\rG'}\lambda_{M_{\lambda_0}}'$ is isomorphic to some direct component of $\res_{\rG'}^{\rG}\pi$, which is isomorphic to $g(\pi')$ for some $g\in \rG$ by Proposition \ref{prop 6}. This implies that $\pi'$ contains $g^{-1}(\lambda_{M_{\lambda_0}}')$. Notice that $g^{-1}(M_{\lambda_0})=M_{g^{-1}(\lambda_0)}$ and $g^{-1}(\lambda_{M_{\lambda_0}}')$ is a direct component of $\ind_{g^{-1}(J')}^{M_{g^{-1}(\lambda_0)}}g^{-1}(\lambda')$, so we could write is as $\lambda_{M_{g^{-1}(\lambda_0)}}'$. Hence by Frobenius reciprocity and Theorem \ref{thm 16}, this implies that $\pi'\cong\ind_{M_{g^{-1}(\lambda_0)}}^{\rG'}g^{-1}(\lambda_{M_{g^{-1}(\lambda_0)}}')$. And $(g^{-1}(J_0),g^{-1}(\lambda))$ is the required maximal simple cuspidal $k$-type.
\end{proof}

\subsection{Whittaker models and maximal simple cuspidal $k$-types of $\rG'$}
\subsubsection{Uniqueness of Whittaker models}
In this section, we will see that the subgroup $M_{\lambda}$ of $\tilde{J}'=\tilde{J}(\lambda)\cap\rG'$ in the definition \ref{defn 22} actually coincides with $\tilde{J}'$. In other words, we will prove that for any element $x\in\rU(\fA)$, if $x$ normalises $J$ and $x(\lambda)\cong\lambda\otimes\chi\circ\det$ for some $k$-quasicharacter $\chi$ of $F^{\times}$, then $x(\lambda')\cong\lambda'$, for any irreducible direct component $\lambda'$ of $\lambda\vert_{J'}$.

Let $\rU=\rU_n(F)$ be the group consisting with those strictly upper triangular matrices in $\rG$.  A non-degenerate character $\psi$ of $\rU$ is a $k$-quasicharacter defined on $\rU$. Let
$P_n=\mathrm{P}_n(F)$ be the mirabolic subgroup of $\mathrm{GL}_n(F)$, and $P_n'=P_n\cap \SL_n(F)$. We denote the unipotent radical of $P_n$ as $V_{n-1}$, which is an abelian group isomorphic to the additive group $F^{n-1}$. The unipotent radical of $P_n'$ is also $V_{n-1}$.

\begin{defn}
\label{defn 25}
\begin{enumerate}
\item $r_{\mathrm{id}}:=r_{\rG_{n-1},P_n}$ the functor of $V_{n-1}$-coinvariants of representations of $P_{n}$, $r_{\mathrm{id}'}:=r_{\rG_{n_1}',P_n'}$ the functor of $V_{n-1}$-coinvariants of representations of $P_{n}'$.
\item $r_{\psi}:=r_{\psi,P_{n-1},P_n}$ the functor of $(V_{n-1},\psi)$-coinvariants of representations of $P_{n}$, $r_{\psi}':=r_{\psi,P_{n-1}',P_n'}$ the functor of $(V_{n-1},\psi)$-coinvariants of representations of $P_{n}'$.
\end{enumerate}
\end{defn}

\begin{defn}
\label{defn 26}
Let $1\leq k\leq n$ and $\pi\in \mathrm{Mod}_{k}P_n$, $\pi'\in\mathrm{Mod}_k P_n'$. We define the $k$-th derivative of $\pi$ to be the representation $\pi^{(k)}:=r_{\mathrm{id}}r_{\psi}^{k-1}\pi$, and the $k$-th derivative of $\pi'$ relative to $\psi$ to be the representation $\pi'^{(\psi,k)}:=r_{\mathrm{id}}'r_{\psi}'^{k-1}\pi'$.
\end{defn}

\begin{rem}
\label{rem 27}
The unipotent radical of $P_n$ and $P_n'$ coincide and $\rU\subset\rG'$, so $\res_{\rG_{n-k}'}^{\rG_{n-k}}\pi^{(k)}$ is equivalent to $(\res_{\rG_{n}'}^{\rG_n}\pi)^{(k)}$, where $\pi\in\mathrm{Mod}_n\rG_n$.
\end{rem}

\begin{prop}
\label{prop 23}
Let $\pi$ be a cuspidal $k$-representation of $\rG$, then the restriction $\res_{\rG'}^{\rG}\pi$ is multiplicity free.
\end{prop}

\begin{proof}
We have proved in Proposition \ref{prop 6}, that the restriction $\res_{\rG_n'}^{\rG_n}\pi$ is semisimple with finite direct components. Hence we could write it as $\oplus_{i=1}^m \pi_i$, where $m\in\mathbb{N}$ and $\pi_i$'s are irreducible $k$-cuspidal representations of $\rG_n'$. Let $\psi$ be any non-degenerate character of $\rU$. As in 1.7 chapter III of \cite{V1}, we obtain that $\mathrm{dim}\pi^{(n)}=1$. We apply Remark \ref{rem 27} above, then
$$\mathrm{dim}(\res_{\rG_n'}^{\rG_n}\pi)^{(n)}=\oplus_{i=1}^m\mathrm{dim}(\pi_i)^{(\psi,n)}=1.$$
So there exists one unique components $\pi_{i_0}$, where $1\leq i_0\leq m$, such that $\pi_{i_0}^{(\psi,n)}$ is non-trivial. And we deduce the result.
\end{proof}

\begin{cor}
Let $\pi'$ be an irreducible cuspidal $k$-representation of $\rG'$. Then there exists a non-degenerate character $\psi$ of $\rU$, such that $\mathrm{dim}\pi'^{(g(\psi),n)}=1$.
\end{cor}

\begin{proof}
This is deduced from Corollary \ref{cor 20}. In fact the direct components of $\res_{\rG'}^{\rG}\pi$ is Corollary \ref{cor 20} are conjugated to each other by diagonal matrices, and the conjugation of non-degenerate characters of $\rU$ by any diagonal matrix is also a non-degenerate character of $\rU$.
\end{proof}

\subsubsection{Distinguished cuspidal $k$-types of $\rG'$}

\begin{prop}
\label{prop 24}
Let $(J,\lambda)$ be a maximal simple cuspidal $k$-type of $\rG$, and $\tilde{J}$ the projective normalizer of $\lambda$. Then the subgroup $M_{\lambda}$ in definition \ref{defn 15} of $\tilde{J}'$ coincides with $\tilde{J}'$.
\end{prop}

\begin{proof}
Let $\Lambda$ be an extension of $\lambda$ to $E^{\times}J$. Then $\ind_{E^{\times}J}^{\rG}\Lambda$ is an irreducible cuspidal representation of $\rG$, we denote it as $\pi$. The restricted representation $\res_{\rG'}^{\rG}\pi$ is semisimple and its direct components are cuspidal. By Theorem \ref{thm 16}, there exists a direct component $\pi'$ of $\res_{\rG'}^{\rG}\pi$, such that $\pi'$ is isomorphic to $\ind_{M_{\lambda}}^{\rG'}\lambda_{M_{\lambda}}'$, for some $\lambda_{M_{\lambda}}'$. In the proof of Theorem \ref{thm 16}, we have showed that the intertwining subgroup $\mathrm{I}_{\rG'}(\lambda_{M_{\lambda}}')$ equals to $M_{\lambda}$. If $\tilde{J}'\neq M_{\lambda}$, and let $x$ be an element belonging to $\tilde{J'}$ but not to $M_{\lambda}$. Then $x(\lambda_{M_{\lambda}}')$ is not isomorphic to $\lambda_{M_{\lambda}}'$. However $x(\pi')\cong\pi'$, so $\res_{M_{\lambda}}^{\rG'}\pi'$ contains $x(\lambda_{M_{\lambda}})$, from which we deduce that
\begin{equation}
\label{equa 5}
\ind_{M_{\lambda}}^{\rG'}x(\lambda_{M_{\lambda}}')\cong\ind_{M_{\lambda}}^{\rG'}\lambda_{M_{\lambda}}'\cong\pi'.
\end{equation}
Meanwhile, by Mackey's decomposition formula
$$\res_{M_{\lambda}}^{\tilde{J}'}\ind_{J'}^{\tilde{J}'}\res_{J'}^{J}\lambda\cong\oplus_{J'\backslash \tilde{J}'/M_{\lambda}}\ind_{J'}^{M_{\lambda}}\res_{J'}^{J}\lambda,$$
hence $x(\lambda_{M_{\lambda}}')$ is another direct component of $\ind_{J'}^{M_{\lambda}}\res_{J'}^{J}\lambda$. Since we could change order of the functor $\ind_{M_{\lambda}}^{\rG'}$ and finite direct sum, and $\ind_{J'}^{\rG'}\res_{J'}^{J}\lambda$ is a subrepresentation of $\res_{\rG'}^{\rG}\pi$, the two representations $\ind_{M_{\lambda}}^{\rG'}x(\lambda_{M_{\lambda}}')$ and $\ind_{M_{\lambda}}^{\rG'}\lambda_{M_{\lambda}}'\cong\pi'$ are two different direct components of $\res_{\rG'}^{\rG}\pi$. By Proposition \ref{prop 23}, they are not isomorphic, which is contradicted to the equivalence \ref{equa 5}. Hence $\tilde{J}'=M_{\lambda}$.
\end{proof}

\begin{defn}
\label{defn 22}
Let $(J,\lambda)$ be a maximal simple cuspidal $k$-type of $\rG$ and $\tilde{J}'=\tilde{J}\cap \rG'$ as in definition \ref{defn 11}, and $\tilde{\lambda}'$ any direct component of $\ind_{J'}^{\tilde{J}'}\res_{J'}^{J}\lambda$. We define the couples $(\tilde{J'},\tilde{\lambda}')$ to be maximal simple cuspidal $k$-types of $\rG'$. By Corollary \ref{cor 21} and Proposition \ref{prop 24}, for any irreducible cuspidal $k$-representation of $\rG'$, there exists a maximal simple cuspidal $k$-type $(\tilde{J'},\tilde{\lambda}')$ of $\rG'$ such that $\pi'$ is isomorphic to $\ind_{\tilde{J}'}^{\rG'}\tilde{\lambda}'$.
\end{defn}

\subsection{Maximal simple cuspidal $k$-types for Levi subgroups of $\rG'$}
\label{section 06}
\subsubsection{Intertwining and weakly intertwining}
In this section, let $\rM$ denote any Levi subgroup of $\rG$ and for any closed subgroup $H$ of $\rG$, we always use $H'$ to denote its intersection with $\rG'$. We will consider the maximal simple cuspidal $k$-types of $\rM$. Recall that Proposition \ref{Lprop 0.1}, Proposition \ref{Lprop 0.2}, Definition \ref{prepdefn 01} and Proposition \ref{Lprop 0.3} will be used in this section.

\begin{prop}
\label{Lprop 1}
Let $(J_{\rM},\lambda_{\rM})$ be a maximal simple cuspidal $k$-type of $\rM$, and $\chi$ a $k$-quasicharacter of $F^{\times}$. If $(J_{\rM},\lambda_{\rM}\otimes\chi\circ\det)$ is weakly intertwined with $(J_{\rM},\lambda_{\rM})$, then they are intertwined. There exists an element $x\in\rU(\fA_{\rM})=\rU(\fA_1)\times\cdots\times\rU(\fA_r)$ such that $x(J_{\rM})=J_{\rM}$ and $x(\lambda_{\rM})\cong\lambda_{\rM}\otimes\chi\circ\det$, where $\fA_i$ is a hereditary order associated to $(J_i,\lambda_i)$ ($i=1,...,r$). Furthermore, for any $g\in\rG$, if $g$ weakly intertwines $(J_{\rM},\lambda_{\rM}\otimes\chi\circ\det)$ and $(J_{\rM},\lambda_{\rM})$, then $g$ intertwines $(J_{\rM},\lambda_{\rM}\otimes\chi\circ\det)$ and $(J_{\rM},\lambda_{\rM})$.

\end{prop}

\begin{proof}

By definition, write $\rM$ as a product $\mathrm{GL}_{n_1}\times\cdots\times\mathrm{GL}_{n_r}$, then $J_{\rM}=J_1\times\cdots\times J_r$ and $\lambda_{\rM}\cong\lambda_1\times\cdots\times\lambda_r$, where $(J_i,\lambda_i)$ are $k$-maximal cuspidal simple type of $\mathrm{GL}_{n_i}$ for $i\in\{ 1,\ldots,r\}$. The group $\rU(\fA_{\rM})=\rU(\fA_1)\times\cdots\times\rU(\fA_r)$. Hence the two results are directly deduced by \ref{prop 1} and \ref{cor 5}.

\end{proof}

\begin{defn}
\label{Ldefinition 21}
Let $(J_{\rM},\lambda_{\rM})$ be a $k$-maximal cuspidal simple type of $\rM$. We define the group of projective normalizer $\tilde{J_{\rM}}$ a subgroup of $J_{\rM}$. An element $x\in \rU(\fA_{\rM})$, where $\fA_{\rM}=\fA_1\times\cdots\times\fA_r$, belongs to $\tilde{J_{\rM}}$, if $x(J_{\rM})=J_{\rM}$, and there exists a $k$-quasicharacter $\chi$ of $F^{\times}$ such that $x(\lambda_{\rM})\cong \lambda_{\rM}\otimes\chi\circ\mathrm{det}$.  
\end{defn}

The induced $k$-representation $\tilde{\lambda}_{\rM}=\ind_{J_{\rM}}^{\tilde{J}_{\rM}}\lambda_{\rM}$ is irreducible by Corollary \ref{cor 99}, and according to \ref{Lprop 0.1}, the restriction $\res_{\tilde{J}_{\rM}'}^{\tilde{J_{\rM}}}\tilde{\lambda}_{\rM}$ is semisimple. Let $\mu_{\rM}$ denote one of its irreducible component. 

\begin{lem}
\label{Llemma 3}
Let $(J_{\rM},\lambda_{\rM})$ be a $k$-maximal cuspidal simple type of $\rM$, and $\nu_{\rM}$ and $\mu_{\rM}$ be two irreducible components of the restricted representation $\res_{\tilde{J}_{\rM}'}^{\tilde{J}_{\rM}}\tilde{\lambda}_{\rM}$. Then :
$$\mathrm{I}_{\rM'}^{w}(\nu_{\rM},\mu_{\rM})=\{ m\in \rM':m(\nu_{\rM})\cong\mu_{\rM} \},$$
whence $\mathrm{I}_{\rM'}^{w}(\nu_M,\mu_M)=\mathrm{I}_{\rM'}(\nu_M,\mu_M)$. In particular, $\mathrm{I}_{\rM'}(\mu_{\rM})$ is the normalizer group of $\mu_{\rM}$ in $\rM'$. Moreover, this group is independent of the choice of $\mu_{\rM}$.
\end{lem}

\begin{proof}
Let $m\in\rM'$ weakly intertwines $\mu_{\rM}$ with $\nu_{\rM}$. Then by \ref{Lprop 0.3}, the element $m$ weakly intertwines $\tilde{\lambda}_{\rM}$ with $\tilde{\lambda}_{\rM}\otimes\chi\circ\mathrm{det}$ for some $k$-quasicharacter $\chi$ of $F^{\times}$. By definition
$$\tilde{\lambda}_{\rM}\vert_{J_{\rM}}\cong\oplus_{x\in \tilde{J_{\rM}}/J_{\rM}}x(\lambda_{\rM})\cong \oplus_{x\in\tilde{J}_{\rM}/J_{\rM}}\lambda_{\rM}\otimes\xi_x\circ\mathrm{det}.$$
Since the induced representation $\ind_{J_{\rM}}^{\tilde{J}_{\rM}}\lambda_{\rM}\otimes\xi_x\circ\mathrm{det}\cong\tilde{\lambda}_{\rM}\otimes\xi_x\circ\mathrm{det}$, by Frobenius reciprocity, we have $\tilde{\lambda}_{\rM}\otimes\xi_x\circ\mathrm{det}\cong\tilde{\lambda}_{\rM}$ for every $x\in\tilde{J}_{\rM}/J_{\rM}$. It follows that for some $g\in \tilde{J}_{\rM}$, the element $gm$ weakly intertwines $\lambda_{\rM}$ with $\lambda_{\rM}\otimes\xi_x\cdot\chi\circ\mathrm{det}$ for some $x\in \tilde{J}_{\rM}/J_{\rM}$. Applying \ref{Lprop 1}, the element $gm$ intertwines $\lambda_{\rM}$ with $\lambda_{\rM}\otimes\xi_x\cdot\chi\circ\mathrm{det}$, and there exists an element $y\in\tilde{J}_{\rM}$ such that $y(\lambda_{\rM})\cong \lambda_{\rM}\otimes\xi_x\cdot\chi\circ\mathrm{det}$. Inducing this isomorphism to $\tilde{J}_{\rM}$, we see tha $\tilde{\lambda}_{\rM}\cong\tilde{\lambda}_{\rM}\otimes\chi\circ\mathrm{det}$, whence $m$ intertwines $\tilde{\lambda}_{\rM}$.

Furthermore, the intertwining set $\mathrm{I}_{\rM}(\lambda_{\rM})=N_{\rM}(\lambda_{\rM})$, the latter group is the normalizer of $\lambda_{\rM}$, which also normalizes $\rU(\fA_{\rM})$, hence normalizes $\tilde{J}_{\rM}$. We deduce that $\mathrm{I}_{\rM}(\tilde{\lambda}_{\rM})=\tilde{J}_{\rM} N_{\rM}(\lambda_{\rM})$. Then each element of $\mathrm{I}_{\rM}^{w}(\mu_{\rM},\nu_{\rM})$ normalizes $\tilde{\lambda}_{\rM}$ and the group $\tilde{J}_{\rM}'$. This gives the first two assertions.

To prove the third assertion, observe that the irreducible components of $\tilde{\lambda}_{\rM}\vert_{\tilde{J}_{\rM}'}$ form a single $\tilde{J}_{\rM}$-conjugacy class. We have to show therefore that $\tilde{J}_{\rM}$ normalizes $N_{\rM'}(\mu_{\rM})$.

The quotient group $N_{\rM}(\tilde{\lambda}_{\rM})/\tilde{J}_{\rM}$ is abelian. In fact, as we have proved above, it is a subgroup of $N_{\rM}(\lambda_{\rM})/J_{\rM}$. The latter group is abelian, since $N_{\rM}(\lambda_{\rM})$ can be written as $E_1^{\times}J_1\times\cdots\times E_r^{\times}J_r$, where $E_1,\ldots,E_r$ are field extensions of $F$. Now let $x\in \tilde{J}_{\rM}$ and $y\in N_{\rM'}(\mu_{\rM})$, we have $x^{-1}yx=y\cdot m$ for some $m\in\tilde{J}_{\rM}'$. Therefore:
$$x^{-1}yx(\mu_{\rM})\cong y(\mu_{\rM})\cong\mu_{\rM},$$
as required.
\end{proof}

\begin{rem}
To be more detailed, we proved that the intertwining group $\mathrm{I}_{\rM'}(\mu_{\rM})$ is the stabilizer group $N_{\rM'}(\mu_{\rM})$, which is a subgroup of $E_1^{\times}\tilde{J}_1\times\cdots\times E_r^{\times}\tilde{J}_r \cap \rM'$, hence a compact group modulo center.
\end{rem}

\subsubsection{Maximal simple cuspidal $k$-types of $\rM'$}
In this section, we construct maximal simple cuspidal $k$-types of $\rM'$ (\ref{Ldefinition 20}). This means that for any irreducible cuspidal $k$-representation $\pi'$, there exists an irreducible component $\mu_{\rM}$ of $\res_{\tilde{J}_{\rM}'}^{\tilde{J}_{\rM}}\tilde{\lambda}_{\rM}$, and an irreducible $k$-representation $\tau_{\rM'}$ of $N_{\rM'}(\mu_{\rM})$ containing $\mu_{\rM}$, such that $\pi'\cong\ind_{N_{\rM'}(\mu_{\rM})}^{\rM'}\tau_{\rM'}$. We follow the same method as in the case of $\rG'$, which is to calculate the intertwining group and verify the second condition of irreducibility (\ref{lem 19}).

\begin{lem}
\label{Llemma 2}
As in the case when $\rM=\rG$, we have a decomposition:
$$\res_{J_{\rM}}^{\rM}\ind_{J_{\rM}}^{\rM}\lambda_{\rM}\cong \Lambda_{\lambda_{\rM}}\oplus W_{\rM},$$
where $\Lambda_{\lambda_{\rM}}$ is semisimple, of which each irreducible component is isomorphic to $\lambda_{\rM}\otimes\chi\circ\mathrm{det}$ for some $k$-quasicharacter $\chi$ of $F^{\times}$. Non of irreducible subquotient of $W_{\rM}$ is contained in $\Lambda_{\lambda_{\rM}}$.
\end{lem}

\begin{proof}
This is directly deduced from the decomposition in \ref{thm 7}.
\end{proof}

\begin{prop}
\label{Lproposition 4}
Let $\mu_{\rM}$ be an irreducible $k$-subrepresentation of $\res_{\tilde{J}_{\rM}'}^{\tilde{J}_{\rM}}\ind_{J_{\rM}}^{\tilde{J}_{\rM}}\lambda_{\rM}$. Then $\mu_{\rM}$ verifies the second condition of irreducibility (\ref{lem 19}). This means that for any irreducible representation $\pi'$ of $\rM'$, if there is an injection $\mu_{\rM}\rightarrow \res_{\tilde{J}_{\rM}'}^{\rM'}\pi'$, then there exists a surjection from $\res_{\tilde{J}_{\rM}'}^{\rM'}\pi'$ to $\mu_{\rM}$.
\end{prop}

\begin{proof}
A same proof can be given as in the case while $\rM'=\rG'$ (\ref{prop 14}). 
\end{proof}

\begin{prop}
\label{Lproposition 3}
Let $\tau_{\rM'}$ be an irreducible representation of $N_{\rM'}(\mu_{\rM})$ containing $\mu_{\rM}$. Then $\tau_{\rM'}$ verifies the second condition of irreducibility.
\end{prop}

\begin{proof}
Let $N_{\rM'}$ denote $N_{\rM'}(\mu_{\rM})$, then we have:
$$\res_{N_{\rM'}}^{\rM'}\ind_{N_{\rM'}}^{\rM'}\tau_{\rM'}\cong\oplus_{N_{\rM'}\backslash\rM' / N_{\rM'}}\ind_{N_{\rM'}\cap a(N_{\rM'})}^{N_{\rM'}}\res_{N_{\rM'}\cap a(N_{\rM'})}^{a(N_{\rM'})}a(\tau_{\rM'}).$$
Notice that $N_{\rM'}$ has a unique maximal open compact subgroup $\tilde{J}_{\rM}'$, hence $\tilde{J}_{\rM}'\cap ba(N_{\rM'})=\tilde{J}_{\rM}'\cap ba(\tilde{J}_{\rM}')$, for any $b,a\in\rM'$. Hence we have the following equivalence:
$$\res_{\tilde{J}_{\rM}'}^{N_{\rM'}}\ind_{N_{\rM'}\cap a(N_{\rM'})}^{N_{\rM'}}\res_{N_{\rM'}\cap a(N_{\rM'})}^{a(N_{\rM'})}a(\tau_{\rM'})$$
$$\cong\oplus_{b\in N_{\rM'}\cap a(N_{\rM'})\backslash N_{\rM'} / \tilde{J}_{\rM}'} \ind_{\tilde{J}_{\rM}'\cap ba(N_{\rM'})}^{\tilde{J}_{\rM}'}\res_{\tilde{J}_{\rM}'\cap ba(N_{\rM'})}^{ba(N_{\rM'})}ba(\tau_{\rM'})$$
$$\cong\oplus_{b \in N_{\rM'}\cap a(N_{\rM'})\backslash N_{\rM'} / \tilde{J}_{\rM}'} \ind_{\tilde{J}_{\rM}'\cap ba(\tilde{J}_{\rM}')}^{\tilde{J}_{\rM}'}\res_{\tilde{J}_{\rM}'\cap ba(\tilde{J}_{\rM}')}^{ba(\tilde{J}_{\rM}')}(\oplus\mu_{\rM}),$$
where $\oplus\mu_{\rM}$ denotes a finite multiple of $\mu_{\rM}$.

Let $a\notin N_{\rM'}$, then $ba$ is an element of $N_{\rM'}\cdot a$, and $N_{\rM'}\cdot a\cap N_{\rM'}=\emptyset$. By \ref{Llemma 3}, this means $ba\notin \mathrm{I}_{\rM'}^{w}(\mu_{\rM})$. This implies that non of irreducible subquotient of $\ind_{\tilde{J}_{\rM}'\cap ba(N_{\rM'})}^{\tilde{J}_{\rM}'}\res_{\tilde{J}_{\rM}'\cap ba(N_{\rM'})}^{ba(N_{\rM'})}ba(\tau_{\rM'})$ is isomorphic to $\mu_{\rM}$. Now combining with the first equivalence in this proof above, we obtain a decomposition:
$$\res_{N_{\rM'}}^{\rM'}\ind_{N_{\rM'}}^{\rM'}\tau_{\rM'}\cong\tau_{\rM'}\oplus W_{N_{\rM'}},$$
non of irreducible subquotient of $W_{N_{\rM'}}$ is isomorphic to $\tau_{\rM'}$.

Now we verify the second condition of $\tau_{\rM'}$. Let $\pi'$ be any irreducible $k$-representation of $\rM'$. If there is an injection $\tau_{\rM'}\hookrightarrow \res_{N_{\rM'}}^{\rM'}\pi'$, then $\res_{N_{\rM'}}^{\rM'}\pi'$ is isomorphic to a quotient representation $(\res_{N_{\rM'}}^{\rM'}\ind_{N_{\rM'}}^{\rM'}\tau_{\rM'})/ W_0$. And the image of the composed morphism below:
$$\tau_{\rM'}\hookrightarrow \res_{N_{\rM'}}^{\rM'}\pi'\cong(\res_{N_{\rM'}}^{\rM'}\ind_{N_{\rM'}}^{\rM'}\tau_{\rM'})/ W_0$$
is not contained in $(W_{N_{\rM'}}+W_0)/ W_0$ by the analysis above. Then we have a non trivial morphism:
$$\res_{N_{\rM'}}^{\rM'}\pi'\rightarrow (\tau_{\rM'}\oplus W_{N_{\rM'}})/ (W_{N_{\rM'}}+W_0)\cong\tau_{\rM'}.$$
Hence we finish the proof.

\end{proof}

\begin{lem}
\label{Llemma 0.3}
Let $\rG$ be a locally pro-finite group, and $K_1,K_2$ two open subgroups of $\rG$, where $K_1$ is the unique maximal open compact subgroup in $K_2$. Let $\pi$ be an irreducible $k$-representation of $K_2$, and $\tau$ an irreducible $k$-representation of $K_1$. Assume that $\pi\vert_{K_1}$ is a multiple of $\tau$. If $x\in\rG$ (weakly) intertwines $\pi$, then there exists an element $y\in K_2$ such that $yx$ (weakly) intertwines $\tau$.
\end{lem}

\begin{proof}
Since $\pi$ is isomorphic to a subquotient of $\ind_{K_2\cap x(K_2)}^{K_2}\res_{K_2\cap x(K_2)}^{x(K_2)}x(\pi)$, the restriction $\res_{K_1}^{K_2}\pi$ is isomorphic to a subquotient of $\res_{K_1}^{K_2}\ind_{K_2\cap x(K_2)}^{K_2}\res_{K_2\cap x(K_2)}^{x(K_2)}x(\pi)$. Applying Mackey's decomposition formula, we have
$$\res_{K_1}^{K_2}\ind_{K_2\cap x(K_2)}^{K_2}\res_{K_2\cap x(K_2)}^{x(K_2)}x(\pi)\cong \bigoplus_{a\in K_2\cap x(K_2)\slash K_2\backslash K_1}\ind_{K_1\cap ax(K_2)}^{K_1}\res_{K_1\cap ax(K_2)}^{ax(K_2)}ax(\pi).$$
Since $K_1\cap ax(K_2)$ is open compact in $ax(K_2)$, by the uniqueness of open compact subgroup in $ax(K_2)$, the intersection $K_1\cap ax(K_2)\subset ax(K_1)$, hence $K_1\cap ax(K_2)=K_1\cap ax(K_1)$. Write $\res_{K_1}^{K_2}\pi\cong\bigoplus_{I}\tau$, where $I$ is an index set. We have an equivalence
$$\res_{K_1\cap ax(K_1)}^{ax(K_1)}ax(\pi)\cong\bigoplus_{I}\res_{K_1\cap ax(K_1)}^{ax(K_1)}ax(\tau)$$
Since functors $\ind,\res$ can change order with infinite direct sum, we reform the first equivalence in this proof
$$\res_{K_1}^{K_2}\ind_{K_2\cap x(K_2)}^{K_2}\res_{K_2\cap x(K_2)}^{x(K_2)}x(\pi)$$
$$\cong \bigoplus_{I}\bigoplus_{a\in K_2\cap x(K_2)\slash K_2\backslash K_1}\ind_{K_1\cap ax(K_1)}^{K_1}\res_{K_1\cap ax(K_1)}^{ax(K_1)}ax(\tau).$$
As in the proof of Lemma \ref{lem 9}, this implies that there exists at least one $y\in K_2$ such that $\tau$ is an subquotient of $\ind_{K_1\cap yx(K_1)}^{K_1}\res_{K_1\cap yx(K_1)}^{yx(K_1)}yx(\tau)$.
\end{proof}

\begin{thm}
\label{Ltheorem 5}
The induced $k$-representation $\ind_{N_{\rM'}(\mu_{\rM})}^{\rM'}\tau_{\rM'}$ is cuspidal and irreducible. Conversely, any irreducible cuspidal representation $\pi'$ of $\rM'$ contains an irreducible $k$-representation $\tau_{\rM'}$ of $N_{\rM'}(\mu_{\rM})$, and $\pi'\cong\ind_{N_{\rM'}(\mu_{\rM}')}^{\rM'}\tau_{\rM'}$, where $\tau_{\rM'}$ and $N_{\rM'}(\mu_{\rM})$ are defined as in Proposition \ref{Lproposition 3} of some maximal simple cuspidal $k$-type $(J_{\rM},\lambda_{\rM})$ of $\rM$.
\end{thm}

\begin{proof}
For the first assertion, we only need to verify the two condition of irreducibility. The second condition has been checked in \ref{Lproposition 3}. By \ref{Llemma 0.3} and \ref{Llemma 3}, we obtain that the induced $k$-representation $\ind_{N_{\rM'}(\mu_{\rM})}^{\rM'}\tau_{\rM'}$ is irreducible. Let $\pi'$ be the induced $k$-representation, and $\pi$ the $k$-irreducible representation as in \ref{prop 6}. We deduce from \ref{lem 9} and the fact that $\pi'$ contains $(J_{\rM}',\lambda_{\rM}')$, that $\pi$ contains $(J_{\rM},\lambda_{\rM}\otimes\chi\circ \det)$. Hence $\pi$ is cuspidal (\ref{cor 3}) and this implies that $\pi'$ is cuspidal. Conversely, let $\pi'$ be an irreducible cuspidal $k$-representation of $\rM'$, and $\pi$ be the irreducible cuspidal $k$-representation of $\rM$ which contains $\pi'$. Then there exists a maximal simple cuspidal $k$-type $(J_{\rM},\lambda_{\rM})$, and an extension $\Lambda_{\rM}$ of $\lambda_{\rM}$ to $N_{\rM}(\lambda_{\rM})$ such that $\pi\cong\ind_{N_{\rM}(\lambda_{\rM})}^{\rM}\Lambda_{\rM}$. Let $\mu=\ind_{J_{\rM}}^{\tilde{J}_{\rM}}\lambda_{\rM}$, and $N_{\rM}(\mu)$ be the normalizer of $\mu$ in $\rM$. By the transitivity of induction:
$$\pi\cong\ind_{N_{\rM}(\mu)}^{\rM}\circ\ind_{N_{\rM}(\lambda_{\rM})}^{N_{\rM}(\mu)}\Lambda_{\rM}.$$
Denote $\ind_{N_{\rM}(\lambda_{\rM})}^{N_{\rM}(\mu)}\Lambda_{\rM}$ as $\tau_{\rM}$, which is an irreducible representation containing $\mu$.

Till the end of this proof, we denote $\mu_{\rM}$ as a direct component of $\mu\vert_{\tilde{J}_{\rM}'}$, $N$ as $N_{\rM}(\mu)$, $N'$ as $N\cap \rM'$, and $N_{\rM'}$ as $N_{\rM'}(\mu_{\rM})$. Let $K$ be an open compact subgroup of $\tilde{J}_{\rM}$ contained in the kernel of $\tau_{\rM}$, and $Z$ be the center of $\rM$. Since the quotient $(Z\cdot N')\slash N$ is compact and the image of $K$ in this quotient is open, we deduce that $Z\cdot N'\cdot K$ is a normal subgroup with finite index of $N$. Hence the restriction $\res_{Z\cdot N'\cdot K}^{N}\tau_{\rM}$ is semisimple with finite length as in the first part of proof of \ref{prop 6}, from which we deduce that the restriction $\res_{N'}^{N}\tau_{\rM}$ is semisimple with finite length as well. After conjugate by an element $m$ in $\rM$, the cuspidal representation $\pi'$ contains a direct component of this restricted representation. We can assume that $m$ is identity, and denote this direct component as $\tau'$. Applying Frobenius reciprocity, the representation $\res_{\tilde{J}_{\rM}'}^{N}\ind_{\tilde{J}_{\rM}}^{N}\mu$ is semisimple, consisting of $\res_{\tilde{J}_{\rM}'}^{\tilde{J}_{\rM}}\mu$, hence $\tau'$ contains a $\mu_{\rM}$. Notice that $N_{\rM'}$ is a normal subgroup with finite index in $N'$. In fact, the group $N_{\rM'}$ contains $Z\cdot\tilde{J}_{\rM}'$. And as we have discussed after the proof of \ref{Llemma 3}, we could write $N'$ as a subgroup of $E_1^{\times}\tilde{J}_1\times\cdots\times E_r^{\times} \tilde{J}_r\cap\rM'=(E_1^{\times}\times\cdots\times E_r^{\times}\cap\rM')(\tilde{J}_{\rM}')$. Hence $\res_{N_{\rM'}}^{N'}\tau'$ is semisimple with finite length, and there must be one direct component $\tau_{\rM'}$ containing $\mu_{\rM}$. Since $\pi'$ contains $\tau_{\rM'}$, we have:
$$\pi'\cong\ind_{N_{\rM'}}^{\rM'}\tau_{\rM'}.$$
This ends the proof.
\end{proof}

\begin{defn}
\label{Ldefinition 20}
Let $(J_{\rM},\lambda_{\rM})$ be a maximal simple cuspidal $k$-type of $\rM$, and $\mu_{\rM}$ be an irreducible component of $\res_{\tilde{J}_{\rM}'}^{\tilde{J}_{\rM}}\tilde{\lambda}_{\rM}$, where $\tilde{J}_{\rM}$ and $\tilde{\lambda}_{\rM}$ are defined as in \ref{Ldefinition 21}. Let $N_{\rM'}(\mu_{\rM})$ be the normalizer group of $\mu_{\rM}$ in $\rM'$, and $\tau_{\rM'}$ an irreducible $k$-representation of $N_{\rM'}(\mu_{\rM})$ containing $\mu_{\rM}$. We define the couples in forms of $(N_{\rM'}(\mu_{\rM}),\tau_{\rM'})$ are the maximal simple cuspidal $k$-types of $\rM'$.
\end{defn}

\subsubsection{The $k$-representations $\pi$}
In this section $\pi'$ is an irreducible cuspidal $k$-representation of $\rM'$. We study the irreducible cuspidal $k$-representations $\pi$ of $\rM$, which contains $\pi'$ as a common component, and we prove that any two of them are different by a $k$-character of $\rM$ factor through determinant (Lproposition 9). This is the key to give the first description of supercuspidal support of $\pi'$ in the next section.

\begin{lem}
\label{Llemma 6}
Let $(J_{\rM},\lambda_{\rM})$ be a maximal simple cuspidal $k$-type of $\rM$, and $\mu=\ind_{J_{\rM}}^{\tilde{J}_{\rM}}\lambda_{\rM}$. Let $\tau$ be any irreducible $k$-representation of $N=N_{\rM}(\mu)$ containing $\mu$, then $\res_{N'}^{N}\tau$ is semisimple with finite length.
\end{lem}

\begin{proof}
By the definition of $N$, we know the center $Z$ of $\rM$ is contained in $N$. Since $F^{\times}/ \det{Z}$ is compact, the quotient group $\det{N}/\det{Z}$ is compact as well. Notice that $\tau\vert_{\tilde{J}}$ is a multiple of $\mu$, then any open subgroup contained in the kernel of $\mu$ is also contained in the kernel $\mathrm{ker}(\tau)$ of $\tau$, which implies $\mathrm{ker}(\tau)$ is open. Hence $Z\cdot N'\cdot \mathrm{ker}(\tau)$ is a normal subgroup with finite index in $N$. Applying Proposition \ref{propver13.1}, the restricted $k$-representation $\res_{Z\cdot N'}^{N}\tau$ is semisimple with finite length, and by Schur's lemma we deduce that $\res_{N'}^{N}\tau$ is semisimple with finite length.
\end{proof}

\begin{lem}
\label{Llemma 7}
If $c_1,c_2$ two characters of $Z$ and they coincide on $Z\cap\rM'$, where $Z$ denotes the center of $\rM$. Then $c_1\circ c_2^{-1}$ can be extended to a character on $\rM$ which factor through $\det$.
\end{lem}

\begin{proof}
First, we extend $c_1\circ c_2^{-1}$ to $Z\cdot\rM'$: For any $a\in Z,b\in\rM'$, define $c_0(ab)=c_1\circ c_2^{-1}(a)$. This is well defines, since for any $a',b'$ such that $a'b'=ab$, then $a^{-1}a'\in Z\cap\rM'$. Hence $c_1\circ c_2^{-1}(a^{-1}a')=1$, which implies $c_0(ab)=c_0(a'b')$. Now consider $\Ind_{Z\cdot\rM'}^{\rM}c_0$, which has finite length. There is a surjection from $\res_{Z\cdot\rM'}^{\rM}\Ind_{Z\cdot\rM'}^{\rM}c_0$ to $c_0$, then of which there exists an irreducible $k$-subquotient $c$ containing $c_0$, by the uniqueness of Jordan-H\"older factors. According to the fact that $\rM'$ is normal in $\rM$ and $c_0$ is trivial on $\rM'$, the $k$-representation $\res_{\rM'}^{\rM}\Ind_{Z\cdot\rM'}^{\rM}c_0$ is a trivial. Hence $c$ is trivial on $\rM'$ as well, and hence factor throught $F^{\times}\cong\rM/\rM'$. Then by Schur's lemma, $c$ is a character factor through $\det$.
\end{proof}

\begin{lem}
\label{Llemma 8}
Let $\tau_1,\tau_2$ be two irreducible $k$-representations of $N$ (notion as in \ref{Llemma 6}). Assume that $\res_{N'}^{N}\tau_1$ and $\res_{N'}^{N}\tau_2$ have one direct component in common, then there exists a $k$-quasicharacter of $F^{\times}$ such that $\tau_1\cong\tau_2\otimes\chi\circ\det$.
\end{lem}

\begin{proof}
The group $N$ is compact modulo center, and $\tilde{J}_{\rM}$ is the unique maximal open compact subgroup of $N$. Hence every irreducible $k$-representation of $N$ is finite dimensional, of which the kernel is always open. Let $U$ be an open compact subgroup contained in $\Ker{\tau_1}\cap\Ker{\tau_2}\cap\tilde{J}_{\rM}$. Let $c_1,c_2$ be the central characters of $\tau_1$ and $\tau_2$ respectively. According to \ref{Llemma 7}, there exists a $k$-quasicharacter $\chi$ of $F^{\times}$ such that $c_1\cong c_2\otimes\chi\circ\det$. After tensoring $\chi\circ\det$, we could assume that $c_1\cong c_2$. Hence:
$$\Hom_{Z\cdot N'\cdot U}(\res_{Z\cdot N'\cdot U}^{N}\tau_1,\res_{Z\cdot N'\cdot U}^{N}\tau_2)\neq 0.$$
Then:
$$\Hom_{N}(\tau_1,\ind_{Z\cdot N'\cdot U}^{N}\res_{Z\cdot N'\cdot U}^{N}\tau_2)\neq 0.$$
Since $\vert N:Z\cdot N'\cdot U\vert$ is finite, the later factor above has finite length and
$$\ind_{Z\cdot N'\cdot U}^{N}\res_{Z\cdot N'\cdot U}^{N}\tau_2\cong \tau_2\otimes\ind_{Z\cdot N\cdot U}^{N} 1.$$
Notice that any Jordan-Holder factor of $\ind_{Z\cdot N'\cdot U}^{N}1$ is a character factor through $\det\vert_{N}$, and $\vert F^{\times}: \det(N)\vert$ is finite. By the same reason as in the proof of \ref{Llemma 7}, we could extend each of them as a character of $\rM$ factor through $\det$. Hence there exists a $k$-quasicharacter $\chi$ of $F^{\times}$, such that $\tau_1\cong\tau_2\otimes\chi\circ\det$. 
\end{proof}

\begin{prop}
\label{Lproposition 9}
Let $\pi'$ be an irreducible cuspidal $k$-representation of $\rM'$. If $\pi_1,\pi_2$ two irreducible cuspidal $k$-representations of $\rM$, such that $\pi'$ appears as a direct component of $\res_{\rM'}^{\rM}\pi_1$ and $\res_{\rM'}^{\rM}\pi_2$ in common, then there exists a $k$-quasicharacter of $F^{\times}$ verifying that $\pi_1\cong \pi_2\otimes\chi\circ\det$.
\end{prop}

\begin{rem}
We will apply the proposition \ref{Lproposition 9} in the proof of the proposition \ref{Lproposition 10}, which is the first part of the uniqueness of supercuspidal support of $\mathrm{SL}_n(F)$. We will state two proofs of the proposition \ref{Lproposition 9} as below. The first proof is given through type theory while the second proof does not concern about type theory, which induce two parallel proofs of uniqueness of supercuspidal support of $\mathrm{SL}_n(F)$, with and without type theory respectively. The second proof is similar to that of the proposition in \S $\mathrm{VI.3.2.}$ in \cite{Re}, and also the proposition $\mathrm{2.4}$ in \cite{TA}.
\end{rem}

\begin{proof} [proof version 1]
Let $(J_{\rM},\lambda_{\rM})$ be a maximal simple cuspidal $k$-type of $\rM$ contained in $\pi_1$, and $\mu=\ind_{J_{\rM}}^{\tilde{J}_{\rM}}\lambda_{\rM}$. Then there is an extension $\tau$ of $\mu$ to $N=N_{\rM}(\mu)$ such that $\pi_1\cong\ind_{N}^{\rM}\tau$. Let $N'$ denote $N\cap \rM'$. As in the proof of \ref{Ltheorem 5}, there exists a direct component $\mu_{\rM}$ of $\res_{\tilde{J}_{\rM}'}^{\tilde{J}_{\rM}}\mu$ such that $\pi'\cong\ind_{N'}^{\rM'}\tau'$, where $\tau'$ is a direct component of $\res_{N'}^N\tau$ and $\tau'\cong\ind_{N_{\rM'}(\mu_{\rM})}^{N'}\tau_{\rM}'$. Here $\tau_{\rM}'$ is an irreducible $k$-representation containing $\mu_{\rM}$. By \ref{lem 9} $\pi_2$ contains $(J_{\rM},\lambda_{\rM}\otimes\chi_0\circ\det)$ for some $k$-quasicharacter $\chi_0$ of $F^{\times}$. Hence there is an extension $\Lambda_{\rM}$ of $\lambda_{\rM}$ on $N_{\rM}(\lambda_{\rM})$ such that $\pi_2\cong\ind_{N_{\rM}(\lambda_{\rM})}^{\rM}\Lambda_{\rM}\otimes\chi_0\circ\det$. Let $\tau_2$ denote $\ind_{N_{\rM}(\lambda_{\rM})}^{N}\Lambda_{\rM}\otimes\chi_0\circ\det$, which is an extension of $\mu\otimes\chi_0\circ\det$. After tensor $\chi_0^{-1}\circ\det$, we could assume that $\tau_2$ is an extension of $\mu$. Now we want to study the relation between $\tau$ and $\tau_2$.

First consider $\res_{N}^{\rM}\pi_2$:
$$\res_{N}^{\rM}\ind_{N}^{\rM}\tau_2\cong\oplus_{N\backslash \rM/N}\ind_{N\cap a(N)}^N\res_{N\cap a(N)}^{a(N)}a(\tau_2).$$
Since $[N:N\cap a(N)]$ is finite, the representation above is a direct sum of $k$-representations with finite length, and of which $\tau'$ is a sub-representation. Hence there exists an irreducible sub-quotient $\tau_1$ of $\res_{N}^{\rM}\pi_2$ such that $\tau'$ is a direct component of $\res_{N'}^N\tau_1$. By lemma \ref{Llemma 8}, there is a $k$-quasicharacter $\chi$ of $F^{\times}$ such that $\tau_1\cong\tau\otimes\chi\circ\det$.

We will prove that $\tau_1\cong\tau_2$. Assume that $\tau_1$ and $\tau_2$ are not isomorphic, then there exists $a\notin N$ such that $\tau_1$ is an irreducible subquotient of $\ind_{N\cap a(N)}^{N}\res_{N\cap a(N)}^{a(N)}a(\tau_2)$, which means $a$ weakly intertwines $\tau_1$ with $\tau_2$. Hence there exists $b\in N$ such that $ba$ weakly intertwines $\mu\otimes\chi\circ\det$ with $\mu$ and $c\in\tilde{J}_{\rM}, d\in cba(\tilde{J}_{\rM})$ such that $dcba$ weakly intertwines $\lambda_{\rM}\otimes\chi\circ\det$ with $\lambda_{\rM}$. Hence there is $g\in\tilde{J}_{\rM}$ such that $g(\lambda_{\rM}\otimes\chi\circ\det)\cong\lambda_{\rM}$. This implies that $\mu\otimes\chi\circ\det\cong\mu$. Then the element $ba$ weakly intertwines $\mu$ to itself. Then $\lambda_{\rM}$, as a subrepresentation of $\res_{J_{\rM}}^{\tilde{J}_{\rM}}\mu$, is a subquotient of:
$$\ind_{\tilde{J}_{\rM}\cap ba(\tilde{J}_{\rM})}^{\tilde{J}_{\rM}}\res_{\tilde{J}_{\rM}\cap ba(\tilde{J}_{\rM})}^{ba(\tilde{J}_{\rM})}\ind_{ba(J_{\rM})}^{ba(\tilde{J}_{\rM})}ba(\lambda_{\rM})$$
$$\cong\oplus_{ba(J_{\rM}) \backslash ba(\tilde{J}_{/rM})/ \tilde{J}_{\rM}\cap ba(\tilde{J}_{/rM})}\ind_{\tilde{J}_{\rM}\cap cba(J_{\rM})}^{cba(J_{\rM})} cba(\lambda_{\rM}).$$
Hence there is $c_0\in\tilde{J}_{\rM}$, such tha $bac_0 \in \mathrm{I}_{\rM}^w(\lambda_{\rM})=\mathrm{I}_{\rM'}(\lambda_{\rM})\subset N$, which is contradicted to our assumption that $a\notin N$. Hence $\tau_1\cong\tau_2$. We conclude that:
$$\pi_2\cong\ind_{N}^{\rM}\tau_1\cong(\ind_{N}^{\rM}\tau)\otimes\chi\circ\det\cong\pi_1\otimes\chi\circ\det.$$
\end{proof}

\begin{proof} [proof version 2]
The assumption implies that the set $\mathrm{Hom}_{\rM'}(\res_{\rM'}^{\rM}\pi_1,\res_{\rM'}^{\rM}\pi_2)$ is non-trivial. The group $\rM$ acts on this $\mathrm{Hom}$ set by
$$g\cdot f:=\pi_1(g)\circ f\circ\pi_2(g)^{-1},f\in\mathrm{Hom}_{\rM'}(\res_{\rM'}^{\rM}\pi_1,\res_{\rM'}^{\rM}\pi_2),g\in\rM.$$
This action factors through $\rM'$, hence induces an action of the abelian quotient group $\rM\backslash\rM'$ on this $\mathrm{Hom}$ set, which is a finitely dimensional $k$-vector space, since $\res_{\rM'}^{\rM}\pi_1$ and $\res_{\rM'}^{\rM}\pi_2$ are semisimple with finite length. Then the elements of $\rM\backslash\rM'$ forms a family of commutative linear operators on a finitely dimensional $k$-vector space, hence they have one common eigenvector. This is to say that there is an $k$-quasicharacter $\chi_0$ of $\rM\backslash\rM'$ such that $g\cdot f=\chi_0(g)f$ for each element $g\in\rM$, hence $\chi_0$ can be written as $\chi\circ\det$ for some $k$-quasicharacter $\chi$ of $F^{\times}$. Notice $f\in\mathrm{Hom}_{\rM'}(\pi_1\otimes\chi^{-1}\circ\det,\pi_2)$, by irreducibility, the $k$-representation $\pi_1\otimes\chi^{-1}\circ\det$ coincides with $\pi_2$.
\end{proof}

\subsubsection{First description of supercuspidal support}
Let $\pi$ and $\pi'$ be as in \ref{prop 6}. The supercuspidal support of $\pi$ is unique up to $\rM$-conjugate (\cite{V2}). We prove in this section, that the supercuspidal support of $\pi'$ is also unique up to $\rM$-conjugation (\ref{Lproposition 10}), which is the first description of supercuspidal support. Eventually, we will prove that the supercuspidal support of $\pi'$ is unique up to $\rM'$-conjugation in the next chapter.

\begin{lem}
\label{Llemma 12}
Let $\pi$ be an irreducible $k$-representation. If $\pi\otimes\chi\circ\det$ is supercuspidal for some $k$-quasicharacter $\chi$ of $F^{\times}$, then $\pi$ is supercuspidal.
\end{lem}

\begin{proof}
If $\pi\otimes\chi\circ\det$ is supercuspidal, then it contains a maximal simple cuspidal $k$-type $(J_{\rM},\lambda_{\rM})$. Hence $\pi$ contains $(J_{\rM},\lambda_{\rM}\otimes\chi^{-1}\circ\det)$, which is also a maximal simple cuspidal $k$-type. Hence $\pi$ is cuspidal. Now assume that there is a supercuspidal representation $\tau$ of some proper Levi $\rL$ of $\rM$ such that $\pi$ is an irreducible subquotient of $i_{\rL}^{\rM}\tau$. Then $\pi\otimes\chi\circ\det$ is a subquotient of $i_{\rL}^{\rM}\tau\otimes\chi\circ\det$. In fact, we have 
$$i_{\rL}^{\rM}\tau\otimes\chi\circ\det\cong (i_{\rL}^{\rM}\tau)\times\chi\circ\det.$$
To obtain the equivalence above, we could apply [\S I,5.2,d)]\cite{V1}, by noticing that for any parabolic subgroup containing $\rL$, its unipotent radical is a subset of the kernel of $\det$.
\end{proof}

\begin{lem}
\label{Llemma 11}
Let $\pi'$ be an irreducible cuspidal $k$-representation of $\rM'$, and $\pi$ an irreducible $k$-representation of $\rM$ containing $\pi'$. Then $\pi'$ is supercuspidal if and only if $\pi$ is supercuspidal.
\end{lem}

A similar result has been proved when $\pi'$ is cuspidal in Corollary \ref{cor 20}.

\begin{proof}
Applying \ref{Ltheorem 5}, there exists a maximal simple cuspidal $k$-type $(J_{\rM},\lambda_{\rM})$ and a direct component $\lambda_{\rM}'$ of $\lambda_{\rM}\vert_{\rM'}$, such that $\pi'$ contains $\lambda_{\rM}'$. Hence by \ref{lem 9}, the irreducible representation $\pi$ contains $\lambda_{\rM}\otimes\chi\circ\det$ for some $k$-quasicharacter $\chi$ of $F^{\times}$. Then by \S IV$1.2,1.3$ in \cite{V2} and \ref{cor 3}, this implies that $\pi$ is an irreducible cuspidal $k$-representation.

We assume that $\pi$ is non-supercuspidal, which means there exists a supercupidal representation $\tau$ of a proper Levi subgroup $\rL$ of $\rM$, the representation $\pi$ is a subquotient of the parabolic induction $i_{\rL}^{\rM}\tau$. Now by \S $5.2$ \cite{BeZe}, we obtain:
$$\res_{\rM'}^{\rM}i_{\rL}^{\rM}\tau\cong\ i_{\rL'}^{\rM'}\res_{\rL'}^{\rL}\tau.$$
There must be a direct component $\tau'$ of $\res_{\rL'}^{\rL}\tau$, and $\pi'$ be an irreducible subquotient of $i_{\rL'}^{\rM'}\tau'$. Hence $\pi'$ is not supercuspidal.

\end{proof}

\begin{prop}
\label{Lproposition 10}
Let $\pi'$ be an irreducible cuspidal $k$-representation of $\rM'$, and $\pi$ an irreducible cuspidal $k$-representation of $\rM$ such that $\pi$ contains $\pi'$. Let $[\rL,\tau]$ be the supercuspidal support of $\pi$, where $\rL$ is a Levi subgroup of $\rM$ and $\tau$ an irreducible supercuspidal $k$-representation of $\rL$. Let $\tau'$ be a direct component of $\res_{\rL'}^{\rL}\tau$. Any element in the supercuspidal support of $\pi'$ is contained in the $\rM$-conjugacy class of $(\rL',\tau')$.
\end{prop}

\begin{proof}
Let $\rL_0'$ be a Levi subgroup of $\rM'$ and $\tau_0'$ an irreducible supercuspidal $k$-representation of $\rL_0'$. Let $\tau_0$ be an irreducible $k$-representation of $\rL_0$ containing $\tau_0'$, hence $\tau_0$ is supercuspidal as in Lemma \ref{Llemma 11}.

If $\pi'$ is an irreducible subquotient of $i_{\rL_0'}^{\rM'}\tau_0'$. 
 By the same reason as in the proof of Lemma \ref{Llemma 11}, we know that there must be an irreducible subquotient of $i_{\rL_0}^{\rM}\tau_0$, noted as $\pi_0$, such that $\pi'$ is a direct component of $\res_{\rM'}^{\rM}\pi_0$. From \ref{Lproposition 9}, there exists a $k$-quasicharacter $\chi$ of $F^{\times}$ such that $\pi_0\cong\pi\otimes\chi\circ\det$. On the other hand, the supercuspidal support of $\pi\otimes\chi\circ\det$ is the $\rM$-conjugacy class of $(\rL,\tau\otimes\chi\circ\det)$. We assume that $\rL_0=\rL$ and $\tau_0\cong\tau\otimes\chi\circ\det$. Then $\tau_0'$ is a direct component of $\res_{\rL'}^{\rL}\tau\otimes\chi\circ\det\cong\res_{\rL'}^{\rL}\tau$.
\end{proof}

\section{Supercuspidal support}
\label{chapter 5}

\subsection{Uniqueness of supercuspidal support}
\subsubsection{The $n$-th derivative and parabolic induction}
\label{subsubsection 5.1.1}
Let $n_1,\ldots,n_m$ be a family of integers, and $\rM_{n_1,\ldots,n_m}$ denote the product $\mathrm{GL}_{n_1}\times\cdots\times\mathrm{GL}_{n_m}$, which can be canonically embedded into $\mathrm{GL}_{n_1+\cdots+n_m}$. Let $\rM_{n_1,\ldots,n_m}'$ denote $\rM_{n_1,\ldots,n_m}\cap\mathrm{SL}_{n_1+\cdots+n_m}$, and $P_{n_1}$ the mirabolic subgroup of $\mathrm{GL}_{n_1}$.

\begin{defn}
\label{ddef 1}
Let $n_1,\ldots,n_m$ be a family of positive integers, and $s\in\{1,\ldots,m\}$. We define:
\begin{itemize}
\item the mirabolic subgroup at place $s$ of $\rM_{n_1,\ldots,n_m}$, as $P_{(n_1,\ldots,n_m),s}=\mathrm{GL}_{n_1}\times\cdots\times\mathrm{GL}_{n_{s-1}}\times P_{n_s}\times\mathrm{GL}_{n_{s+1}}\times\cdots\times\mathrm{GL}_{n_m}$;
\item the mirabolic subgroup at place $s$ of $\rM_{n_1,\ldots,n_m}'$, as $P_{(n_1,\ldots,n_m),s}'=\mathrm{GL}_{n_1}\times\cdots\times\mathrm{GL}_{n_{s-1}}\times P_{n_s}\times\mathrm{GL}_{n_{s+1}}\times\cdots\times\mathrm{GL}_{n_m}\cap\rM_{n_1,\ldots,n_m}'$.
\end{itemize}
\end{defn}

For any $i\in\{1,\ldots,m\}$, let $\mathrm{U}_{n_i}$ be the subset of $\mathrm{GL}_{n_i}$, consisted with upper-triangular matrix with $1$ on the diagonal. We fix $\theta_{i}$ a non-degenerate character of $\rU_{n_i}$. It is clear that $\rU_{n_1,\ldots,n_m}=\rU_{n_1}\times\cdots\times\rU_{n_m}$ is a subgroup of $P_{(n_1,\ldots,n_m),s}$ and $P_{(n_1,\ldots,n_m),s}'$ for any $s\in\{1,\ldots,m\}$. Let $V_{n_s-1}$ denote the additive group of $k$-vector space with dimension $n_s-1$, which can be embedded canonically as a normal subgroup in $\rU_{n_1}\times\cdots\times\rU_{n_m}$. The subgroup $V_{n_s-1}$ is normal both in $P_{(n_1,\ldots,n_m),s}$ and $P_{(n_1,\ldots,n_m),s}'$, furthermore, we have $P_{(n_1,\ldots,n_m),s}=\rM_{n_1,\ldots,n_s-1,\ldots,n_m}\cdot V_{n_s-1}$ and $P_{(n_1,\ldots,n_m),s}'=\rM_{n_1,\ldots,n_s-1,\ldots,n_m}'\cdot V_{n_s-1}$.

Note $\gamma$ be any character of $\rU_{n_1}\times\cdots\times\rU_{n_m}$. For any $k$-representation $(E,\rho)\in\mathrm{Rep}_k(P_{(n_1,\ldots,n_m),s}')$, let $E_{s,\gamma}$ denote the subspace of $E$ generated by elements in form of $\rho(g)a-\gamma(h)a$, where $g\in V_{n_s-1},a\in E$. We define the coinvariants of $(E,\rho)$ according to $\theta$ as $E/E_{s,\gamma}$, and note it as $E(\gamma,s)$, and view $E(\gamma,s)$ as a $k$-representation of $\rM_{n_1,\ldots,n_s-1,\ldots,n_m}'$.

\begin{defn}
\label{ddef 2}
Fix a non-degenerate character $\theta$ of $\rU_{n_1}\times\cdots\times\rU_{n_m}$. 
\begin{itemize}
\item Let $(E,\rho)\in\mathrm{Rep}_k(P_{(n_1,\ldots,n_m),s}')$,
$$\Psi_{s}^{-}:\mathrm{Rep}_k(P_{(n_1,\ldots,n_m),s}')\rightarrow\mathrm{Rep}_k(\rM_{n_1,\ldots,n_s-1,\ldots,n_m}'),$$
which maps $E$ to $E(\mathds{1},s)$;
\item Let $(E,\rho)\in\mathrm{Rep}_k(\rM_{n_1,\ldots,n_s-1,\ldots,n_m}')$,
$$\Psi_{s}^{+}:\mathrm{Rep}_k(\rM_{n_1,\ldots,n_s-1,\ldots,n_m}')\rightarrow\mathrm{Rep}_k(P_{(n_1,\ldots,n_m),s}').$$
Write $P_{(n_1,\ldots,n_m),s}'=\rM_{n_1,\ldots,n_s-1,\ldots,n_m}' \cdot V_{n_s-1}$. Define $\Psi_{s}^{+}(E,\rho)=(E,\Psi^{+,s}(\rho))$ by $\Psi_{s}^{+}(\rho)(mg)(a)=\rho(m)(a)$, for any $m\in\rM_{n_1,\ldots,n_s-1,\ldots,n_m}',g\in V_{n_s-1}$ and $a\in E$;
\item Let $(E,\rho)\in\mathrm{Rep}_k(P_{(n_1,\ldots,n_m),s}')$,
$$\Phi_{\theta,s}^{-}:\mathrm{Rep}_k(P_{(n_1,\ldots,n_m),s}')\rightarrow\mathrm{Rep}_k(P_{(n_1,\ldots,n_s-1,\ldots,n_m),s}'),$$
which maps $E$ to $E(\theta,s)$;
\item Let $(E,\rho)\in\mathrm{Rep}_k(P_{(n_1,\ldots,n_s-1,\ldots,n_m),s}')$,
$$\Phi_{\theta,s}^{+}:\mathrm{Rep}_k(P_{(n_1,\ldots,n_s-1,\ldots,n_m),s}')\rightarrow\mathrm{Rep}_k(P_{(n_1,\ldots,n_m),s}'),$$
by $\Phi_{\theta,s}^{+}(\rho)=\ind_{P_{(n_1,\ldots,n_s-1,\ldots,n_m),s}'\cdot V_{n_s-1}}^{P_{(n_1,\ldots,n_m),s}'}\rho_{\theta}$, where $\rho_{\theta}(pg)(a)=\theta(g)\rho(p)(a)$, for any $p\in P_{(n_1,\ldots,n_s-1,\ldots,n_m),s}',g\in V_{n_s-1}$ and $a\in E$.
\end{itemize}
\end{defn}

\begin{rem}
\label{drem 3}
By the reason that for any $m\in\mathbb{Z}$ the group $V_m$ is a limite of pro-$p$ open compact subgroups, the four functors defined above are exact. In the definition of $\Phi_{\theta,s}^{+}$, we view $P_{(n_1,\ldots,n_s-1,\ldots,n_m),s}'$ as a subgroup of $P_{(n_1,\ldots,n_m),s}'$.
\end{rem}

The notion of derivatives is well defined for $k$-representations of $\rG$, now we consider the parallel operator of derivatives for Levi subgroups of $\rG'$.
\begin{defn}
\label{ddef 4}
We fix a non-degenerate character $\theta$ of $\rU_{n_1}\times\cdots\times\rU_{n_m}$. Let $(E,\rho)\in\mathrm{Rep}_k(P_{(n_1,\ldots,n_m,s)}')$, for any interger $s\in\{1,\ldots,m\}$ and $1\leq d\leq n_1+\ldots+n_s$, we define the derivative $\rho_{\theta,s}^{(d)}$:
\begin{itemize}
\item when $1\leq d\leq n_s$, $\rho_{\theta,s}^{(d)}=\Psi_s^{-}\circ(\Phi_{\theta,s}^-)^{d-1}\rho$;
\item when $n_s+1\leq d=n_s+\ldots+n_{s-l}+n'$, where $0\leq l\leq s-1$ and $1\leq n'\leq n_{s-l-1}$, then $\rho_{\theta,s}^{(d)}=\Psi_{s-l-1}^{-}\circ(\Phi_{\theta,s-l-1})^{n'-1}\circ(\Phi_{\theta,s-l})^{n_{s-l}-1}\circ\ldots\circ(\Phi_{\theta,s}^-)^{n_s-1}\rho$
\end{itemize}
\end{defn}

\begin{defn}
\label{ddef 5}
To simplify our notations, we need to introduce $\ind_{m}^{m-1}:\mathrm{Rep}_{k}(\rG_1)\rightarrow\mathrm{Rep}_k(\rG_2)$ according to different cases:
\begin{itemize}
\item When $\rG_1=\rM_{n_1,\ldots,n_m}'$ and $\rG_2=\rM_{n_1,\ldots,n_{m-1}+n_m}'$, we embed $\rG_1$ into $\rG_2$ as in the figure case $\mathrm{I}$, and $\ind_{m}^{m-1}$ is defined as $i_{\rU,1}$, and the later one is defined as in \S$1.8$ of \cite{BeZe};
\item When $\rG_1=P_{(n_1,\ldots,n_m),m}'$ and $\rG_2=P_{(n_1,\ldots,n_{m-1}+n_m),m-1}'$, we embed $\rG_1$ into $\rG_2$ as in the figure case $\mathrm{II}$, and $\ind_{m}^{m-1}$ is defined as $i_{\rU,1}$;
\item When $\rG_1=P_{(n_1,\ldots,n_m),m-1}'$ and $\rG_2=P_{(n_1,\ldots,n_{m-1}+n_m),m-1}'$, we embed $\rG_1$ into $\rG_2$ as in the figure case $\mathrm{III}$, and $\ind_{m}^{m-1}$ is defined as $i_{\rU,1}\circ\varepsilon$. Here $\varepsilon$ is a character of $P_{(n_1,\ldots,n_m),m-1}'$. Write $g\in P_{(n_1,\ldots,n_m),m-1}'\subset\rM_{n_1,\ldots,n_m}'$ as $(g_1,\ldots,g_m)$, define $\varepsilon(g)=\vert \det(g_m) \vert$, the absolute value of $\det(g_m)$. This $k$-character is well defined since $p\neq l$.
\end{itemize}
\end{defn}

\begin{figure}
\centering
\begin{tikzpicture}
\draw[fill=gray] (0,1) rectangle (1,0); 
\draw[fill=gray] (1,0) rectangle (3,-2);
\draw[fill=gray] (3,-2) rectangle (5,-4);
\draw[fill=white] (1,-2) rectangle (3,-4) ;
\draw [fill=white] (3,0) rectangle (5,-2);
\draw[decorate,decoration={brace,amplitude=7pt}]  (1,-2)--(1,0);
\draw[decorate,decoration={brace,amplitude=7pt}]  (1,-4)--(1,-2);
\node at (0.3,-1){$n_{m-1}$};
\node at (0.3,-3){$n_m$};
\node at (2,-3){$0$};
\node at (4,-1){$U$};
\end{tikzpicture}
\caption{Case I}

\centering
\begin{tikzpicture}
\draw[fill=gray] (0,1) rectangle (1,0);
\draw[fill=gray] (1,0) rectangle (3,-2);
\draw[fill=white] (3,0) rectangle (5,-2);
\draw[fill=white] (1,-2) rectangle (3,-4);
\draw[fill=gray] (3,-2) rectangle (5,-3.5);
\draw[fill=white] (3,-3.5) rectangle (4.5,-4);
\draw[fill=gray] (4.5,-3.5) rectangle (5,-4);
\draw[decorate,decoration={brace,amplitude=7pt}]  (1,-2)--(1,0);
\draw[decorate,decoration={brace,amplitude=7pt}]  (5,-2)--(5,-3.5);
\draw[decorate,decoration={brace,amplitude=3pt}]  (5,-3.5)--(5,-4);
\node at (0.3,-1){$n_{m-1}$};
\node at (5.85,-2.75) {$n_m-1$};
\node at (5.25,-3.75) {$1$};
\node at (2,-3){$0$};
\node at (4,-1){$U$};
\node at (4.75,-3.75) {$1$};
\node at (3.75,-3.75) {$0$};
\end{tikzpicture}
\caption{Case II}
\end{figure}

\begin{figure}
\centering
\begin{tikzpicture}
\draw[fill=gray] (0,1) rectangle (1,0); 
\draw[fill=gray] (1,0) rectangle (2.5,-1.5);
\draw[fill=gray] (2.5,-1.5) rectangle (4.5,-3.5);
\draw[fill=gray] (4.5,-3.5) rectangle (5,-4);
\draw[fill=white] (1,-1.5) rectangle (2.5,-3.5);
\draw[fill=white] (2.5,0) rectangle (4.5,-1.5);
\draw[fill=gray] (4.5,0) rectangle (5,-1.5);
\draw[fill=white] (4.5,-1.5) rectangle (5,-3.5);
\draw[fill=white] (1,-3.5) rectangle (2.5,-4);
\draw[fill=white] (2.5,-3.5) rectangle (4.5,-4);
\draw[decorate,decoration={brace,amplitude=7pt}]  (1,-1.5)--(1,0);
\draw[decorate,decoration={brace,amplitude=7pt}]  (5,-1.5)--(5,-3.5);
\draw[decorate,decoration={brace,amplitude=3pt}]  (5,-3.5)--(5,-4);
\node at (0.01,-0.75) {$n_{m-1}-1$};
\node at (5.6,-2.5) {$n_m$};
\node at (5.25,-3.75) {$1$};
\node at (4.75,-3.75) {1};
\node at (3.5,-0.75) {$U$};
\node at (1.75,-2.5) {$0$};
\node at (1.75,-3.75) {$0$};
\node at (3.5,-3.75) {$0$};
\node at (4.75,-2.5) {$0$};
\end{tikzpicture}
\caption{Case III}
\end{figure}
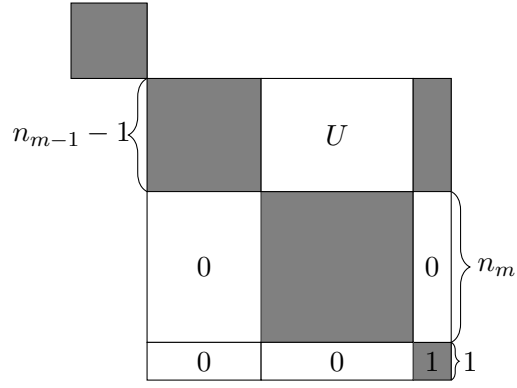

\begin{prop}
\label{dprop 5}
Assume that $\rho_1\in\mathrm{Rep}_k(\rM_{n_1,\ldots,n_m}'), \rho_2\in\mathrm{Rep}_k(P_{(n_1,\ldots,n_m),m}')$, and $\rho_3\in\mathrm{Rep}_k(P_{(n_1,\ldots,n_m),m-1}')$. The functor $\ind_{m}^{m-1}$ is defined as in \ref{ddef 5} according to different cases.
\begin{enumerate}
\item 
In $\mathrm{Rep}_k(P_{(n_1,\ldots,n_{m-1}+n_m),m-1}')$, there exists an exact sequence:
$$0\rightarrow\ind_m^{m-1}(\rho_1\vert_{P_{m,m-1}'})\rightarrow(\ind_m^{m-1}\rho_1)\vert_{P_{m-1,m-1}'}\rightarrow \ind_m^{m-1}(\rho_1\vert_{P_{m,m}'})\rightarrow 0,$$
where $P_{m,m-1}'=P_{(n_1,\ldots,n_m),m-1}',P_{m-1,m-1}'=P_{(n_1,\ldots,n_{m-1}+n_m),m-1}'$, and $P_{m,m}'=P_{(n_1,\ldots,n_m),m}'$.
\item When $2\leq m$, let $\dot{\theta}$ be a non-degenerate character of $\rU_{n_1}\times\cdots\times\rU_{n_{m-2}}\times\rU_{n_{m-1}+n_m}$, such that $\dot{\theta}\vert_{U_{n_1}\times\cdots\times \rU_{n_m}}\cong\theta$. We have equivalences:
\begin{itemize}
\item $\ind_{m}^{m-1}\circ\Psi_m^{-}\rho_2\cong\Psi_{m-1}^{-}\circ\ind_{m}^{m-1}\rho_2;$
\item $\ind_{m}^{m-1}\circ\Phi_{\theta,m}^{-}\rho_2\cong\Phi_{\dot{\theta},m-1}^{-}\circ\ind_{m}^{m-1}\rho_2.$
\end{itemize}
\item We have an equivalence:
$$\Psi_{m-1}^{-}\circ\ind_{m}^{m-1}\rho_3\cong\ind_{m}^{m-1}\circ\Psi_{m-1}^{-}\rho_3,$$
and an exact sequence:
$$0\rightarrow\ind_{m}^{m-1}\circ\Phi_{\theta,m-1}^{-}\rho_3\rightarrow\Phi_{\dot{\theta},m-1}^-\circ\ind_{m}^{m-1}\rho_3\rightarrow\ind_{m}^{m-1}((\Psi_{m-1}^{-}\rho_3)\vert_{P'})\rightarrow 0,$$
where $P'=P_{(n_1,\ldots,n_{m-1}-1,n_m),m}'$.
\end{enumerate}
\end{prop}

\begin{proof}
As proved in the Appendix, Theorem 5.2 in \cite{BeZe} holds for $k$-representations of $\rM'$. And let $n=n_1+\ldots+n_m$

For $(1)$: Let $\rM'=\rM_{n_1,\ldots,n_m}'$ be embedded into $\rG'=\rM_{n_1,\ldots,n_{m-1}+n_m}'$ as in definition \ref{ddef 5}, figure $\mathrm{I}$. Define functor $\mathrm{F}$ as $\mathrm{F}(\rho_1)=\rho_1\vert_{P_{(n_1,\ldots,n_{m-1}+n_m),m-1}'}$, where the functor $\mathrm{F}$ is defined as in 5.1 \cite{BeZe} in the following situation: 
$$\rU=\rU_{n_{m-1}},\vartheta=1,\rN=P_{(n_1,\ldots,n_{m-1}+n_m),m-1}',\rV=\{ e\}.$$
To compute $\mathrm{F}$, we apply theorem 5.2 \cite{BeZe}. Condition $(1),(2)$ and $(\ast)$ from 5.1 \cite{BeZe} hold trivially. Let $\rT$ be the group of diagonal matrix, the $\mathrm{Q}$-orbits on $X=\mathrm{P}\backslash \rG$ is actually the $\rT\cdot\rN$-orbits, and the group $\rT\cdot\rN$ is a parabolic subgroup. By Bruhat decomposition $\rT\cdot\rN$ has two orbits: the closed orbit $Z$ of point $\mathrm{P}\cdot e\in X$ and the open orbit $Y$ of the point $\mathrm{P}\cdot\omega^{-1}\in X$, where $\omega$ is the matrix of the cyclic permutation $sgn(\sigma)\mathds{1}_{n_m}\cdot\sigma$, where 
$$\sigma=(N_1+\cdots+n_{m-1}\rightarrow n\rightarrow n-1\rightarrow\cdots\rightarrow n_1+\cdots+n_{m-1}),$$
and $sgn(\sigma)\mathds{1}_{n_m}$ denote an element in $\rM_{n_1,\ldots,n_m}'$, which equals to identity on the first $m-1$ blocs, and $sgn(\sigma)$ times identity on the last bloc, and $sgn(\sigma)$ denote the signal of $\sigma$. Now we consider condition $(4)$ from 5.1 \cite{BeZe}:
\begin{itemize}
\item Since $V=\{ e\}$, it is clear that $\omega(\mathrm{P}),\omega(\rM)$ and $\omega(U)$ are decomposable with respect to $(\rN,\rV)$;
\item Let us consider $\omega^{-1}(\mathrm{Q})=\omega^{-1}(\rN)$.
\end{itemize}
To study the intersection $\omega^{-1}(\rN)\cap (\rM\cdot\rU)$, first we consider the Levi subgroup $\rM_{n_1,\ldots,n_{m-1}+n_m-1,1}'$ and the corresponding standard parabolic subgroup 
$$\mathrm{P}'=\rM_{n_1,\ldots,n_{m-1}+n_m-1,1}'\cdot \rV_{n_{m-1}+n_m-1},$$
where $\rV_{n_{m-1}+n_m-1}$ denotes the unipotent radical of $\mathrm{P}'$. We have $\mathrm{N}\subset \mathrm{P}'$, hence $\omega^{-1}(\mathrm{N})\subset\omega^{-1}(\mathrm{P}')$. As in 6.1 of \cite{BeZe}, after fix a system $\Omega$ of roots, and denote $\Omega^{+}$ the set of positive roots. Then by Proposition in 6.2 \cite{BeZe}, we could write $\omega^{-1}(\mathrm{P}')=\rG(\mathcal{S})$ and $\mathrm{P}=\rG(\mathcal{P}),\rU=\rU(\mathcal{M})$ in the manner as in 6.1\cite{BeZe}, where $\mathcal{S},\mathcal{P}$ and $\mathcal{M}$ are convex subset of $\Omega$. So by Proposition in 6.1 \cite{BeZe}, we have:
$$\omega^{-1}(\mathrm{P}')\cap \mathrm{P}=\rG(\mathcal{S}\cap\mathcal{P});$$
$$\omega^{-1}(\mathrm{P}')\cap\rU=\rU(\mathcal{S}\cap\mathcal{P}\backslash\mathcal{M});$$
$$\omega^{-1}(\mathrm{P}')\cap\rM=\rG(\mathcal{S}\cap\mathcal{M}).$$
Hence 
$$\omega^{-1}(\mathrm{P}')\cap\mathrm{P}=(\omega^{-1}(\mathrm{P}')\cap\rM)\cdot(\omega^{-1}(\mathrm{P}')\cap\rU).$$
Notice that $\omega^{-1}(\mathrm{P}')\cap\rU=\omega^{-1}(\rN)\cap\rU$, we deduce that:
$$\omega^{-1}(\rN)\cap\mathrm{P}=(\omega^{-1}(\rN)\cap\rM)\cdot(\omega^{-1}(\rN)\cap\rU).$$
In the formula of $\Phi_{Z}$ in 5.2 \cite{BeZe}, since $\rU\cap\omega^{-1}(\rN)=\rU$, the characters $\varepsilon_1=\varepsilon_2=1$. Hence we obtain the exact sequence desired.

For $(2)$. In this part, the functor $\ind_{m}^{m-1}$ is always defined as the case $\mathrm{II}$ in \ref{ddef 5}. First we consider the case $\Psi_m^-$. Define functor $\mathrm{F}$ as $\Psi_{m-1}^-\circ\ind_{m}^{m-1}$. We write $\mathrm{F}$ as in \S $5.1$ \cite{BeZe} in the situation:
$$\rG=P_{(n_1,\ldots,n_{m-1}+n_m),m-1}',\rM=P_{(n_1,\ldots,n_m),m}',$$
$$\rN=\rM_{n_1,\ldots,n_{m-1}+n_m-1}',\mathrm{V}=\rV_{n_{m-1}+n_m-1},$$
and $\rU$ are defined in the \ref{ddef 5} case $\mathrm{II}$. Condition $(1)$ and $(2)$ of \S 5.1 \cite{BeZe} is clear. Since $\mathrm{Q}=\rG$, and there is only one $\mathrm{Q}$-orbit on $X=\mathrm{P} \backslash\rG$, conditions $(3),(4)$ hold trivially. Thus we obtain the equivalence:
$$\ind_{m}^{m-1}\circ\Psi_m^{-}\rho_2\cong\Psi_{m-1}^{-}\circ\ind_{m}^{m-1}\rho_2.$$

For the case $\Phi_{\dot{\theta},m-1}^-$: define functor $\mathrm{F}$ as $\Phi_{\dot{\theta},m-1}^-\circ\ind_{m}^{m-1}$. We write $\mathrm{F}$ as in \S $5.1$ \cite{BeZe} in the situation:
$$\rG=P_{(n_1,\ldots,n_{m-1}+n_m),m-1}',$$
$$\rN=P_{(n_1,\ldots,n_{m-1}+n_m-1),m-1}',\mathrm{V}=\mathrm{V}_{n_{m-1}+n_m-1},$$
and $\rM,\rU$ are defined as the case $\mathrm{II}$ of \ref{ddef 5}. Conditions $(1),(2)$ of \S $5.1$ \cite{BeZe} hold, and 
$$\mathrm{P}\backslash\rG/ \mathrm{Q}\cong \rM_{n_1,\ldots,n_{m-1},n_m-1}'\backslash\rM_{n_1,\ldots,n_{m-1}+n_m-1}'/ P_{(n_1,\ldots,n_{m-1}+n_m-1),m-1}'.$$
Hence as proved in $(1)$, the group $\mathrm{Q}$ has two orbits on $X=\mathrm{P}\backslash\rG$: the closed orbit of $\mathrm{P}\cdot e$ and the open orbit $\mathrm{P}\cdot\omega_0^{-1}$, where $\omega_0$ is the matrix $sgn(\sigma_0)\mathds{1}_{n_m-1}\cdot\sigma_0$. The matrix $\sigma_0$ corresponding to the cyclic permutation:
$$(n_1+\cdots+n_{m-1}\rightarrow n\rightarrow n-1\rightarrow \cdots\rightarrow n_1+\cdots+n_{m-1}).$$
Now we check the condition $(4)$ of \S $5.1$ \cite{BeZe}. Since
$$\mathrm{P}=\rM_{n_1,\ldots,n_{m-1}+n_m}'\cdot\mathrm{V},$$
$$\rM=\rM_{n_1,\ldots,n_m-1}'\cdot\mathrm{V}_{n_m-1},$$
and $\omega_0(\mathrm{V})=\mathrm{V}$, hence $\omega_0(\mathrm{P})$ and $\omega_0(\rM)$ are decomposable with respect to $(\rN,\rV)$. We deduce that $\omega_0(\rU)$ is decomposable with respect to $(\rN,\rV)$ by noticing that 
$\omega_0(\rU)=(\omega_0(\rU)\cap\rN)\cdot(\omega_0(\rU)\cap\rV)$. Now consider $\omega_0^{-1}(\mathrm{Q})$, $\omega_0^-(\rN)$ and $\omega_0^-(\rV)$. Since $\omega_0^{-1}(\rV)=\rV$, which is decomposable with respect to $(\rM,\rU)$ clearly. Notice that $\omega_0^-(\rV)\cap\mathrm{P}=\omega_0^-(\rV)$, and $\omega_0^-(\rN)$ is decomposable with respect to $(\rM,\rU)$ by $(1)$. We deduce that $\omega_0^-(\mathrm{Q})$ is decomposable with respect to $(\rM,\rU)$. And the condition $(\ast)$ does not hold for the orbit $\mathrm{P}\cdot\sigma_0$. Then by \S $5.2$ \cite{BeZe}, we obtain the equivalence
$$\ind_{m}^{m-1}\circ\Phi_{\theta,m}^{-}\cong\Phi_{\dot{\theta},m-1}^{-}\circ\ind_{m}^{m-1},$$
for every $\rho_2\in\mathrm{Rep}_k(P_{(n_1,\ldots,n_m),m}')$.

For part $(3)$. In the case of $\mathrm{F}=\Psi_{m-1}^-\circ\ind_m^{m-1}$, we have (in the manner of \S $5.1$ \cite{BeZe}):
$$\rG=P_{(n_1,\ldots,n_m),m}',$$
$$\rN=\rM_{n_1,\ldots,n_m-1}',\mathrm{V}=\mathrm{V}_{n_m-1},$$
and $\rM,\rU$ as in \ref{ddef 5} case $\mathrm{II}$. There is only one $\mathrm{Q}$-orbit on $\mathrm{P}\backslash\rG$, and condition $(1)-(4)$ and $(\ast)$ in \S $5.1$ \cite{BeZe} hold. Notice that $\varepsilon\circ\Psi_{m-1}^-\cong\Psi_{m-1}^-$ ($\varepsilon$ is defined in \ref{ddef 5}). After applying theorem $5.2$ of \cite{BeZe}, we obtain the equivalence:
$$\Psi_{m-1}^-\circ\ind_{m}^{m-1}\rho_3\cong\ind_{m}^{m-1}\circ\Psi_{m-1}^-\rho_3.$$

For the case $\mathrm{F}=\Phi_{\dot{\theta},m-1}^-\circ\ind_m^{m-1}$. We have (in the manner of \S $5.1$ \cite{BeZe}):
$$\rG=P_{(n_1,\ldots,n_m),m}',$$
$$\rN=P_{(n_1,\ldots,n_{m-1}+n_m-1),m-1}',\mathrm{V}=\mathrm{V}_{n_{m-1}+n_m-1},$$
and $\rM,\rU$ are defined as the case $\mathrm{II}$ of \ref{ddef 5}. As in the proof of part $(2)$, the group $\mathrm{Q}$ has two orbits on $\mathrm{P}\backslash\rG$: the closed one $\mathrm{P}\cdot e$ and the open one $\mathrm{P}^{-1}\cdot \omega_0$. The condition $(4)$ can be justified as part $(2)$, and condition $(\ast)$ is clear since $\omega_0(\rU)\cap\mathrm{V}=\mathds{1}$. Now we apply theorem $5.2$ \cite{BeZe}. The functor corresponds to the orbit $\mathrm{P}\cdot e$ is $\ind_{m}^{m-1}\circ\Phi_{\theta,m-1}^-$ by noticing $\varepsilon\circ\Phi_{\theta,m-1}^-\cong\Phi_{\theta,m-1}\circ\varepsilon$. Now we consider the functor corresponds to the orbit $\mathrm{P}\cdot\omega_0^{-1}$. Following the notation as \S $5.1$ \cite{BeZe}, the character $\psi'=\omega_0^{-1}(\psi)\vert_{\rM\cap\omega_0^{-1}(\mathrm{V})}$ is trivial. The character $\varepsilon_1$ is trivial, and $\varepsilon_2\cong\varepsilon^{-1}$. Hence the functor corresponded to the fixed orbit is
$$\ind_{m}^{m-1}\circ\res_{P'}^{\rM_{n_1,\ldots,n_{m-1},n_m}'}\circ\Psi_{m-1}^-,$$
from which we deduce the exact sequence desired.
\end{proof}

\begin{cor}
\label{dcor 6}
\begin{enumerate}
\item Let $\rho\in\mathrm{Rep}_k(P_{(n_1,\ldots,n_m),m}')$. Assume that $1\leq i\leq n_m$, then $(\ind_{m}^{m-1}\rho)_{\dot{\theta}, m-1}^{(i)}\cong\ind_{m}^{m-1}\rho_{\theta,m}^{(i)}$;
\item Let $\rho\in\mathrm{Rep}_k(P_{(n_1,\ldots,n_m),m-1}')$. Assume that $1\leq i\leq n_{m-1}+n_m$, then $(\ind_{m}^{m-1}\rho)_{\dot{\theta},m-1}^{(i)}$ is filtrated by
$\ind_{m}^{m-1}((\rho_{\theta,m}^{(i-j)})_{\theta,m-1}^{(j)})$, where $i-n_m\leq j\leq i$;
\item Let $\rho\in\mathrm{Rep}_k(\rM_{n_1,\ldots,n_m}')$. Assume that $i\geq 0$, then $(\ind_{m}^{m-1}\rho)_{\dot{\theta},m-1}^{(i)}$ is filtrated by $\ind_{m}^{m-1}((\rho_{\theta,m}^{(i-j)})_{\theta,m-1}^{(j)})$, where $i-n_m\leq j\leq i$;
\item Let $\rho\in\mathrm{Rep}_k(\rM_{n_1,\ldots,n_m}')$, there is an equivalence:
$$(\ind_2^1\circ\cdots\circ\ind_{m-1}^{m-2}\circ\ind_{m}^{m-1}\rho)_{\dot{\theta},1}^{(n_1+\cdots+n_m)}\cong(\cdots((\rho_{\theta,m}^{(n_m)})_{\theta,m-1}^{(n_{m-1})})\cdots)_{\theta,1}^{(n_1)}.$$
\end{enumerate}
\end{cor}

\begin{proof}
Part $(1)$ follows from the exactness of $\Phi_{\theta,m}^-,\Psi_m^-$ and \ref{dprop 5} $(2)$; $(2)$ from $(1)$ and \ref{dprop 5} $(3)$, $(3)$ from $(1),(2)$ and \ref{dprop 5} $(1)$. Part $(4)$ follows from $(3)$, by noticing that
$$\ind_2^1\circ\cdots\circ\ind_{m-1}^{m-2}\circ\ind_{m}^{m-1}\rho\cong i_{\rM_{n_1,\ldots,n_m}}^{\mathrm{GL}_{n_1+\cdots+n_m}}\rho.$$
In fact, this is the transitivity of parabolic induction.
\end{proof}

\subsubsection{Uniqueness of supercuspidal support}
\begin{prop}
\label{Uprop 1}
Let $\tau\in\mathrm{Rep}_{k}(\rM_{n_1,\ldots,n_m}')$, and $\theta$ a non-degenerate character of $\rU_{n_1,\ldots,n_m}$. Then $\tau_{\theta,m}^{(n_1+\ldots+n_m)}\neq 0$ is equivalent to say that $\Hom_{k[\rU_{n_1,\ldots,n_m}]}(\tau,\theta)\neq 0$. In particular, this is equivalent to say that $(\rU_{n_1,\ldots,n_m},\theta)$-coinvariants of $\tau$ is non-trivial.
\end{prop}

\begin{proof}
In this proof, we use $\rU$ to denote $\rU_{n_1,\ldots,n_m}$. For the first equivalence, notice that $\Phi_{\theta,m}^-(\tau)\neq 0$ is equivalent to say that $(V_{n_m-1},\theta)$-coinvariants of $\tau$ is non-trivial. For $1\leq s\leq n_m-1$, let $V_{s}$ denote the subgroup of $\rU$ consisting with the matrices with non-zero coefficients only on the $(s+1)$-th line and the diagonal. Let $W$ denote the representation space of $\tau$. The space of $\tau_{\theta,m}^{(n_1+\cdots+n_m)}$ is isomorphic to the quotient of $W$ by the subspace $W_{\theta}$ generated by $g_s(w)-\theta(g_s)w$, for every $s$ and $g_s\in V_{s}$, $w\in W$. Meanwhile, since the subgroups $V_{s}$'s generate $\rU$, and $\theta$ is determined by $\theta\vert_{V_{s}}$ while considering every $s$, the subspace $W_{\theta}$ of $W$ is isomorphic to the subspace generated by $g(w)-\theta(g)w$, where $g\in\rU$. Hence $\tau_{\theta,m}^{(n_1+\cdots+n_m)}\neq 0$ is equivalent to say that $(\rU_{n_1,\ldots,n_m},\theta)$-coinvariants of $\tau$ is non-trivial. The second equivalence is clear, since the $(\rU,\theta)$-coinvariants of $\tau$ is the largest quotient of $\tau$ such that $\rU$ acts as a multiple of $\theta$. 
\end{proof}

\begin{prop}
\label{Uprop 2}
Let $\tau\in\mathrm{Rep}_{k}(\rM_{n_1,\ldots,n_m}')$, and $\rho$ be a subquotient of $\tau$. Let $\theta$ be a non-degenerate character of $\rU_{n_1,\ldots,n_m}$, and $\rho_{\theta,m}^{(n_1+\cdots+n_m)}$ is non-trivial, then $\tau_{\theta,m}^{(n_1+\cdots+n_m)}$ is non-trivial.
\end{prop}

\begin{proof}
We consider the $(n_1+...+n_m)$-th derivative functor corresponding to the non-degenerate character $\theta$, from the category $\mathrm{Rep}_{k}(\rM_{n_1,\ldots,n_m}')$ to the category of $k$-vector spaces, which maps $\tau$ to $\tau_{\theta,m}^{(n_1+\cdots+n_m)}$. By Definition \ref{ddef 4} and Remark \ref{drem 3}, this functor is a composition of functors $\Psi_{\cdot}^{-}$ and $\Phi_{\theta,\cdot}^{-}$, hence is exact. Let $\rho_0$ be a sub-representation of $\tau$ such that $\rho$ is a quotient representation of $\rho_0$. The exactness of derivative functor implies first that $\rho_{0_{\theta,m}}^{n_1+...+n_m}\neq 0$, and apply again the exactness we conclude that $\rho_{\theta,m}^{n_1+...+n_m}\neq 0$.
\end{proof}

\begin{thm}
\label{Uthm 3}
Let $\rM'$ be a Levi subgroup of $\rG'$, and $\rho$ an irreducible $k$-representation of $\rM'$. The supercuspidal support of $\rho$ is a $\rM'$-conjugacy class of one unique supercuspidal pair.
\end{thm}

\begin{proof}
Since the cuspidal support of irreducible $k$-representation is unique, to prove the uniqueness of supercuspidal support, it is sufficient to assume that $\rho$ is cuspidal. Let $\pi$ be an irreducible cuspidal $k$-representation of $\rM$, such that $\rho$ is a sub-representation of $\res_{\rM'}^{\rM}\pi$. Let $(\rL,\tau)$ be a supercuspidal pair of $\rM$, and $[\rL,\tau]$ consists the supercuspidal support of $\pi$.  By \ref{prop 6}, we have $\res_{\rL'}^{\rL}\tau\cong\oplus_{i\in I}\tau_i$, where $I$ is a finite index set. According to \ref{Lproposition 10}, the supercuspidal support of $(\rM',\pi')$ is contained in the union of $\rM'$-conjugacy class of $(\rL',\tau_i)$, for every $i\in I$. To finish the proof of our theorem, it remains to prove that there exists one unique $i_0\in I$ such that $(\rL',\tau_{i_0})$ is contained in the supercuspidal support of $(\rM',\rho)$.

After conjugation by $\rG'$, we could assume that $\rM'=\rM_{n_1,\ldots,n_m}'$ and $\rL'=\rM'_{k_1,\ldots,k_l}$ for a familly of integers $m,l, n_1,\ldots,n_m,k_1,\ldots,k_l\in\mathbb{N}^{\ast}$. There exists a non-degenerate character $\theta$ of $\rU=\rU_{n_1,\ldots,n_m}$, such that $\rho_{\theta,m}^{(n_1+\cdots+n_m)}\neq 0$. In fact, let $\theta$ be any non-degenerate character of $\rU$ and we write $\res_{\rM'}^{\rM}\pi\cong\oplus_{s\in S}\pi_s$, where $S$ is a finite index set. We have:
$$\pi^{(n_1+...+n_m)}\cong(\pi\vert_{\rM'})_{\theta,m}^{(n_1+...+n_m)}\cong\oplus_{s\in S}(\pi_s)_{\theta,m}^{(n_1+...+n_m)},$$
where $\pi^{(n_1+...+n_m)}$ indicates the $(n_1+...+n_m)$-th derivative of $\pi$. As in Section [\S III, 5.10, 3)] of \cite{V1}, $\mathrm{dim}(\pi^{(n_1+...+n_m)})=1$, hence there exists one element $s_0\in S$ such that $(\pi_{s_0})_{\theta,m}^{(n_1+...+n_m)}\neq 0$. Notice that $\tau$ are isomorphic to some $\pi_s$, hence there exists a diagonal element $t\in\rM$, such that the $t$-conjugation $t(\pi_{s_0})\cong\tau$. The character $t(\theta)$ is also non-degenerate of $\rU$, and we have $(t(\pi_{s_0}))_{t(\theta),m}^{(n_1+...+n_m)}\cong (\pi_{s_0})_{\theta,m}^{(n_1+...+n_m)}$ as $k$-vector spaces. We conclude that $\mathrm{dim}\tau_{t(\theta),m}^{(n_1+...+n_m)}=1$. To simplify the notations, we assume $t=1$.

If $\rho$ is a subquotient of $i_{\rL'}^{\rM'}\tau_i$ for some $i\in I$. By (4) of \ref{dcor 6} and \ref{Uprop 2}, the derivative $\tau_{i_{\theta,l}}^{(k_1+\cdots+k_l)}\neq 0$. Note $\rU\cap\rL'$ as $\rU_{\rL'}$. By section [\S III, 5.10, 3)] of \cite{V1}, the derivative $\tau_{\theta,l}^{(k_1+\cdots+k_l)}=1$, which means the dimension of $(\rU_{\rL'},\theta)$-coinvariants of $\tau$ is $1$ (by \ref{Uprop 1}). Notice that the $(\rU_{\rL'},\theta)$-coinvariants of $\tau$ is the direct sum of $(\rU_{\rL'},\theta)$-coinvariants of $\tau_i$ for every $i\in I$. This implies that there exists one unique $i_0\in I$ whose $(\rU_{\rL'},\theta)$-coinvariants is non-zero with dimension $1$. By \ref{Uprop 1} and \ref{Uprop 2}, this is equivalent to say that there exists one unique $i_0\in I$, such that the derivative $\tau_{i_{\theta,l}}^{(k_1+\cdots+k_l)}\neq 0$.
\end{proof}

\appendix
\section{Theorem 5.2 of Bernstein and Zelevinsky}
We need the results of Theorem 5.2 of \cite{BeZe} in the case of $k$-representations. In fact, the proof in \cite{BeZe} is in the language of $\ell$-sheaves, which can be translated to a representation theoretical proof and be applied to our case. In the reason for being self-contained, I rewrite the proof following the method in \cite{BeZe}.

Let $\rG$ be a locally compact totally disconnected group, $\rP$, $\rM$, $\rU$, $\rQ$, $\rN$, $\rV$ are closed subgroups of $\rG$, and $\theta$, $\psi$ be $k$-characters of $\rU$ and $\rV$ respectively. Suppose that they verify conditions $(1)-(4)$ in \S $5.1$ of \cite{BeZe}, and denote $\rX=\rP\backslash\rG$. The numbering we choose in condition $(3)$ is $\rZ_1,...,\rZ_k$ of $\rQ$-orbits on $X$, and for any orbit $\rZ\subset X$, we choose $\overline{\omega}\in\rG$ and $\omega$ as in condition $(3)$ of \cite{BeZe}.

We introduce condition $(\ast)$:

$(\ast)$ The characters $\omega(\theta)$ and $\psi$ coincide when restricted to the subgroups $\omega(\rU)\cap\rV$.

We define $\Phi_{\rZ}$ equals $0$ if $(\ast)$ does not hold, and define $\Phi_{\rZ}$ as in \S $5.1$ \cite{BeZe} if $(\ast)$ holds.

\begin{defn}
Let $\rM,\rU$ be closed subgroups of $\rG$, and $\rM\cap\rU=\{e\}$, and the subgroup $\rP=\rM\rU$ is closed in $\rG$. Let $\theta$ be a $k$-character of $\rU$ normalized by $\rM$. 
\begin{itemize}
\item Define functor $i_{\rU,\theta}\mathrm{Rep}_{k}(\rM)\rightarrow\mathrm{Rep}_k(\rG)$. Let $\rho\in\mathrm{Rep}_k(\rM)$, then $i_{\rU,\theta}(\rho)$ equals $\ind_{\rP}^{\rG}\rho_{\rU,\theta}$, where $\rho_{\rU,\theta}\in\mathrm{Rep}_k(\rP)$, such that
$$\rho_{\rU,\theta}(mu)=\theta(u)\mathrm{mod}_{\rU}^{\frac{1}{2}}(m)\rho(m)$$.
\item Define functor $r_{\rU,\theta}\mathrm{Rep}_k(\rG)\rightarrow\mathrm{Rep}_k(\rM)$. Let $\pi\in\mathrm{Rep}_k(\rG)$, then $r_{\rU,\theta}(\pi)$ equals $\mathrm{mod}_{\rU}^{-\frac{1}{2}}(\res_{\rP}^{\rG}\pi)\slash(\res_{\rP}^{\rG}\pi)(\rU,\theta)$, where $(\res_{\rP}^{\rG}\pi)(\rU,\theta)\subset\res_{\rP}^{\rG}\pi$, generated by elements $\pi(u)w-\theta(u)w$, for any $w\in W$, where $W$ is the space of $k$-representation $\pi$.
\end{itemize}
\end{defn}

\begin{rem}
\label{rem 5.2.3}
By replacing $\ind$ to $\mathrm{Ind}$, we could define $I_{\rU,\theta}$. Notice that $r_{\rU,\theta}$ is left adjoint to $I_{\rU,\theta}$.
\end{rem}

\begin{prop}
\label{prop 5.2.2}
The functors $i_{\rU,\theta}$ and $r_{\rU,\theta}$ commute with inductive limits.
\end{prop}

\begin{proof}
The functor $r_{\rU,\theta}$ commutes with inductive limits since it has a right adjoint as in \ref{rem 5.2.3}.

For $i_{\rU,\theta}$. Let $(\pi_{\alpha},\alpha\in\mathcal{C})$ be a inductive system, where $\mathcal{C}$ is a directed pre-ordered set. We want to prove that $i_{\rU,\theta}(\underrightarrow{\mathrm{lim}}\pi_{\alpha})\cong\underrightarrow{\mathrm{lim}}(i_{\rU,\theta}\pi_{\alpha}).$ The inductive limit $\underrightarrow{\mathrm{lim}}\pi_{\alpha}$ is defined as $\oplus_{\alpha\in\mathcal{C}}\pi_{\alpha}\slash\sim$, where $\sim$ denotes an equivalent relation: When $\alpha\prec\beta,x\in W_{\alpha},y\in W_{\beta}$, $x\sim y$ if $\phi_{\alpha,\beta}(x)=y$, where $W_{\alpha}$ denotes the space of $k$-representation $\pi_{\alpha}$, and $\phi_{\alpha,\beta}$ denotes the morphism from $\pi_{\alpha}$ to $\pi_{\beta}$ defined in the inductive system.

First, we prove that $i_{\rU,\theta}$ commutes with direct sum. By definition, $\oplus_{\alpha\in\mathcal{C}}i_{\rU,\theta}\pi_{\alpha}$ is a subrepresentation of $i_{\rU,\theta}\oplus_{\alpha\in\mathcal{C}}\pi_{\alpha}$, and the natural embedding is a morphism of $k$-representations of $\rG$. We will prove that the natural embedding is actually surjective. For any $f\in \pi:=i_{\rU,\theta}\oplus_{\alpha\in\mathcal{C}}\pi_{\alpha}$, there exists an open compact subgroup $K$ of $\rG$ such that $f$ is constant on each right $K$ coset of $\rM\rU\backslash\rG$. Furthermore, the function $f$ is non-trivial on finitely many right $K$ cosets. Hence there exists a finite index subset $J\subset\mathcal{C}$, such that $f(g)\in\oplus_{j\in J} W_{j}$, which means $f\in i_{\rU,\theta}\oplus_{j\in J}\pi_{j}$. Since $i_{\rU,\theta}$ commutes with finite direct sum, we finish this case.

The functor $i_{\rU,\theta}$ is exact, we have:
$$i_{\rU,\theta}(\underrightarrow{\mathrm{lim}}\pi_{\alpha})\cong i_{\rU,\theta}(\oplus_{\alpha\in\mathcal{C}}\pi_{\alpha})\slash i_{\rU,\theta}\langle x-y\rangle_{x\sim y}.$$
Notice that $\underrightarrow{\mathrm{lim}} i_{\rU,\theta}\pi_{\alpha}\cong\oplus i_{\rU,\theta}\pi_{\alpha}\slash\thicksim$, where $\thicksim$ denotes the equivalent relation: When $\alpha\prec\beta,f_{\alpha}\in V_{\alpha},f_{\beta}\in V_{\beta}$, where $V_{\alpha}$ is the space of $k$-representation $i_{\rU,\theta}\pi_{\alpha}$, then $f_{\alpha}\thicksim f_{\beta}$ if $i_{\rU,\theta}(\phi_{\alpha,\beta})(f_{\alpha})=f_{\beta}$, which is equivalent to say that $\phi_{\alpha,\beta}(f_{\alpha}(g))=f_{\beta}(g)$ for any $g\in\rG$. In left to prove that the natural isomorphism from $\oplus_{\alpha\in\mathcal{C}}(i_{\rU,\theta}\pi_{\alpha})$ to $i_{\rU,\theta}(\oplus_{\alpha\in\mathcal{C}}\pi_{\alpha})$, induces an isomorphism from $\langle f_{\alpha}-f_{\beta}\rangle_{f_{\alpha}\thicksim f_{\beta}}$ to $i_{\rU,\theta}\langle x-y\rangle_{x\sim y}$. This can be checked directly through definition as in the case of direct sum above.
\end{proof}

\begin{thm}[Bernstein, Zelevinsky]
\label{thm 5.2}
Under the conditions above, the functor $\rF=r_{\rV,\psi}\circ i_{\rU,\theta}:\mathrm{Rep}_{k}(\rM)\rightarrow\mathrm{Rep}_{k}(\rN)$ is glued from the functor $\rZ$ runs through all $\rQ$-orbits on $\rX$. More precisely, if orbits are numerated so that all sets $\rY_i=\rZ_1\cup...\cup\rZ_i$ $(i=1,...,k)$ are open in $\rX$, then there exists a filtration $0=\rF_0\subset\rF_1\subset...\subset\rF_k=\rF$ such that $\rF_i\slash\rF_{i-1}\cong\Phi_{\rZ_i}$.
\end{thm}

The quotient space $\rX=\rP\backslash\rG$ is locally compact totally disconnected. Let $\rY$ be a $\rQ$-invariant open subset of $\rX$. We define the subfunctor $\rF_{\rY}\subset\rF$. Let $\rho$ be a $k$-representation of $\rM$, and $W$ be its representation space. We denote $i(W)$ the representation $k$-space of $i_{\rU,\theta}(\rho)$. Let $i_{\rY}(W)\subset i(W)$ the subspace consisting of functions which are equal to $0$ outside the set $\rP\rY$, and $\tau$, $\tau_{\rY}$ be the $k$-representations of $\rQ$ on $i(W)$ and $i_{\rY}(W)$. Put $\rF_{\rY}(\rho)=r_{\rV,\psi}(\tau_{\rY})$, which is a $k$-representation of $\rN$. The functor $\rF_{\rY}$ is a subfunctor of $\rF$ due to the exactness of $r_{\rV,\psi}$.

\begin{prop}
\label{prop 5.2.1}
Let $\rY$, $\rY'$ be two $\rQ$-invariant open subset in $\rX$, we have:
$$\rF_{\rY\cap\rY'}=\rF_{\rY}\cap\rF_{\rY'},\quad \rF_{\rY\cup\rY'}=\rF_{\rY}+\rF_{\rY'},\quad \rF_{\emptyset}=0,\quad \rF_{\rX}=\rF.$$
\end{prop}

\begin{proof}
Since $r_{\rV,\psi}$ is exact, it is sufficient to prove similar formulae for $\tau_{\rY}$. The only non-trivial one is the equality $\tau_{\rY\cup\rY'}=\rF_{\rY}+\rF_{\rY'}$. As in \S $1.3$ \cite{BeZe}, let $K$ be a compact open subgroup of $\rY\cup\rY'$, there exists $\varphi$ and $\varphi'$, which are idempotent $k$-function on $\rY$ and $\rY'$, such that $(\varphi+\varphi')\vert_{K}=1$. We deduce the result from this fact.
\end{proof}

Let $\rZ$ be any $\rQ$-invariant locally closed set in $\rX$, we define the functor 
$$\Phi_{\rZ}:\mathrm{Rep}_{k}(\rM)\rightarrow\mathrm{Rep}_{k}(\rN)$$
to be the functor $\rF_{\rY\cup\rZ}\slash\rF_{\rY}$, where $\rY$ can be any $\rQ$-invariant open set in $\rX$ such that $\rY\cup\rZ$ is open and $\rY\cap\rZ=\emptyset$. Let $\rZ_1,...,\rZ_k$ be the numeration of $\rQ$-orbits on $\rX$ as in \ref{thm 5.2}, and let $\rF_i=\rF_{\rY_i}$ $(i=1,...,k)$, which is a filtration of the functor $\rF$ be the definition. To prove Theorem \ref{thm 5.2}, it is sufficient to prove that $\rF_{\rZ_i}\cong\Phi_{\rZ_i}$. 

By replace $\rP$ to $\omega(\rP)$, we could assume that $\omega=\mathds{1}$. Now we consider the diagram in figure $\mathrm{BZ}$.
\begin{figure}
\centering
\begin{tikzpicture}
\node at (0,0){$\rP\cap\rG$};
\node at (-6,0) {$\rM$};
\node at (6,0){$\rN$};
\node at (-3,1) {$\rP$};
\node at (3,1) {$\rQ$};
\node at (0,2) {$\rG$};
\node at (-2,-1) {$\rM\cap\rQ$};
\node at (-4,-1) {$\rM\cap\rQ$};
\node at (2,-1) {$\rN\cap\rP$};
\node at (4,-1) {$\rN\cap\rP$};
\node at (0,-2) {$\rM\cap\rN$};
\node at (-3,0) {$\mathrm{I}$};
\node at (0,-1) {$\mathrm{II}$};
\node at (0,1) {$\mathrm{III}$};
\node at (3,0) {$\mathrm{IV}$};

\draw[->](-1.6,-0.8)--(-0.45,-0.225);
\draw[->](0.45,-0.225)--(1.6,-0.8);
\draw[->](-1.6,-1.2)--(-0.45,-1.775);
\draw[->](0.45,-1.775)--(1.6,-1.2);
\node at (-1.65,-0.4) {$\rU\cap\rQ$};
\node at (1.65,-0.4) {$\rV\cap\rP$};
\node at (-1.65,-1.6) {$\rV\cap\rM$};
\node at (1.65,-1.6) {$\rU\cap\rN$};

\draw[->](-5.75,0.089)--(-3.25,0.95);
\draw[->](-2.75,1.09)--(-0.25,1.92);
\draw[dashed,->](0.25,1.92)--(2.75,1.09);
\draw[->](3.25,0.91)--(5.75,0.08);
\draw[->](-2.75,0.91)--(-0.5,0.16);
\draw[->](0.5,0.16)--(2.75,0.91);
\node at (-4.75,0.7) {$\rU$};
\node at (-1.75,1.7) {$e$};
\node at (-1.5,0.7) {$e$};
\node at (1.5,0.7) {$e$};
\node at (4.75,0.7) {$\rV$};

\draw[->](-5.75,-0.2)--(-4.6,-0.85);
\draw[->](5.75,-0.2)--(4.6,-0.85);
\node at (-5.3,-0.6) {$e$};
\node at (5.3,-0.6) {$e$};

\draw[-stealth,decorate,decoration={snake,amplitude=2pt,pre length=1pt,post length=1pt}] (-3.4,-1) -- (-2.6,-1);
\draw[-stealth,decorate,decoration={snake,amplitude=2pt,pre length=1pt,post length=1pt}] (2.6,-1) -- (3.4,-1);
\node at (-3,-0.75) {$\varepsilon_1$};
\node at (3,-0.75) {$\varepsilon_2$};

\end{tikzpicture}
\caption{BZ}
\end{figure}

This is the same diagram as in \S $5.7$ \cite{BeZe}, in which a group a point $\rH$ means $\mathrm{Rep}_k(\rH)$, an arrow $\stackrel{\rH}{\nearrow}$ means the functor $i_{\rH,\theta}$, an arrow $\underset{\rH}{\searrow}$ means the functor $i_{\rH,\psi}$, and an arrow $\overset{\varepsilon}{\leadsto}$ means the functor $\varepsilon$ (consult \S $5.1$ \cite{BeZe} for the definition of $\varepsilon$). Notice that $\rG\dashrightarrow\rQ$ does not mean any functor, but the functor $\rP\rightarrow\rG\dashrightarrow\rQ$ is well-defined as explained above \ref{prop 5.2.1}. The composition functors along the highest path is of this diagram is $\rF_{\rZ}$, and if the condition $(\ast)$ holds, the composition functors along the lowest path is $\Phi_{\rZ}$. We prove Theorem \ref{thm 5.2} by showing that this diagram is commutative if condition $(\ast)$ holds, and $\rF_{\rZ}$ equals $0$ otherwise. Notice that parts $\mathrm{I},\mathrm{II},\mathrm{III},\mathrm{IV}$ are four cases of \ref{thm 5.2}, and we prove the statements through verifying them under the four cases respectively. 

Let $\rho$ be any $k$-representation of $\rM$, and $W$ is its representation space. We use $\pi$ to denote $\rF_{\rZ}(\rho)$, and $\tau$ to denote $\Phi_{\rZ}(\rho)$.

Case $\mathrm{I}$: $\rP=\rG,\rV=\{ e \}$. The $k$-representations $\pi$ and $\tau$ act on the same space $W$, and the quotient group $\rM\backslash(\rP\cap\rQ)$ is isomorphic to $(\rM\cap\rQ)\backslash(\rP\cap\rQ)$. We verify directly by definition that $\pi\cong\tau$.

Case $\mathrm{II}$: $\rP=\rG=\rQ$. The representation space of $\pi$ is still $W$. We have the equation:
$$r_{\rV,\psi}(W)\cong r_{\rV\cap\rM,\psi}(r_{\rV\cap\rU,\psi}(W)).$$
If $\theta\vert_{\rU\cap\rV}\neq\psi\vert_{\rU\cap\rV}$, then $\pi=0$ since $\rU\cap\rV=\rU\cap\rQ\cap\rV\cap\rP$ and $r_{\rV\cap\rU,\psi}(W)=0$. This means that after proving the diagrams of cases $\mathrm{I},\mathrm{III},\mathrm{IV}$ are commutative, the functor $\rF_{\rZ}$ equals $0$ if condition $(\ast)$ does not hold.

Now we assume that $(\ast)$ holds. The $k$-representations $\pi$ and $\tau$ act on the same space $W\slash W(\rV\cap\rM,\psi)$, because the fact $r_{\rV\cap\rU,\psi}(i_{\rV\cap\rU,\theta}(W))=W$ and the equation above. Notice that we have equations for $k$-character $\mathrm{mod}$:
$$\mathrm{mod}_{\rU}=\mathrm{mod}_{\rU\cap\rM}\cdot\mathrm{mod}_{\rU\cap\rV},\quad \mathrm{mod}_{\rV}=\mathrm{mod}_{\rV\cap\rN}\cdot\mathrm{mod}_{\rV\cap\rU},$$
from which we deduce that $\pi\cong\tau$ when condition $(\ast)$ holds.

Case $\mathrm{III}$: $\rU=\rV=\{ e \}$. Let $i(W)$ be the representation space of $i_{\rP}^{\rG}\rho$, then $\pi$ acts on a quotient space $W_1$ of $i(W)$. Let:
$$E=\{ f\in i(W)\vert f(\overline{\rP\rQ}\backslash \rP\rQ)=0 \},$$
$$E'=\{ f\in i(W)\vert f(\overline{\rP\rQ})=0 \},$$
then $W_1\cong E\slash E'$. The $k$-representation $\tau$ acts on $i(W)'$, which is the representation space of $i_{\rP\cap\rQ}^{\rQ}\rho$. By definition,
$$i(W)'=\{ h:\rQ\rightarrow W \vert h(pq)=\rho(p)h(q),p\in\rP\cap\rQ,q\in\rQ   \}.$$
We define a morphism $\gamma$ from $W_1$ to $i(W)'$, by sending $f$ to $f\vert_{\rQ}$, which respects $\rQ$-actions and is actually a bijection. For injectivity, let $f_1,f_2\in W_1$ and $f_1\vert_{\rQ}=f_2\vert_{\rQ}$, then $f_1-f_2$ is trivial on $\rP\rQ$, hence $f_1-f_2$ is trivial on $\overline{\rP\rQ}$ by the definition of $E$. This means $f_1-f_2=0$ in $W_1$. Now we prove $\gamma$ is surjective. Let $h\in i(W)'$, there exists an open compact subgroup $K'$ of $(\rP\cap\rQ)\backslash\rQ$ such that $h$ is constant on the right $K'$ cosets of $(\rP\cap\rQ)\backslash\rQ$, and denote $S$ the compact support of $h$. Let $K$ be an open compact subgroup of $\rP\backslash\rG$ such that $(\rP\cap\rQ)\backslash(\rQ\cap K)\subset K'$, and $S\cdot K\cap (\overline{\rP\rQ}\slash \rP\rQ)=\emptyset$. We define $f$ such that $f$ is constant on the right $K$ cosets of $\rP\backslash\rG$, and $f\vert_{(\rP\cap\rQ)\backslash\rQ}=h$. The function $f$ is smooth with compact support on the complement of $\overline{\rP\rQ}\slash \rP\rQ$, hence belongs to $E$, and $\gamma(f)=h$ as desired.

Case $\mathrm{IV}$: $\rU=\{ e\},\rQ=\rG$. We divide this case into two cases $\mathrm{IV}_1$ and $\mathrm{IV}_2$ as in the diagram of figure $\mathrm{Case} \mathrm{IV}$.

\begin{figure}
\begin{tikzpicture}
\node at (0,0) {$\rM$};
\node at (2,1) {$(\rM\cap\rN)\rV$};
\node at (2,-1) {$\rM\cap\rN$};
\node at (5,-1) {$\rM\cap\rN$};
\node at (8,0.5) {$\rN$};
\node at (5,2) {$\rQ=\rN\rV$};

\node at (2,0) {$\mathrm{IV}_2$};
\node at (5,0.5) {$\mathrm{IV}_1$};

\draw[->](0.17,0.15)--(1.25,0.75);
\draw[->](0.17,-0.15)--(1.45,-0.8);

\node at (0.5,0.5) {$e$};
\node at (0.4,-0.7) {$\rV\cap\rM$};

\draw[->](2.75,1.25)--(4.35,1.85);
\draw[->](5.65,1.85)--(7.85,0.66);
\draw[->](2.75,0.85)--(4.47,-0.84);
\draw[->](5.53,-0.84)--(7.85,0.34);
\node at (3.5,1.7) {$e$};
\node at (3.75,0.2) {$\rV$};
\node at (6.7,1.53) {$\rV$};
\node at (6.5,-0.1) {$e$};

\draw[-stealth,decorate,decoration={snake,amplitude=2pt,pre length=1pt,post length=1pt}] (2.6,-1) -- (4.4,-1);
\node at (3.5,-0.75) {$\varepsilon_2$};

\end{tikzpicture}
\caption{Case $\mathrm{IV}$}
\end{figure}

Case $\mathrm{IV}_1$: $\rU=\{ e\},\rQ=\rG,\rV\subset\rM=\rP$. The $k$-representation $\pi$ acts on $i(W)^{+}=i(W)\slash i(W)(\rV,\psi)$, where
$$i(W)_{\rV,\psi}=\langle vf-\psi(v)f, \forall f\in i(W),v\in\rV \rangle.$$
The $k$-representations $\tau$ acts on $i(W^{+})$, which is the smooth functions with compact support on $(\rM\cap\rN)\backslash\rN$ defined as below:
$$\{ h:\rN\rightarrow W\slash W(\rV,\psi)\vert f(mn)=\rho(m)f(n), \forall m\in\rM\cap\rN,n\in\rN \}.$$
There is a surjective projection from $i(W)$ to $i(W^{+})$, which projects $f(n)$ in $W^+=W\slash W_{\rV,\psi}$, for any $f\in i(W)$. In fact, let $h\in i(W^+)$, there exists an open compact subgroup $K$ of $\rP\backslash\rG\cong(\rM\cap\rN)\backslash\rN$, such that $f=\sum_{i=1}^m h_i$, $m\in\mathbb{N}$, where $h_i\in i(W^+)$ is nontrivial on one right $K$ coset $a_i K$ of $\rP\backslash\rG$. We have $h_i\equiv \overline{w_i}$ on $a_i K$, where $w_i\in W$ and $\overline{w_i}\in W^+$. Define $f=\sum_{i=1}^m f_i$, where $f_i\equiv w_i$ on $a_i K$, and equals $0$ otherwise. The function $f\in i(W)$, and the projection image is $h$.

It is clear that this projection induces a morphism from $i(W)^+$ to $i(W^+)$, and we prove this morphism is injective. Let $f,f'\in i(W)^+$, and $f=f'$ in $i(W^+)$. As in the proof above, there exists an open compact subgroup $K_0$ of $\rP\backslash\rG$, and $f_j\in i(W)^+$ such that $f_j$ is non-trivial on one right $K_0$ coset of $\rP\backslash\rG$ and $f-f'=\sum_{j=1}^sf_j$. Furthermore, the supports of $f_j$'s are two-two disjoint. Hence the image of $f_j$ on its support is contained in $W^+$, since $f_j$ is constant on its support, it equals $0$ in $i(W)^+$, whence $f-f'$ equals $0$ in $i(W)^+$. We conclude that this morphism is bijection, and the diagram case $\mathrm{IV}_1$ is commutative. 

Case $\mathrm{IV}_2$: $\rU=\{ e\},\rG=\rQ,\rN\subset\rM$. In this case:
$$\rX=\rN\rV'\backslash\rN\rV\cong\rV'\rV,$$
where $\rV'=\rV\cap\rM$. We choose one Haar measures $\mu$ of $\rX$ (the existence see \S$\mathrm{I}$, 2.8, \cite{V1}). Let $W^+$ denote the quotient $W\slash W(\rV',\psi)$ and $p$ the canonical projection $p: W\rightarrow W^+$. Let $i(W)$ be the space of $k-$representation $\tau= i_{\{e\},1}(\rho)$.

Define $\overline{A}$ a morphism of $k$-vector spaces from $i(W)$ to $W^+$ by:
$$\overline{A}f=\int_{\rV'\backslash\rV}\psi^{-1}(v)p(f(v))\mathrm{d}\mu(v).$$
This is well defined since the function $\psi^{-1}f$ is locally constant with compact support of $\rV'\slash\rV$, and the integral is in fact a finite sum. Since $\mu$ is stable by right translation, we have for any $v\in\rV$:
$$\overline{A}(\tau(v,f))=\psi(v)\overline(A)(f).$$
Hence $\overline{A}$ induces a morphism of $k$-vector spaces:
$$A:i(W)\slash i(W)(V,\psi)\rightarrow W^+.$$
Now we justify that $A\in\mathrm{Hom}_{k[\rN]}(\pi,\tau)$, where $k$-representations $\pi=r_{\rV,\psi}(\tau)$ equals $\rF(\rho)$, and $\tau=\varepsilon_2\cdot r_{\rV',\psi}(\rho)$ equals $\Phi(\rho)$. For any $n\in \rN$:
\begin{eqnarray}
A(\pi(n)f) &=&\mathrm{mod}_{\rV}^{-\frac{1}{2}}(n)\int_{\rV'\slash\rV}\psi^{-1}(v)p(f(vn))\mathrm{d}\mu(v)\\
&=&  \mathrm{mod}_{\rV}^{-\frac{1}{2}}(n)\sigma(n)\mathrm{mod}_{\rV'}^{\frac{1}{2}}(n)\varepsilon_2^{-1}\cdot\int_{\rV'\slash\rV}\psi^{-1}(v)p(f(n^{-1}vn))\mathrm{d}\mu(v)
\end{eqnarray}
By replacing $v'=n^{-1}vn$, the equation above equals to:
$$\sigma(n)\int_{\rV'\slash\rV}\psi^{-1}(v')p(f(v'))\mathrm{d}\mu(v')=\sigma(n)A(f).$$
Therefore $A$ belongs to $\Hom_{k[N]}(\pi,\tau)$, and hence a morphism from functor $\rF$ to $\Phi$. Now we prove that $A$ is an isomorphism.

Let $\rho'$ be the trivial representations of $\{e\}$ on $W$, then $i(W)'$ the space of $k$-representation $\ind_{\rV'}^{\rV}\rho'$ is isomorphic to $i(W)$ the space of $k$-representation $i_{\{e\},1}\rho$. And the diagram \ref{fig 6} $\mathrm{IV}_2$ is commutative, where $A$ indicates the morphism of $k$-vector spaces associated to the functor $A$. Hence it is sufficient to suppose that $\rN=\{e\},\rM=\rV'$. Replacing $\rho$ by $\psi^{-1}\rho$, we can suppose that $\psi=1$.

\begin{figure}
\label{fig 6}
\centering
\begin{tikzpicture}
\node at (-2,0) {$i(W)'\slash i(W)'(\rV,\psi)$};
\node at (1,0) {$W^+$};
\node at (-2,-1.5) {$i(W)\slash i(W)(\rV,\psi)$};
\node at (1,-1.5) {$W^+$};

\draw[->](-0.4,0)--(0.5,0);
\draw[->](-2,-0.2)--(-2,-1.2);
\draw[->](1,-0.2)--(1,-1.2);
\draw[->](-0.4,-1.5)--(0.5,-1.5);

\node at (-2.3,-0.7) {$\cong$};
\node at (1.3,-0.7) {$\cong$};
\node at (0.05,0.2) {$A$};
\node at (0.05,-1.3) {$A$};

\end{tikzpicture}
\caption{$\mathrm{IV}_2$}
\end{figure}

First of all, we consider $\rho=i_{\{e\},1}\mathds{1}=\ind_{e}^{\rV'}\mathds{1}$ the regular $k$-representation on $S(\rV')$, which is the space of locally constant functions with compact support on $\rV'$. Then $\tau=i_{\{e\},1}\rho$ is the regular $k$-representation of $\rV$ on $S(\rV)$ by the transitivity of induction functor. Any $k$-linear form on $r_{\rV',1}(S(\rV'))$ gives a Haar measure on $\rV'$, and conversely any Haar measure gives a $k$-linear form on $S(\rV')$, whose kernel is $S(\rV')(\rV',1)$, hence the two spaces is isomorphic, and the uniqueness of Haar measures implies that the dimension of $r_{\rV',1}(S(\rV'))$ equals one. We obtain the same result to $r_{\rV,1}(S(\rV'))$. Since in this case the morphism $A$ is non-trivial, then it is an isomorphism. The functors $i_{\{e\},1},r_{\rV,\psi},r_{\rV',\psi}$ commute with direct sum (as in \ref{prop 5.2.2}), and the morphism $A$ between $k$-vector spaces also commutes with direct sum, hence $A:\pi\rightarrow\tau$ is an isomorphism when $\rho$ is free, which means $\rho$ is a direct sum of regular $k$-representations of $\rV'$. Notice that any $\rho$ can be viewed as a module over Heck algebra, then $\rho$ is a quotient of some free $k$-representation. Hence $\rho$ has a free resolution. The exactness of $\rF$ and $\Phi$ implies that $A:\rF(\rho)\rightarrow\Phi(\rho)$ is an isomorphism for any $\rho$.


\pagestyle{empty}

\includepdf[pages={last-}, pagecommand={}, offset=0cm -1cm]{couverture.pdf}


\begin{thebibliography}{99}
\bibitem[AMS]{AMS} A.-M. Aubert; A. Moussaoui; M. Solleveld, {\em Generalizations of the Springer correspondence and cuspidal Langlands parameters}. Manuscripta Math. 157 (2018), no. 1-2, 121-192.

\bibitem[BeZe]{BeZe} I.N. Bernstein; A.V. Zelevinsky, {\em Induced representations of reductive $p$-adic groups. I.} Ann. Sci. \'Ecole Norm. Sup. (4) 10 (1977).

\bibitem[Bon]{Bonn} C. Bonnaf\'e, {\em Sur les caract\`eres des groupes r\'eductifs finis \`a centre non connexe: applications aux groupes sp\'eciaux lin\'eaires et unitaires}. Ast\'erisque No. 306, 2006.

\bibitem[BrMi]{BrMi} M. Brou\'e; J. Michel, {\em Blocs et s\'eries de Lusztig dans un groupe r\'eductif fini}. J. Reine Angew. Math. 395 (1989), 56-67. 

\bibitem[BuHe]{BuHe} C.J. Bushnell; G. Henniart, {\em Modular local Langlands correspondence for $\GL_n$}. Int. Math. Res. Not. 15 (2014), 4124-4145.

\bibitem[BuKu]{BuKu}  C.J. Bushnell, P.C. Kutzko,  {\em The admissible dual of $\GL(N)$ via compact open subgroups.} Annals of Mathematics Studies, 129. Princeton University Press, Princeton, NJ, 1993.

\bibitem[BuKuI]{BuKuI}  C.J. Bushnell; P.C.  Kutzko,  {\em The admissible dual of $\SL(N)$. I.} Ann. Sci. \'Ecole Norm. Sup. (4) 26 (1993), no. 2, 261-280. 

\bibitem[BuKuII]{BuKuII}  C.J. Bushnell; P.C. Kutzko, {\em The admissible dual of $\SL(N)$. II.} Proc. London Math. Soc. (3) 68 (1994), no. 2, 317-379.
 
\bibitem[BSS]{BSS} P. Broussous; V. S\'echerre; S. Stevens,  {\em Smooth representations of $\GL_m(D)$ V: Endo-classes.} Doc. Math. 17 (2012), 23-77.

\bibitem[Da]{Da} J.-F. Dat, {\em Simple subquotients of big parabolically induced representations of $p$-adic groups}, J. Algebra \textbf{510} (2018), 499-507, with an 
Appendix: {\em Non-uniqueness of supercuspidal support for finite reductive groups} by O. Dudas

\bibitem[DaII]{DaII} J.-F. Dat. {\em A functoriality principle for blocks of linear $p$-adic groups.} Contemporary
Mathematics, 691:103-131, 2017.

\bibitem[DeLu]{DeLu} P. Deligne; G. Lusztig, {\em Representations of reductive groups over finite fields.} Ann. of Math. (2) 103 (1976), no. 1, 103-161. 

\bibitem[DiFl]{DiFl} R. Dipper; P. Fleischmann, {\em Modular Harish-Chandra theory. II} Arch. Math. (Basel) 62, 1994.

\bibitem[DiMi]{DiMi} F. Digne; J. Michel,  {\em Representations of finite groups of Lie type.} London Mathematical Society Student Texts, 21. Cambridge University Press, Cambridge, 1991. 

\bibitem[Geck]{Geck} M. Geck, {\em Modular Harish-Chandra series, Hecke algebras and (generalized) $q$-Schur algebras.} Modular representation theory of finite groups (Charlottesville, VA, 1998), 1?66, de Gruyter, Berlin, 2001.

\bibitem[GoRo]{GoRo} D. Goldberg; A. Roche,  {\em Types in $\mathrm{SL}_n$}. Proc. London Math. Soc. (3) 85 (2002), no. 1, 119-138. 

\bibitem[GrHi]{GrHi} J. Gruber; G. Hiss, {\em Decomposition numbers of finite classical groups for linear primes.} J. Reine Angew. Math. 485 (1997), 55-91. 

\bibitem[HT]{Ha} M. Harris and R. Taylor. {\em The geometry and cohomology of some simple Shimura varieties.}
Number 151 in Ann. of Math. studies. Princeton Univ. Press, 2001.

\bibitem[Helm]{Helm} D. Helm, {\em The Bernstein center of the category of smooth $W(k)[\GL_n(F)]$-modules.} Forum Math. Sigma 4, 2016.

\bibitem[Hen]{Hen} G. Henniart. {\em Une preuve simple des conjectures de Langlands pour $\mathrm{GL}(n)$ sur un corps $p$-adique.} Invent. Math., 139:439-455, 2000.

\bibitem[Hiss]{Hiss} G. Hiss, {\em Supercuspidal representations of finite reductive groups.} J. Algebra 184, 1996.

\bibitem[KuMa]{KuMa} R. Kurinczuk; N. Matringe, {\em The $\ell$-modular local Langlands correspondence and local factors}, arXiv:1805.05888, 2018.

\bibitem[LRS]{LRS} G. Laumon; M. Rapoport; U. Stuhler,   {\em $\cD$-elliptic sheaves and the Langlands correspondence}. Invent. math. 113 (1993) 217-238.

\bibitem[MS]{MS} A. M$\acute{\i}$nguez; V. S\'echerre,  {\em Types modulo $\ell$ pour les formes int\'erieures de $\GL_n$ sur un corps local non archim\'edien.}  With an appendix by Vincent S\'echerre et Shaun Stevens. Proc. Lond. Math. Soc. (3) 109 (2014), no. 4, 823-891.

\bibitem[Re]{Re} D.Renard, {\em Repr\'esentations des groupes r\'eductifs $p$-adiques (French)} [Representations of p-adic reductive groups] Cours Sp\'ecialit\'es [Specialized Courses], 17. Soci\'et\'eMath\'ematique de France, Paris, 2010.

\bibitem[Sch]{Sch} P. Scholze, {\em The local Langlands correspondence for $\GL_n$ over $p$-adic fields}. Invent. Math. 192 (2013), no. 3, 663-715. 

\bibitem[Ser]{Serre} J.-P. Serre, {\em Linear representations of finite groups}. Translated from the second French edition by Leonard L. Scott. Graduate Texts in Mathematics, Vol. 42. Springer-Verlag, New York-Heidelberg, 1977.

\bibitem[Si]{Si} A. Silberger, {\em Isogeny Restrictions of Irreducible Admissible Representations are Finite Direct Sums of Irreducible Admissible Representations}. Proceedings of The American Mathematical Society, 1979.

\bibitem[SeSt]{SeSt} V. S\'echerre; S. Stevens, {\em Block decomposition of the category of $\ell$-modular smooth representations of $\mathrm{GL}(n,F)$ and its inner forms} Ann. Sci. \'Ec. Norm. Sup\'er. (4) 49 (2016), no. 3, 669-709.

\bibitem[Ta]{TA} M. Tadi$\acute{\mathrm{c}}$, {\em Notes on representations of non-Archimedean $\SL(n)$.} Pacific J. Math. 152 (1992), no. 2, 375-396. 

\bibitem[V1]{V1} M.-F. Vign\'eras, {\em Repr\'esentations $\ell$-modulaires d'un groupe r\'eductif $p$-adique avec $\ell\neq p$.} Progress in Mathematics, 137. Birkhauser Boston, Inc., Boston, MA, 1996.

\bibitem[V2]{V2} M.-F. Vign\'eras, {\em Induced $R$-representations of $p$-adic reductive groups.} Selecta Math. (N.S.) 4 (1998), 549-623, with an appendix by A. Arabia.

\bibitem[V3]{V3} M.-F.  Vign\'eras,  {\em Irreducible modular representations of a reductive $p$-adic group and simple modules for Hecke algebras.} European Congress of Mathematics, Vol. I (Barcelona, 2000), 117-133, Progr. Math., 201, Birkh\"auser, Basel, 2001.

\bibitem[V4]{V4} M.-F. Vign\'eras, {\em Correspondance de Langlands semi-simple pour $\mathrm{GL}(n,F)$ modulo $\ell\neq p$.} Invent. Math. 144 (2001), no. 1, 177-223. 

\bibitem[ZeII]{ZeII} A.V. Zelevinsky, {\em Induced representations of reductive $p$-adic groups. II. On irreducible representations of $\mathrm{GL}(n)$.} Ann. Sci. \'Ecole Norm. Sup. (4) 13 (1980), no. 2, 165-210.


\end{thebibliography}
\end{document}